\documentclass{amsart}[12pt]

\usepackage{yfonts} %
\usepackage{amssymb} %
\usepackage{amsthm}
\usepackage{array}
\usepackage{booktabs}%
\usepackage{hhline}%
\usepackage{xy} %
\usepackage{epsfig}%
\usepackage{color}%
\usepackage{upgreek}
\usepackage[english]{babel}
\usepackage{epigraph}%
\usepackage{fancybox}%
\usepackage{shadow}
\usepackage{afterpage}
\usepackage{mathrsfs}
\usepackage{enumitem}
\usepackage{tabularx}
\usepackage{subcaption}
\usepackage{graphicx}
\usepackage{type1cm}
\usepackage{eso-pic}
\usepackage{color}
\usepackage{upgreek}
\usepackage[foot]{amsaddr}

\newtheorem{theorem}{Theorem}
\newtheorem{lemma}{Lemma}
\newtheorem{proposition}{Proposition}
\theoremstyle{definition}
\newtheorem{definition}{Definition}
\newtheorem{remark}{Remark}
\newtheorem{example}{Example}

\theoremstyle{plain}
\newtheorem{corollary}{Corollary}

\newtheorem*{MCP}{Maximum-Continuation Principle}

\newcommand{\vt}{\vspace{.1cm}}

\newcommand{\R}{\mathbb{R} }
\newcommand{\q}{\mathbb{Q} }
\newcommand{\Z}{\mathbb{Z} }
\newcommand{\C}{\mathbb{C} }

\newcommand{\h}{\mathbb{H}}

\newcommand{\hf}{\mathbb{H}_{\mathbb F}^m}
\newcommand{\hfr}{\mathbb{H}_{\mathbb F}^m\times\R}

\newcommand{\s}{\mathbb{S}}

\renewcommand{\rho}{\varrho}
\renewcommand{\theta}{\varTheta}
\renewcommand{\Theta}{\varTheta}
\renewcommand{\Sigma}{\varSigma}
\renewcommand{\tau}{\uptau}
\usepackage{amsmath}% http://ctan.org/pkg/amsmath

\newcommand{\overbar}[1]{\mkern 1.5mu\overline{\mkern-1.5mu#1\mkern-1.5mu}\mkern 1.5mu}

\newcommand{\transv}{\mathrel{\text{\tpitchfork}}}
\makeatletter
\newcommand{\tpitchfork}{%
  \vbox{
    \baselineskip\z@skip
    \lineskip-.52ex
    \lineskiplimit\maxdimen
    \m@th
    \ialign{##\crcr\hidewidth\smash{$-$}\hidewidth\crcr$\pitchfork$\crcr}
  }%
}
\makeatother

\begin{document}

\title[Hypersurfaces of Constant Higher Order Mean Curvature]
{Hypersurfaces of  Constant Higher Order \\ Mean Curvature in  $M\times\R$}
\author{R. F. de Lima, F. Manfio \and J. P. dos Santos}
\address[A1]{Departamento de Matem\'atica - Universidade Federal do Rio Grande do Norte}
\email{ronaldo@ccet.ufrn.br}
\address[A2]{ICMC–Universidade de São Paulo}
\email{manfio@icmc.usp.br}
\address[A3]{Departamento de Matem\'atica - Universidade de Brasília}
\email{joaopsantos@unb.br}
\maketitle

\begin{abstract}
We consider hypersurfaces of products $M\times\R$ with
constant $r$-th mean curvature $H_r\ge 0$ (to be called $H_r$-hypersurfaces), where
$M$ is an arbitrary Riemannian $n$-manifold. We develop a general method
for constructing them,
and employ it to  produce many examples for a variety of manifolds $M,$
including  all simply connected space forms and the hyperbolic spaces $\hf$
(rank $1$ symmetric spaces of noncompact type).
We construct and classify complete rotational $H_r(\ge 0)$-hypersurfaces in
$\hfr$ and in $\s^n\times\R$ as well. They include spheres, Delaunay-type annuli and, in the case of
$\hfr,$  entire graphs.
We also construct and classify
complete $H_r(\ge 0)$-hypersurfaces of $\hfr$ which are invariant
by either parabolic isometries or hyperbolic translations.
We establish a  Jellett-Liebmann-type theorem by showing that
a compact, connected and strictly convex
$H_r$-hypersurface of $\h^n\times\R$ or
$\s^n\times\R$ $(n\ge 3)$  is a rotational embedded sphere.  Other uniqueness
results for complete $H_r$-hypersurfaces of these ambient spaces are obtained.
\end{abstract}

\section{Introduction}

In his pioneering work \cite{rosenberg}, H. Rosenberg initiated the study of minimal and
constant mean curvature hypersurfaces
of product spaces $M\times\R,$ where $M$ is an arbitrary Riemannian $n$-manifold. Since then,
many results on this subject have been obtained by many authors, and considerably understanding
of the geometry of these hypersurfaces has been achieved, mostly in the particular case
$M$ is a simply connected space form.

Following this path, we approach here hypersurfaces of $M\times\R$
with constant
$r$-th mean curvature $H_r\ge 0$ (for some $r\in\{1,\dots, n\}$),
which we call $H_r$-\emph{hypersurfaces}. Let us recall that
the (non normalized) $r$-th mean curvature $H_r$ of a hypersurface is
the $r$-th elementary symmetric polynomial of its principal curvatures, so that
it constitutes a natural extension of the mean curvature ($r=1$)
and the Gauss-Kronecker curvature ($r=n$).

We focus on  constructing and classifying $H_r$-hypersurfaces of products $M\times\R.$
To this end, we use a special type of graph built on
families of parallel hypersurfaces of $M.$
In fact, for any given constant $H_r\ge 0,$
we  obtain $H_r$-graphs in  $M\times\R$ for those Riemannian manifolds $M$ which
admit a local  family of parallel hypersurfaces,
each of them having constant principal curvatures. Following \cite{berndtetal}, such hypersurfaces are called
\emph{isoparametric}.
We point out that many Riemannian  manifolds $M$ admit isoparametric hypersurfaces, such as
space forms, hyperbolic spaces, warped products, and  $\mathbb E(\kappa,\tau)$ homogeneous spaces.

By suitably ``gluing'' pieces of  $H_r$-graphs, we construct
properly embedded $H_r$-hyper\-sur\-faces in $M\times\R$ %including topological $H_r$-spheres,
when $M$ is either the standard $n$-sphere $\s^n$
or one of the hyperbolic spaces $\hf$ (rank $1$ symmetric spaces
of non compact type).
In this setting, we show that there exists a rotational
$H_r$-sphere in $\hf\times\R$  if and only if
$H_r>C_{\mathbb F}(r),$ where the  constant $C_{\mathbb F}(r)$
is defined as the limit of the $r$-th mean curvature of
a geodesic sphere of $\hf$ as its radius goes to infinity.
(In particular, $C_{\mathbb F}(r)$ is positive for $1\le r<n=\dim\hf$ and vanishes for $r=n$.)
On the other hand, as we also show, for any $r\in\{1,\dots ,n\}$ and any constant $H_r>0,$
there exists a rotational $H_r$-sphere in $\s^n\times\R.$

We remark that  rotational hypersurfaces of a general product
$M\times\R$ are defined here as those which are foliated by horizontal
geodesic spheres centered at an ``axis'' $\{o\}\times\R,$ $o\in M.$
(See  Section \ref{sec-rotationalHrhyp} for more details.)

We provide other examples of properly embedded rotational $H_r(>0)$-hyper\-sur\-faces in
$\hfr$ and in $\s^n\times\R$ as well,
including Delaunay-type annuli and, in the case of $\hfr$, entire graphs over $\hf.$
Then, we classify  those complete connected
rotational $H_r(>0)$-hyper\-surfaces of these product spaces
whose height functions are Morse-type (i.e., have isolated critical points), which include
all  the  properly embedded rotational $H_r(>0)$-hypersurfaces we obtain here.

We also construct and classify complete connected $H_r(>0)$-hypersurfaces of
$\hfr$ with no horizontal points (critical points of the height function)
which are invariant by parabolic isometries (i.e, foliated by horospheres)
or by hyperbolic translations (i.e., foliated by equidistant hypersurfaces).
In the latter case, of course, only the real hyperbolic space
$\h_\mathbb R^m:=\h^n$ is considered.

Our methods work equally well for $H_r$-hypersurfaces
with $H_r=0,$ the so called $r$-\emph{minimal} hypersurfaces.
By applying them, we  obtain a one-parameter family of
rotational, properly embedded catenoid-type $r$-minimal  $n$-annuli in $\hfr.$
Similarly, we obtain a one-parameter family of rotational, properly embedded
Delaunay-type $r$-minimal $n$-annuli
in $\s^n\times\R.$ Then, we show that
these annuli are the only complete connected
$r$-minimal rotational hypersurfaces of these product spaces (besides  horizontal hyperplanes
and, in the case $r=n,$ cylinders over geodesic spheres).

Analogously to the case of  $H_r(>0)$-hypersurfaces, we construct and classify
the complete connected $r$-minimal hypersurfaces of $\hfr$ which are invariant by either
parabolic isometries or  hyperbolic translations.

The study of $H_r$-hypersurfaces of a Riemannian
manifold leads naturally to considerations on their uniqueness properties.
On this matter, Montiel and Ros \cite{montiel-ros} (see also \cite{korevaar})
showed the following Alexandrov-type theorem:
\vt

\emph{The only compact, connected, and embedded  $H_r$-hypersurfaces
in $\R^n,$ $\h^n,$ or an open hemisphere of \,$\s^n$ are geodesic spheres.}

\vt

In \cite{elbert-earp}, this result was extended to the context of $H_r$-hypersurfaces of $\h^n\times\R,$
where the  geodesic spheres in the statement are replaced by rotational spheres.

Here, we establish uniqueness results for  rotational $H_r$-spheres
of $\h^n\times\R$ and $\s^n\times\R$, $n\ge 3.$ The case $n=2$ was settled  in
\cite{abresch-rosenberg} (for $r=1$) and in \cite{esp-gal-rosen} (for $r=2$).
More precisely, we show that any compact connected
strictly convex $H_r$-hypersurface $\Sigma$ of $\h^n\times\R$ or
$\s^n\times\R$ $(n\ge 3)$
is necessarily an embedded rotational sphere. Assuming
$\Sigma$ complete, instead of compact,  the same conclusion
holds if, in addition,
the height function of $\Sigma$ has a critical point and, in the case
$\Sigma\subset\h^n\times\R$,
the least principal curvature of $\Sigma$ is bounded away from zero. Finally, we
show that, for $n\ge 3,$ any connected,  properly immersed, and strictly
convex $H_r(>0)$-hypersurface of  \,$\s^n\times\R$ is necessarily  an embedded rotational $H_r$-sphere.

It is worth mentioning that these  uniqueness results constitute applications of
the main theorems in \cite{delima}, which concern convexity properties of
hypersurfaces in $M\times\R,$  $M$ being  either a Hadamard manifold or the sphere $\s^n.$
Besides, the noncompact cases are  based on  height estimates we establish here for
strictly convex vertical graphs in arbitrary products $M\times\R.$

The paper is organized as follows. In Section \ref{sec-preliminaries},
in addition to the usual setting of notation and basic concepts, we include
a brief presentation of the hyperbolic spaces $\hf$ and of the
Maximum-Continuation Principle for $H_r$-hypersurfaces.
In Section \ref{sec-Hrgraphs}, we introduce
graphs on parallel hypersurfaces and establish two  key lemmas.
In Section \ref{sec-rotationalHrhyp} (resp. Section \ref{sec-rotationalr-minimal}), we  construct and classify
complete rotational $H_r(>0)$-hypersurfaces (resp. $r$-minimal hypersurfaces) in $\hf\times\R$ and $\s^n\times\R,$ whereas in
Section \ref{sec-nonrotationalHrhyp} (resp. Section \ref{sec-translationalr-minimal})
we do the same for complete translational ones (i.e, invariant by either parabolic or hyperbolic isometries).
In the concluding Section \ref{sec-uniqueness}, we prove  the aforementioned uniqueness results.

\section{Preliminaries} \label{sec-preliminaries}
Let $\Sigma^n$, $n\ge 2,$
be an oriented hypersurface of  a Riemannian manifold $\overbar M^{n+1}.$ Set
$\overbar\nabla$ for the Levi-Civita connection of $\overbar M,$
$N$ for the unit normal field of $\Sigma,$ and   $A$ for its shape operator  with respect to
$N,$ so that
\[
AX=-\overbar\nabla_XN,  \,\, X\in T\Sigma,
\]
where  $T\Sigma$ stand for  the tangent bundle of $\Sigma$.
The principal curvatures of $\Sigma,$ that is,
the eigenvalues of the shape operator  $A,$ will be denoted by $k_1\,, \dots ,k_n$.

Given an integer $r\ge 0,$ we define the (non normalized)
$r$-th \emph{mean curvature} $H_r$ of the hypersurface $\Sigma\subset\overbar M$ as:
\begin{equation}
H_r:=\left\{
\begin{array}{cl}
  1 & {\rm if} \,\, r=0. \\ [1ex]
  \displaystyle\sum_{i_1<\cdots <i_r}k_{i_1}\dots k_{i_r} & {\rm if} \,\, 1\le r\le n. \\ [3ex]
  0 & {\rm if} \,\, r>n.
\end{array}
\right.
\end{equation}

Notice that $H_1$ and $H_n$ are  the  non normalized mean curvature  and
Gauss-Kronecker curvature functions of $\Sigma,$ respectively, i.e.,
\[
H_1={\rm trace}\,A\quad\text{and}\quad H_n=\det A.
\]

\begin{definition}
With the above notation, given a constant $H_r\in\R,$ we say that $\Sigma\subset\overbar M$ is an
$H_r$-\emph{hypersurface} of $\overbar M$ if
its $r$-th mean curvature is constant and equal to $H_r.$ In the case $H_r=0$,
we say that $\Sigma$ is an $r$-\emph{minimal} hypersurface of $\overbar M.$
\end{definition}

\begin{definition}
A hypersurface $\Sigma\subset\overbar M$ is said to be \emph{convex}
at $x\in\Sigma$ if, at this point,
all the nonzero principal curvatures
have the same sign. If, in addition,  none of  these principal curvatures is zero, then
$\Sigma$ is said to be \emph{strictly convex} at $x.$ We call  $\Sigma$
\emph{convex} (resp. \emph{strictly convex}) if it is convex (resp. {strictly convex}) at all of
its points.
\end{definition}

The ambient spaces we shall consider are the products
$\overbar M^{n+1}=M^n\times\R$ --- where $M^n$ is some Riemannian manifold ---
endowed with the standard product metric:
\[\langle\,,\,\rangle=\langle\,,\,\rangle_{M}+dt^2.\]
In this setting, we denote the  gradient  of the projection $\pi_{\R}$ of $M\times\R$ by $\partial_t$\,,
which is easily seen to be a parallel field on $M\times\R.$

Let $\Sigma$ be a hypersurface of $M\times\R$. Its
\emph{height function} $\xi$ and its \emph{angle function} $\Theta$
are defined by the following identities:
\[
\xi(x):=\pi_{\scriptscriptstyle\R}|_\Sigma \quad\text{and}\quad \Theta(x):=\langle N(x),\partial_t\rangle, \,\, x\in\Sigma.
\]

A critical point of $\xi$ is called \emph{horizontal}, whereas a point on which $\theta$ vanishes is
called \emph{vertical}. Notice that $x\in\Sigma$ is horizontal if and only if $\theta(x)=\pm 1.$

We shall denote the gradient field and the Hessian of a function $\zeta$ on $\Sigma$  by
$\nabla\zeta$ and ${\rm Hess}\,\zeta,$ respectively. It is easily checked that
\begin{equation} \label{eq-gradxi}
\nabla\xi=\partial_t-\Theta N \quad\text{and}\quad \nabla\theta=-A\nabla\xi.
\end{equation}

From \eqref{eq-gradxi}, for all $X,Y\in T\Sigma,$  one has
$\overbar{\nabla}_{X}\,\nabla\xi=-(\theta\overbar{\nabla}_XN+X(\theta)N).$
Hence,
\begin{equation}\label{eq-hessianheightfunction}
{\rm Hess}\,\xi(X,Y)=\Theta\langle AX,Y\rangle \,\,\, \forall X, Y\in T\Sigma.
\end{equation}

Given $t\in\R,$ the set  $P_t:=M\times\{t\}$ is called a \emph{horizontal hyperplane}
of $M\times\R.$ Horizontal hyperplanes are all isometric to $M$ and totally geodesic in
$M\times\R.$ In this context,
we call a transversal intersection $\Sigma_t:=\Sigma\transv P_t$ a \emph{horizontal section}
of $\Sigma.$ Any horizontal section $\Sigma_t$ is a hypersurface of $P_t$\,.
So, at any point $x\in\Sigma_t\subset\Sigma,$ the tangent space $T_x\Sigma$ of $\Sigma$ at $x$ splits
as the orthogonal sum
\begin{equation}\label{eq-sum}
T_x\Sigma=T_x\Sigma_t\oplus {\rm Span}\{\nabla\xi\}.
\end{equation}

We will adopt the notation $\q_\epsilon^n$
for the simply connected space form of constant sectional
curvature $\epsilon\in\{0,1,-1\};$ the Euclidean space $\R^n$ ($\epsilon=0$),
the unit sphere $\s^n$ ($\epsilon =1$), and the hyperbolic space $\h^n$ ($\epsilon=-1$).

\subsection{The hyperbolic spaces $\hf$}
Many of our results in this paper involve the Riemannian manifolds known as \emph{hyperbolic spaces}, which
include the canonical (real) $n$-dimensional hyperbolic space $\h^n.$
These manifolds are precisely the rank $1$ symmetric spaces of non-compact type, which can be
described through  the four normed division algebras:
$\R$ (real numbers), $\C$ (complex numbers), $\mathbb K$ (quaternions) and $\mathbb O$ (octonions).
They are denoted by
\[\h_\R^m, \,\,\h_\C^m, \,\, \h_{\mathbb K}^m \,\,\, \text{and} \,\,\, \h_{\mathbb O}^2, \,\,\,\, m\ge 1,\]
and called \emph{real hyperbolic space, complex hyperbolic space, quaternionic hyperbolic space} and
\emph{Cayley hyperbolic plane,} respectively.

We will adopt the unified notation $\hf$ for the hyperbolic spaces, where $m=2$ for $\mathbb F=\mathbb O.$
The real dimension  of $\hf$ is $n=m\dim \mathbb F.$ In particular,
$\h_{\mathbb O}^2$ has dimension $n=16.$ We will keep
the standard notation $\h^n$
for the real hyperbolic space $\h_\R^n$ and assume $n\ge 2.$

Denoting by  $|\,\,|$ the norm in $\mathbb F\ne\mathbb O$, and taking
in $\mathbb F^m$ standard coordinates $(z_1,\dots, z_m),$ we have that
$\hf$ is modeled by the unit ball
\[
B:=\left\{(z_1\,,\dots, z_m)\in\mathbb F^m\,;\, \sum_{i=1}^{m}|z_i|^2<1\right\}
\]
equipped with the Hermitian form whose coefficients in these coordinates are
\[
g_{ij}:=\frac{\delta_{ij}}{1-\sum_{i=1}^{m}|z_i|^2}+\frac{z_i\bar z_j}{(1-\sum_{i=1}^{m}|z_i|^2)^2}\,,
\]
where the bar denotes conjugation.

For $\mathbb F=\R,$ the metric $g$ defined by $g_{ij}$
correspond to the Klein model for $\h^n.$ Also, for $m=1$, $g$ reduces
to the canonical Poincaré metric of $\h^n$. In particular,
$\h_{\mathbb C}^1$ is isometric to $\h^2$, and
$\h_{\mathbb K}^1$ is isometric to $\h^4$.
(See \cite{chen-greenberg}
for a detailed discussion on the hyperbolic spaces $\hf$, $\mathbb F\ne\mathbb O.$)

The description of a model for the Cayley hyperbolic plane $\h_\mathbb O^2$ is more involved. We refer
to \cite{aravinda-farrel} and the references therein for an account on this space.

We remark that, being symmetric, the
hyperbolic spaces $\hf$ are homogeneous.
In addition, they
are included in a distinguished class of
Lie groups known as \emph{Damek-Ricci spaces} (see Example \ref{exam-damek-ricci} in the next
section). In this context, it can be shown that any hyperbolic
space $\hf$ is a Hadamard-Einstein manifold
with nonconstant (except for $\h^n$) sectional curvatures pinched between $-1$ and $-1/4$
(cf. \cite[Sections 4.1.9 and 4.2]{berndtetal}).

\subsection{The Maximum-Continuation Principle}
Two major tools employed in the study of hypersurfaces of  constant curvature
are the Maximum Principle and the Continuation Principle for solutions of
elliptic PDE's.
In the case of the $r$-th mean curvature $H_r$, it was shown in \cite[Proposition 3.2]{cheng-rosenberg}
that a hypersurface $\Sigma$ of a Riemannian manifold $M$ with $H_r>0$ which is strictly convex at a point
is given locally by a graph of a solution of an elliptic PDE.
From the Continuation Principle (see, e.g., \cite{kazdan, protter}), if two such solutions are defined in a domain $\Omega$ and
coincide in a subdomain $\Omega'\subset\Omega,$ then they coincide in $\Omega.$
These facts, together with \cite[Theorem 1.1]{fontenele-silva},  give the following
result.

\begin{MCP}
Let  $\Sigma_1\,, \Sigma_2\subset M$ be  complete connected $H_r(>0)$-hypersurfaces of a
Riemannian manifold $M$ which are tangent at a point $x\in\Sigma_1\cap\Sigma_2$\,.
Let $N(x)\in T_x\Sigma_i^\perp$ $(i=1,2)$ be the common unit normal and assume that
$\Sigma_2$ is strictly convex at one of its points. Then, if
(near $x$) $\Sigma_1$ remains above
$\Sigma_2$ with respect to $N(x),$  one has $\Sigma_1=\Sigma_2$\,.
\end{MCP}

\section{$H_r$-Graphs on  Parallel Hypersurfaces}  \label{sec-Hrgraphs}
In this section, we give a detailed description of   graphs in $M\times\R$
which are built on families of parallel hypersurfaces of $M.$ As we mentioned before,
they will  constitute our main tool for constructing
$H_r$-hypersurfaces in product spaces $M\times\R.$

With this purpose, consider
an isometric immersion
\[f:M_0^{n-1}\rightarrow M^n\]
between two Riemannian manifolds $M_0^{n-1}$ and $M^n,$
and suppose that there is a neighborhood $\mathscr{U}$
of $M_0$ in $TM_0^\perp$  without focal points of $f,$ that is,
the restriction of the normal exponential map $\exp^\perp_{M_0}:TM_0^\perp\rightarrow M$ to
$\mathscr{U}$ is a diffeomorphism onto its image. In this case, denoting by
$\eta$ the unit normal field  of $f,$   there is an open interval $I\owns 0,$
such that, for all $p\in M_0,$ the curve
\begin{equation}\label{eq-geodesic}
\gamma_p(s)=\exp_{\scriptscriptstyle M}(f(p),s\eta(p)), \, s\in I,
\end{equation}
is a well defined geodesic of $M$ without conjugate points. Thus,
for all $s\in I,$
\[
\begin{array}{cccc}
f_s: & M_0 & \rightarrow & M\\
     &  p       & \mapsto     & \gamma_p(s)
\end{array}
\]
is an immersion of $M_0$ into $M,$ which is said to be \emph{parallel} to $f.$
Observe that, given $p\in M_0$, the tangent space $f_{s_*}(T_p M_0)$ of $f_s$ at $p$ is the parallel transport of $f_{*}(T_p M_0)$ along
$\gamma_p$ from $0$ to $s.$ We also remark that,  with the induced metric,
the unit normal  $\eta_s$  of $f_s$ at $p$ is given by
\[\eta_s(p)=\gamma_p'(s).\]

\begin{definition}
Let $\phi:I\rightarrow \phi(I)\subset\R$ be an increasing diffeomorphism, i.e., $\phi'>0.$
With the above notation, we call the set
\begin{equation}\label{eq-paralleldescription1}
\Sigma:=\{(f_s(p),\phi(s))\in M\times\R\,;\, p\in M_0, \, s\in I\},
\end{equation}
the \emph{graph} determined by $\{f_s\,;\, s\in I\}$ and $\phi,$ or $(f_s,\phi)$-\emph{graph}, for short.
\end{definition}

Notice that, for a given $(f_s,\phi)$-graph $\Sigma$, and
for any $s\in I$\,, $f_s(M_0)$ is  the projection on $M$ of the horizontal
section $\Sigma_{\phi(s)}\subset\Sigma,$
that is, these sets are the {level hypersurfaces}
of    $\Sigma.$

For an arbitrary  point $x=(f_s(p),\phi(s))$ of such a graph $\Sigma,$ one has
\[T_x\Sigma=f_{s_*}(T_p M_0)\oplus {\rm Span}\,\{\partial_s\}, \,\,\, \partial_s=\eta_s+\phi'(s)\partial_t.\]
So, a  unit normal  to $\Sigma$ is
\begin{equation} \label{eq-normal}
N=\frac{-\phi'}{\sqrt{1+(\phi')^2}}\eta_s+\frac{1}{\sqrt{1+(\phi')^2}}\partial_t\,.
\end{equation}
In particular, its  angle function  is
\begin{equation} \label{eq-thetaparallel}
\Theta=\frac{1}{\sqrt{1+(\phi')^2}}\,\cdot
\end{equation}

A key property of
  $(f_s,\phi)$-graphs is that the trajectories of
$\nabla\xi$ on them are lines of curvature, that is, $\nabla\xi$ is one of its  principal directions.
(Notice that, by \eqref{eq-thetaparallel}, $0<\Theta<1,$ so $\nabla\xi$ never vanishes on an $(f_s,\phi)$-graph.)
More precisely (cf. \cite{delima-roitman, tojeiro}),
\begin{equation}\label{eq-principaldirection}
A\nabla\xi=\frac{\phi''}{(\sqrt{1+(\phi')^2})^3}\nabla\xi.
\end{equation}

We  point out that,
%A key feature of the trajectories of $\nabla\xi$ on an $(f_s\phi)$-graph $\Sigma,$
besides being lines of curvature, the trajectories of $\nabla\xi$ on an $(f_s,\phi)$-graph $\Sigma$,
when properly reparametrized, are also geodesics. This follows from the fact that
$\theta$, and consequently $\|\nabla\xi\|,$ is constant along the horizontal sections
of $\Sigma$ (see \cite[Lemma 5]{tojeiro}). It should also be noticed that
these trajectories  project on the geodesics $\gamma_p=\gamma_p(s)$
given by \eqref{eq-geodesic} (Fig. \ref{fig-phigraph}).

\begin{figure}[htbp]
\includegraphics{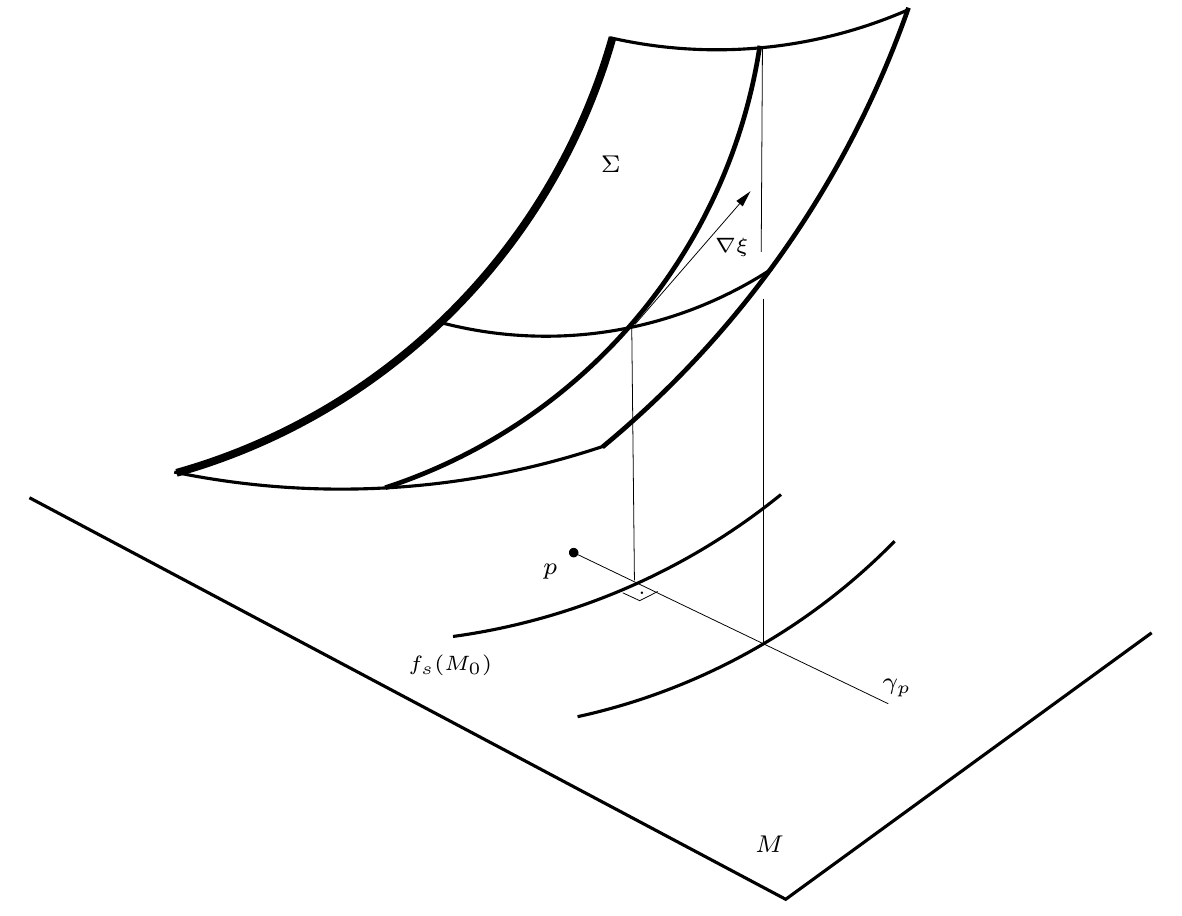}
\caption{Trajectory of $\nabla\xi$ on an $(f_s,\phi)$-graph.}
\label{fig-phigraph}
\end{figure}

Let us compute now the principal curvatures of an
$(f_s,\phi)$-graph $\Sigma.$
For that, let $\{X_1\,,\dots ,X_n\}$ be
the orthonormal frame of principal directions of $\Sigma$
in which $X_n=\nabla\xi/\|\nabla\xi\|.$ In this case, for $1\le i\le n-1,$
the fields $X_i$  are all horizontal, that is, tangent to $M$ (cf. \eqref{eq-sum}).
Therefore, setting
\begin{equation}\label{eq-rho}
\rho:=\frac{\phi'}{\sqrt{1+(\phi')^2}}
\end{equation}
and considering \eqref{eq-normal}, we have, for all $i=1,\dots ,n-1,$ that
\[
k_i=\langle AX_i,X_i\rangle=-\langle\overbar\nabla_{X_i}N,X_i\rangle=\rho\langle\overbar\nabla_{X_i}\eta_s,X_i\rangle=-\rho k_i^s,
\]
where $k_i^s$ is the $i$-th principal curvature of $f_s\,.$ Also,
it follows from \eqref{eq-principaldirection}
that $k_n=\rho'.$ Thus, the array of principal curvatures of the $(f_s,\phi)$-graph $\Sigma$ is
\begin{equation}\label{eq-principalcurvatures}
k_i=-\rho k_i^s \,\, (1\le i\le n-1) \quad\text{and}\quad k_n=\rho'.
\end{equation}

Now, considering the above identities and writing,  for $1\le r\le n,$
\[
  H_r = \sum_{i_1<\cdots <i_r\ne n}k_{i_1}\dots k_{i_r}+\sum_{i_1<\cdots <i_{r-1}}k_{i_1}\dots k_{i_{r-1}}k_{n}\,,
\]
we have that the $r$-th mean curvature of the $(f_s,\phi)$-graph $\Sigma$ is
\begin{equation} \label{eq-Hr}
H_r= (-1)^rH_r^s\rho^r+(-1)^{r-1}H_{r-1}^s\rho^{r-1}\rho',
\end{equation}
where $H_r^s$ denotes the $r$-th mean curvature of $f_s.$

Due to equality \eqref{eq-Hr},
the function defined in \eqref{eq-rho} --- to be called the $\rho$-\emph{function} of
the $(f_s,\phi)$-graph $\Sigma$ --- will play a major role in the sequel.
We remark that, up to a constant, the $\rho$-function of $\Sigma$ determines its
$\phi$-function.
Indeed, it follows from equality \eqref{eq-rho} that
\begin{equation}\label{eq-phi}
\phi(s)=\int_{s_0}^{s}\frac{\rho(u)}{\sqrt{1-\rho^2(u)}}du+\phi(s_0), \,\,\, s_0\in I.
\end{equation}

We introduce  now a special type of family of parallel hypersurfaces which will
be used for constructing $H_r$-hypersurfaces in $M\times\R.$

\begin{definition}
Following \cite{berndtetal}, we call a parallel family $\{f_s:M_0\rightarrow M\,;\, s\in I\}$ \emph{isoparametric} if,
for each $s\in I,$ any principal curvature of $f_s$ is  constant (depending on $s$).
If so, each hypersurface $f_s$ is also called \emph{isoparametric.}
\end{definition}

We should mention that there is some mismatch regarding the nomenclature
for isoparametric hypersurfaces. In some contexts,
isoparametric hypersurfaces are defined as those which, together
with its parallel nearby hypersurfaces, have constant mean curvature.
It is shown that some manifolds $M$ admit  hypersurfaces which are
isoparametric in this sense, and non isoparametric  as we defined.

Let $\Sigma$ be an $(f_s,\phi)$-graph such that the parallel family
$\{f_s\,;\, s\in I\}$ is isoparametric. Then, for any
$r=1,\dots ,n-1,$ the $r$-th mean curvature of $f_s$ is a function of
$s$ alone, which  we assume to be no vanishing.
In this setting, writing $\tau:=\rho^r$, and
considering \eqref{eq-Hr} with $H_r$ constant, we obtain
the following result, which turns out to be our main lemma.

\begin{lemma} \label{lem-parallel}
Let $\{f_s:M_0\rightarrow M\,;\, s\in I\}$ be an
isoparametric family of hypersurfaces whose $r(<n)$-mean curvatures $H_r^s$ never vanish.
Given $r\in\{1,\dots ,n\}$ and $H_r\in\R,$ let $\tau$ be a solution of the first order
differential equation
\begin{equation}  \label{eq-difequation}
y'=a(s)y+b(s), \,\,\, s\in I,
\end{equation}
where the coefficients $a$ and $b$ are the functions
\begin{equation}  \label{eq-a&b}
a(s):=\frac{rH_r^s}{H_{r-1}^s}\quad{\rm and}\quad b(s):=\frac{(-1)^{r-1}rH_r}{H_{r-1}^s}\,\cdot
\end{equation}
Then, if  \,$0<\tau<1,$  the
$(f_s,\phi)$-graph $\Sigma$ with $\rho$-function $\tau^{1/r}$  is an $H_r$-hypersurface of
the product $M\times\R.$ Conversely, if an $(f_s,\phi)$-graph $\Sigma$ has constant $r$-th mean curvature $H_r$\,,
then $\tau:=\rho^r$ is a solution of \eqref{eq-difequation}.
\end{lemma}

Regarding equation \eqref{eq-difequation}, recall that
its general solution is
\begin{equation}\label{eq-solutiondifequation}
\tau(s)=\frac{1}{\mu(s)}\left(\tau_0+\int_{s_0}^{s}{b(u)}{\mu(u)}du\right), \,\, s_0, s\in I, \, \tau_0\in\R,
\end{equation}
where $\mu$ is the exponential function
\[\mu(s)=\exp\left(-\int_{s_0}^sa(u)du\right), \,\,\, s\in I.\]

It follows from Lemma \ref{lem-parallel} that, as long as $M$ admits isoparametric
hypersurfaces with non vanishing $r$-th mean curvature,
for any $H_r\in\R,$  there exist $H_r$-graphs in $M\times\R$. (Notice that,
the interval $I$ and the constant $\tau_0$ in \eqref{eq-solutiondifequation} can be chosen
in such a way that the corresponding solution $\tau$ of \eqref{eq-difequation} satisfies $0<\tau<1$.)
This includes, as a trivial case,  the Euclidean space $\R^n.$
In the next examples, we shall consider other manifolds $M$ to which  Lemma \ref{lem-parallel}  applies.

\begin{example}[\emph{sphere $\s^n$}] It is a well known fact that isoparametric hypersurfaces
in $\s^n$  are abundant and include all its geodesic spheres.
In fact, the classification of the isoparametric hypersurfaces of $\s^n$ is a long stand open problem
(see, e.g., \cite{dominguez-vazquez}).
\end{example}

\begin{example}[\emph{warped products}]
Let $M=I\times_\omega F^{n-1}$ be a warped product, where the basis $I\subset\R$ is an open interval and
the fiber $F$ is an arbitrary $(n-1)$-dimensional manifold. For each $s\in I,$
define $f_s$ as the standard  immersion $F\hookrightarrow \{s\}\times_\omega F\subset M.$
It is well known that, with the induced metric,  $\mathscr F=\{f_s\,;\, s\in I\}$ is a parallel
family of totally umbilical hypersurfaces
of $M$ with constant principal  curvatures $\omega'/\omega$ (see, e.g., \cite{bishop-oneill}).
In particular, $\mathscr F$ is isoparametric. Hence, if $\omega'$ never vanishes,
Lemma \ref{lem-parallel} applies to $M.$
\end{example}

\begin{example}[\emph{Damek-Ricci spaces}] \label{exam-damek-ricci}
Let us consider the Riemannian manifolds known as \emph{Damek-Ricci spaces}.
These are Lie groups endowed with a left invariant metric with especial
 properties (see \cite{berndtetal, dominguez-vazquez}). For instance,
all Damek-Ricci spaces are both Hadamard and  Einstein manifolds.
As we have mentioned, the hyperbolic spaces $\hf$ are  Damek-Ricci spaces.
In fact, they are the only ones which are symmetric. 
Their isoparametric hypersurfaces include 
their geodesic spheres, as well as their horospheres.
%(cf. \cite[Theorem 7 - Chapter 4]{berndtetal}).
Finally, we point out  that,
in \cite{diazramos-dominguezvazquez}, the authors obtained
families of isoparametric hypersurfaces with non vanishing
$r$-th curvatures in Damek-Ricci harmonic spaces.
%Thus, Lemma \ref{lem-parallel} applies to all these Damek-Ricci spaces.
%From Lemma \ref{lem-parallel}, given $H_r\in\R,$ each of these families
%can be used for constructing  $H_r$-graphs in $M\times\R,$
%where $M$ is the corresponding Damek-Ricci space.
\end{example}

\begin{example}[$E(\kappa,\tau)$-\emph{spaces}]
In \cite{dominguezvazquez-manzano}, it was proved that there exist
isoparametric families of
parallel surfaces with nonzero constant principal curvatures  in
$\mathbb E(k,\tau)$ spaces satisfying $k-4\tau^2\ne 0.$ (Those include
the products $\h^2\times\R$ and $\s^2\times\R$, the Heisenberg space ${\rm Nil}_3$,
the Berger spheres, and  the universal cover of the special linear group ${\rm SL}_2(\R)$).
%As in the previous examples, these families may be used for
%constructing $H_r$-hypersurfaces in $\mathbb E(k,\tau)\times\R$ for any $H_r\in\R.$
\end{example}

In the next two sections we construct properly embedded $H_r$-hypersurfaces in products
$M\times\R$ by suitably ``gluing''  $H_r$-graphs. To this task,
the following elementary fact will be considerably helpful.

\begin{lemma} \label{lem-convergenceintegral}
Let $\rho:[a_1,a_2]\rightarrow\R$ be a differentiable function which satisfies:
\[0<\rho{|_{(a_1,a_2)}}<1.\]
Assume that one (or both) of the following hold:
\begin{itemize}[parsep=1ex]
\item[\rm i)] $\rho(a_2)=1$ and $\rho'(a_2)>0.$
\item[\rm ii)] $\rho(a_1)=1$ and $\rho'(a_1)<0.$
\end{itemize}
Under these conditions, there exists $\delta >0$ such that
the improper integral
\begin{equation}  \label{eq-integral}
\int_{a_2-\delta}^{a_2}\frac{\rho(s)ds}{\sqrt{1-\rho^2(s)}}
\end{equation}
is convergent if \,{\rm (i)} occurs. Analogously, the improper  integral
\[
\int_{a_1}^{a_1+\delta}\frac{\rho(s)ds}{\sqrt{1-\rho^2(s)}}
\]
is convergent if {\rm (ii)} occurs.
\end{lemma}
\begin{proof}
Assume that (i) occurs.
In this case, there exist positive constants, $\delta$ and $C,$ such that
$\rho'(s)\ge C>0\,\forall s\in(a_2-\delta ,a_2).$ Therefore,
since $0<\rho{|_{(a_1,a_2)}}<1,$
\begin{eqnarray}
\int_{a_2-\delta}^{a_2}\frac{\rho(s)ds}{\sqrt{1-\rho^2(s)}} & \le & \int_{a_2-\delta}^{a_2}\frac{\rho'(s)ds}{\rho'(s)\sqrt{1-\rho^2(s)}}
\le\frac{1}{C} \int_{\rho(a_2-\delta)}^{1}\frac{d\rho}{\sqrt{1-\rho^2}} \nonumber\\
            & = & \frac{1}{C}\left(\frac{\pi}{2}-\arcsin(\rho(a_2-\delta))\right)\le \frac{\pi}{2C}\,, \nonumber
\end{eqnarray}
which proves the convergence of the integral \eqref{eq-integral}. The case (ii) is analogous.
\end{proof}

\section{Rotational $H_r(>0)$-hypersurfaces of $\hfr$ and $\s^n\times\R.$} \label{sec-rotationalHrhyp}

Rotational hypersurfaces in  simply connected space forms $\q^n_\epsilon$ or products
$\q_\epsilon^n\times\R$ are among the most classical  hypersurfaces  of these spaces.
In the case of $\q^n_\epsilon\times\R,$  they are obtained
by rotating (with the aid of the group of isometries of $\q_\epsilon^n$)
a  plane curve about an axis $\{o\}\times\R,$  $o\in\q_\epsilon^n$.
Consequently, any connected component of any horizontal section $\Sigma_t$  of a rotational hypersurface
$\Sigma$ in $\q_\epsilon^n\times\R$ lies in a  geodesic sphere
of $\q_\epsilon^n\times\{t\}$ with center at the axis.
%Conversely, any hypersurface of \,$\q_\epsilon^n\times\R$ with this property is easily seen to be rotational.
This fact suggests the following definition.

\begin{definition}
A hypersurface  $\Sigma\subset M\times\R$ is called \emph{rotational}, if there exists
a fixed point $o\in M$ such that any connected component of
any horizontal section $\Sigma_t$ is contained in a geodesic  sphere  of $M\times\{t\}$
with center at $o\times\{t\}.$
If so, the set $\{o\}\times\R$ is called the \emph{axis} of $\Sigma.$
In particular, any horizontal hyperplane
$P_t:=M\times\{t\}$ is a rotational hypersurface of
$M\times\R$ with axis at any point $o\in M.$
\end{definition}

In what follows, we construct and classify
complete rotational $H_r(>0)$-hyper\-surfaces in $\hfr$ and
$\s^n\times\R.$

\subsection{Rotational $H_r(>0)$-hypersurfaces of $\hfr$}
Let us consider a family
\begin{equation} \label{eq-Fgeodesicspheres}
\mathscr F:=\{f_s:\s^{n-1}\rightarrow\hf\,; \,\,\, s\in (0,+\infty)\}
\end{equation}
of isoparametric concentric geodesic spheres
of $\hf$ indexed by their radiuses, that is,
for a fixed  $o\in\hf$, and  each $s\in (0,+\infty),$
$f_s(\s^{n-1})$ is the geodesic sphere $S_s(o)$ of $\hf$ with center at $o$ and radius $s.$

We remark that any sphere $f_s\in\mathscr F$ is strictly convex.
Also, in accordance to the notation of Section \ref{sec-Hrgraphs}, for each $s\in (0,+\infty),$
we  choose the outward orientation of $f_s$\,,
so that \emph{any principal curvature $k_i^s$ of \,$f_s$ is negative}.
In this setting, the function $s\in (0,+\infty) \mapsto H_r^s$ is positive for
$r$ even and negative for $r$ odd. Hence,
for any constant $H_r>0,$ the coefficients $a$ and $b$  in \eqref{eq-a&b} are given by
\begin{equation} \label{eq-a&bagain}
a(s)=-\frac{r|H_r^s|}{|H_{r-1}^s|} \quad\text{and}\quad b(s)=\frac{r H_r}{|H_{r-1}^s|} \,\cdot
\end{equation}

The principal curvatures $k_i^s$ of the geodesic spheres $f_s\in\mathscr F$ are ($n=\dim\hf$):
\begin{equation} \label{eq-princcurvhf}
\begin{aligned}
    k_1^s &= -\frac{1}{2}\coth(s/2) \,\,\,\text{with multiplicity}\,\,\,  n-p-1 \\[1.5ex]
    k_2^s &=-\coth(s) \,\,\, \text{with multiplicity} \,\,\,  p
  \end{aligned}\,\,,
\end{equation}
where $p=n-1$ for $\h^n$,  $p=1$ for $\h_\C^m$, $p=3$ for $\h_{\mathbb K}^m$, and
$p=7$ for $\h_{\mathbb O}^2$  (see, e.g., \cite[pgs. 353, 543]{cecil-ryan} and \cite{kimetal}).

From equalities \eqref{eq-princcurvhf}, we obtain the $r$-mean curvatures
$H_r^s$ of the geodesic spheres $f_s$ of $\hf.$
For instance, in $\h^n,$ $n\ge 2,$ we have
\begin{equation}  \label{eq-hrhn}
H_r^s=(-1)^r{{n-1}\choose{r}}\coth^r(s) \quad (1\le r\le n-1),
\end{equation}
whereas for $\h_\C^m,\, m>1,$ one has
\begin{eqnarray}  \label{eq-hrhc}
H_r^s &=& \left(-\frac 12\right)^r{{n-2}\choose{r}}\coth^r(s/2)\\
       &+& (-1)^r\left(\frac 12\right)^{r-1}{{n-2}\choose{r-1}}\coth^{r-1}(s/2)\coth(s)\nonumber
\end{eqnarray}
if $1\le r\le n-2$,  and
\begin{equation}  \label{eq-hrhc2}
H_{n-1}^s=(-1)^{n-1}\coth(s)\coth^{n-2}(s/2)/2^{n-2}.
\end{equation}

Analogously, one obtains the $r$-th mean curvature functions
$H_r^s$ for the other hyperbolic spaces.
A direct computation from this data yields the following

\begin{lemma} \label{lem-decreasing}
The functions $a$ e $b$ defined in \eqref{eq-a&bagain} have the following properties:
\begin{itemize}[parsep=1ex]
\item [\rm i)] $a$ is negative and increasing for $1\le r\le n-1$, and vanishes for $r=n.$
\item [\rm ii)] $b$ is positive and increasing for $1<r\le n$, and $b=H_1$ for $r=1.$
\end{itemize}
In particular, we have the inequalities
\begin{equation}\label{eq-a'+b'}
a'(s)\ge 0, \,\,\,\,  b'(s)\ge 0, \,\, \,\,\, \text{and}\ \,\,\,\, a'(s)+b'(s)>0 \,\, \, \forall s\in(0,+\infty).
\end{equation}
\end{lemma}

We point out that, in the above setting,
one has (cf. \cite{mahmoudi})
\begin{equation} \label{eq-approximationHrs}
|H_r^s|={{{n-1}\choose{r}}}s^{-r}+\mathcal O(s^{2-r})
\end{equation}
in a neighborhood of $s=0.$ In particular,
\begin{equation} \label{eq-hrinfity}
\lim_{s\rightarrow 0}|H_r^s|=+\infty.
\end{equation}

In what follows, by means of the family $\mathscr F,$ we
will construct  complete  rotational $H_r$-hyper\-sur\-faces in $\hf\times\R$
which  are made of pieces of $(f_s,\phi)$-graphs. With this purpose, we will look for
solutions $\tau(s)$ of the equation $y'=ay+b$ (with $a$ and $b$ as in \eqref{eq-a&bagain})
satisfying  suitable initial conditions.
%As we shall see,
%the geometry  of the  resulting $H_r$-hypersurface depends on the behavior of the solution $\tau(s)$
%as $s\rightarrow\mathcal R.$
Let us recall that, in this context, the general solution of
$y'=ay+b$ is
\begin{equation} \label{eq-generalsolution}
\begin{aligned}
    \tau(s)&=\frac{1}{\mu(s)}\left(\tau_0+\int_{s_0}^{s}{b(u)}{\mu(u)}du\right) \\[1.5ex]
    \mu(s)&=\exp\left(-\int_{s_0}^sa(u)du\right)
  \end{aligned}\,\,,
  \quad s_0\,, s\in (0,+\infty), \,\,\tau_0\in\R.
\end{equation}

Concerning the solutions $\tau(s)$, we will be also interested in those
which can be extended to $s=0.$ Notice that,
in principle, neither $a$ nor $b$ are defined at $s=0$, which makes this point
a singularity. However,
the function $b$ is easily extendable to $s=0.$ Indeed,  we can set $b(0)=H_r$ if $r=1$,
and $b(0)=0$ if $r>1$ (from \eqref{eq-hrinfity}).
As for $a,$ it follows from \eqref{eq-approximationHrs} that, for $1\le r<n,$
\begin{equation}  \label{eq-|a|}
\lim_{s\rightarrow 0}|a(s)|=+\infty \quad\text{and}\quad \lim_{s\rightarrow 0}s|a(s)|<+\infty,
\end{equation}
which characterizes $s=0$ as a regular singular point of
$y'=ay+b.$ This means that, despite the fact that $a$ is not defined at $s=0,$
this equation has a nonnegative solution $\tau$ defined at  $s=0$
that satisfies $\tau(0)=0$ (cf. \cite[Theorem 3.1]{reignier}, \cite[Lemma 4.4]{teschl}).
More precisely, this solution is
\begin{equation}\label{eq-tau}
\tau(s):=\left\{
\begin{array}{lcl}
  \frac{1}{\mu(s)}\int_{0}^{s}{b(u)}{\mu(u)}du & \text{if} & s\in (0,+\infty) \\[1.5ex]
  0 & \text{if} & s=0
\end{array}
\right.\,,
\end{equation}
where $\mu$ is a solution of $y'+ay=0.$
Notice that the function $\tau$ defined in \eqref{eq-tau} is also the solution of
$y'=ay+b$ in the case $r=n,$ i.e., for $a=0.$ (Just set $\mu(s)=1.$)

As we shall see, the geometry of the $H_r$-hypersurfaces we construct from $(f_s,\phi)$-graphs is closely
related to the growth of $\tau$ as $s\rightarrow+\infty.$ Taking that into account,
for a given family of parallel geodesic spheres $\mathscr F$ in $\hf,$ we define
\begin{equation}\label{eq-CFr}
  C_{\mathbb F}(r):=\lim_{s\rightarrow+\infty}|H_r^s|, \,\,\, r=1,\dots ,n.
\end{equation}
In particular, $C_{\mathbb F}(n)=0.$
Notice that, since $\hf$ is homogeneous, the constant $C_{\mathbb F}(r)$ is well defined, that is, it does not
depend on the family $\mathscr F$ of geodesic spheres.

It follows from equalities \eqref{eq-hrhn}--\eqref{eq-hrhc2} that
\begin{itemize}[parsep=1ex]
  \item [i)] $C_\R(r)={{n-1}\choose{r}} \quad (1\le r\le n-1).$
  \item [ii)] $C_\C(r)=\left(\frac 12\right)^r{{n-2}\choose{r}}+\left(\frac 12\right)^{r-1}{{n-2}\choose{r-1}} \quad (1\le r\le n-2).$
  \item [iii)] $ C_\C(n-1)=\frac{1}{2^{n-2}}\,\cdot$
\end{itemize}

Similarly, one can compute the other constants $C_{\mathbb F}(r)$  and easily conclude that
\[
C_\mathbb F(r)>0 \,\,\, \forall r\in\{1,\dots ,n-1\}.
\]

The next proposition shows the relation between the solution $\tau$ of
$y'=ay+b$ and the constant $C_{\mathbb F}(r).$ Notice that, for $1\le r\le n-1,$  the identities
\eqref{eq-a&bagain} yield:
\[
\lim_{s\rightarrow+\infty}\frac{-b(s)}{\phantom{-}a(s)}=\frac{H_r}{C_{\mathbb F}(r)}\,\cdot
\]

\begin{proposition} \label{pro-limttauR}
The following assertions hold:
\begin{itemize}[parsep=1ex]
  \item [\rm i)] The solution $\tau$ defined in \eqref{eq-tau} is increasing, i.e., $\tau'>0$ in $(0,+\infty).$
  \item [\rm ii)] Both the solutions $\tau$ in \eqref{eq-generalsolution} and \eqref{eq-tau} satisfy the following equality:
\begin{equation} \label{eq-limF}
\lim_{s\rightarrow+\infty}\tau(s)=
\left\{
\begin{array}{lcl}
+\infty & \text{if} &  r=n.\\[1.5ex]
H_r/C_{\mathbb F}(r) & \text{if} &  1\le r\le n-1.
\end{array}
\right.
\end{equation}
\end{itemize}
\end{proposition}
\begin{proof}
To prove (i), let us observe first  that, since the solution  $\tau$
in \eqref{eq-tau} is positive in $(0,+\infty)$ and $\tau(0)=0,$ we have
that $\tau$ is increasing near $0.$ Assume that  $\tau$ is not
increasing in $(0,+\infty).$ In this case, $\tau$ has a first critical point $s_0$ in
$(0,+\infty)$ which is necessarily a local maximum.
However, considering \eqref{eq-a'+b'} and the equality $\tau'=a\tau+b,$ we have
\[
\tau''(s_0)=a'(s_0)\tau(s_0)+a(s_0)\tau'(s_0)+b'(s_0)=a'(s_0)\tau(s_0)+b'(s_0)>0,
\]
which implies that $s_0$ is  a local minimum --- a contradiction.
Therefore, $\tau$ is increasing in $(0,+\infty),$  which proves (i).

Suppose that $\tau$ is as in \eqref{eq-generalsolution}. If $r=n,$
since $b$ is increasing, one has
\[
\tau(s)=\tau_0+\int_{s_0}^sb(u)du\ge\tau_0+b(s_0)(s-s_0),
\]
which implies that $\tau(s)\rightarrow+\infty$ as $s\rightarrow+\infty.$

Now, assume $1\le r\le n-1.$ We claim
that, in this case, $\mu(s)\rightarrow+\infty$ as $s\rightarrow+\infty$.
Indeed, for any fixed $s_0>0,$ and $s>s_0$\,,
\[
|a(s)|=\frac{r|H_r^s|}{|H_{r-1}^s|}>\frac{rC_{\mathbb F}(r)}{|H_{r-1}^{s_0}|}>0.
\]
Since $|a|$ is decreasing, this  inequality gives
that $\inf |a|>0$ in $(0,+\infty).$ Hence,
\[
\mu(s)=e^{-\int_{s_0}^{s}a(u)du}=e^{\int_{s_0}^{s}|a(u)|du}>e^{(\inf|a|) (s-s_0)},
\]
which clearly implies the claim on $\mu.$

From the expression of $\tau,$ we have (apply  l'Hôpital to the second summand):
\[
\lim_{s\rightarrow +\infty}\tau(s)=\lim_{s\rightarrow+\infty}\frac{\tau_0}{\mu(s)}+\lim_{s\rightarrow+\infty}\frac{b(s)\mu(s)}{\mu'(s)}=
\lim_{s\rightarrow+\infty}\frac{b(s)\mu(s)}{-a(s)\mu(s)}=\frac{H_r}{C_{\mathbb F}(r)}\,,
\]
which finishes the proof of \eqref{eq-limF} when $\tau$ is defined as  in \eqref{eq-generalsolution}.
The proof  for the solution $\tau$ in \eqref{eq-tau} is completely analogous.
\end{proof}

The above proposition immediately gives the following result.

\begin{corollary} \label{cor-s0}
Let $\tau$ be as in \eqref{eq-tau}. Then,
there exists $s_0\in(0, +\infty)$  satisfying
\[
0<\tau|_{(0,s_0)}<1 \,\,\,\, \text{and} \,\,\,\,\, \tau(s_0)=1
\]
if and only if $H_r>C_{\mathbb F}(r)$ (Fig. \ref{fig-graph1}).
Consequently,  for $1\le r\le n-1$ and any constant $H_r\in(0,C_{\mathbb F}(r)),$ the following inequality
holds (Fig. \ref{fig-graph2}):
\[
0<\tau(s)<1 \,\,\, \forall s\in (0,+\infty).
\]
\end{corollary}

\begin{figure} %\label{fig-graphs}
  \begin{subfigure}[hbt]{.5\textwidth}
    \includegraphics{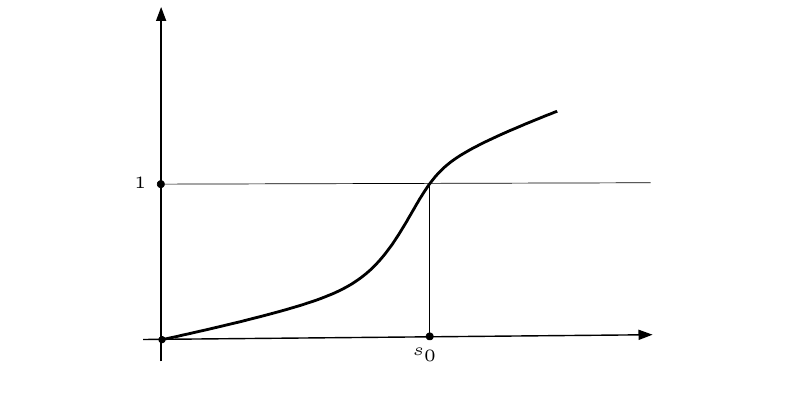}
    \caption{\small \,$H_r>C_{\mathbb F}(r)$}
    \label{fig-graph1}
  \end{subfigure}
  \hspace{1.5cm}%
  \begin{subfigure}[h]{.5\textwidth}
     \includegraphics{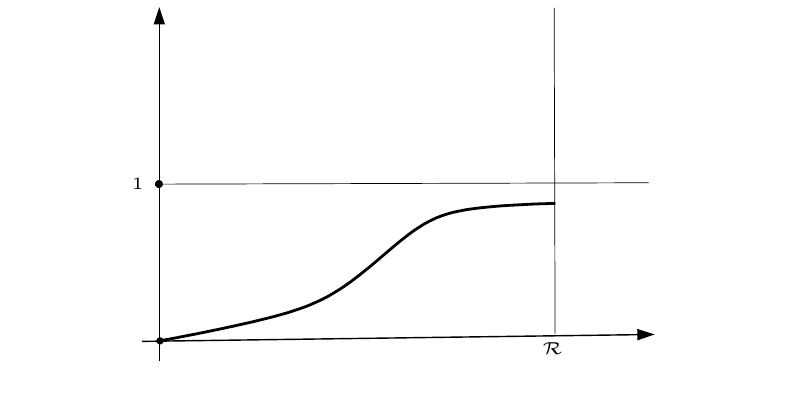}
    \caption{\small \,$0<H_r\le C_{\mathbb F}(r)$}
    \label{fig-graph2}
  \end{subfigure}
  \caption{\small Graphs of $\tau$ (as in \eqref{eq-tau}) according to the sign  of $H_r-C_{\mathbb F}(r).$}
\end{figure}

Now, we are in position to establish our first existence result.

\begin{theorem} \label{th-Hrrotational}
Given $r\in\{1,\dots, n\}$ and a constant $H_r>0,$ the following hold:
\begin{itemize}[parsep=1ex]
\item[\rm i)] If \,$H_r>C_{\mathbb F}(r),$  there exists an embedded strictly convex
rotational $H_r$-sphere in $\hfr$ which is symmetric with respect to a horizontal hyperplane.

\item[\rm ii)] If \,$0<H_r\le C_{\mathbb F}(r),$  there
exists an entire strictly convex rotational $H_r$-graph  in $\hf\times[0,+\infty)$
which is tangent to $\hf\times\{0\}$ at a single point, and whose height function is unbounded above.
Consequently, there are no compact  $H_r$-hypersurfaces of
\,$\hfr$ for such values of $H_r$\,.
\end{itemize}
\end{theorem}

\begin{proof}
Let $\mathscr F$ be an arbitrary  family of parallel geodesic spheres of
$\hf$ as in \eqref{eq-Fgeodesicspheres}. Consider the functions
$a$ and $b$ defined in \eqref{eq-a&bagain} and let
$\tau$ be the solution \eqref{eq-tau} of the ODE $y'=ay+b.$

If  $H_r>C_{\mathbb F}(r),$ we have from Corollary \ref{cor-s0} that there exists
$s_0\in (0,+\infty)$ satisfying
\[
\tau(0)=0<\tau|_{(0,s_0)}<1=\tau(s_0).
\]
Hence, by Lemma \ref{lem-parallel}, the $(f_s,\phi)$-graph
$\Sigma'$ with $\rho$-function $\rho:=\sqrt[r]{\tau|_{[0,s_0)}}$ is a rotational
$H_r$-graph of $\hfr$ over the open ball $B_{s_0}(o)\subseteq\hf$  such that
\begin{equation}\label{eq-phi}
\phi(s)=\int_{0}^{s}\frac{\rho(u)}{\sqrt{1-\rho^2(u)}}du, \,\,\,\, s\in [0,s_0).
\end{equation}

From Proposition \ref{pro-limttauR}-(i), one has $\tau'(s_0)>0,$
which implies that $\rho'(s_0)>0.$
Hence, by Lemma \ref{lem-convergenceintegral},
$\phi$ extends to $s_0$\,, i.e.,
\[\phi(s_0):=\lim_{s\rightarrow s_0}\phi(s)\]
is well defined. In particular, $\partial\Sigma'=S_{s_0}\times\{\phi(s_0)\}.$

\begin{figure}[htbp]
\includegraphics{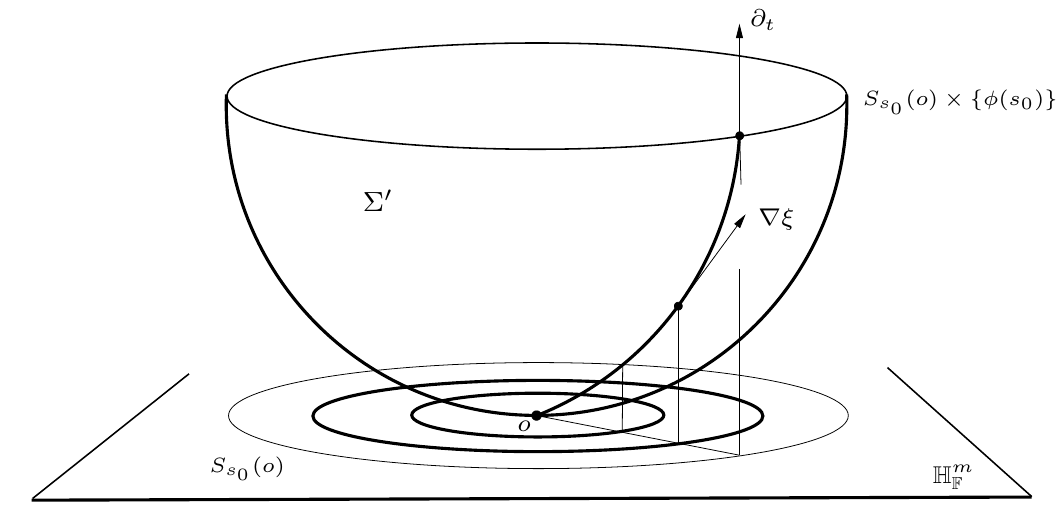}
\caption{ \small The trajectories of  $\nabla\xi$ on  $\Sigma'$ emanate  from
$o$ and meet  $\partial\Sigma'$ orthogonally.}
\label{fig-Hrsphere}
\end{figure}

Notice that $o\in\Sigma'$ is an isolated minimum of the height function
$\xi$ of $\Sigma'.$ Thus, $\Sigma'$ is strictly convex at $o.$ In addition,
by the identities \eqref{eq-principalcurvatures}, at any point of $\Sigma'-\{o\},$
all the principal curvatures are positive. Therefore, $\Sigma'$ is strictly
convex.

As we know,  the angle function $\theta$ of $\Sigma'$ is given by
\begin{equation} \label{eq-theta}
\theta=\frac{1}{\sqrt{1+(\phi')^2}}\,\cdot
\end{equation}
Since $\rho(s_0)=1$, we have  from  \eqref{eq-phi} that
$\phi'(s)\rightarrow +\infty$ as $s\rightarrow s_0$\,. This,
together with \eqref{eq-theta}, implies
that the tangent spaces of $\Sigma'$ along $\partial\Sigma'$
are vertical. Hence,  the trajectories of $\nabla\xi$ all emanate from $o$  and
meet  $\partial\Sigma'$ orthogonally (Fig. \ref{fig-Hrsphere}).
Recall that these trajectories are  geodesics which foliate  $\Sigma'.$

Now, set $\Sigma''$ for the reflection of $\Sigma'$ with respect to
$\hf\times\{\phi(s_0)\}$ and define
\[
\Sigma:={\rm closure}\,(\Sigma')\cup {\rm closure}\,(\Sigma''),
\]
that is, $\Sigma$ is the ``gluing'' of $\Sigma'$ and $\Sigma''$ along
the  $(n-1)$-sphere $S_{s_0}(o)\times\{\phi(s_0)\},$ which is  $C^\infty$-differentiable.
Since the  trajectories of $\nabla\xi$  are also $C^\infty,$ for being geodesics,
the resulting hypersurface $\Sigma$ is $C^\infty$-differentiable
with vertical tangent spaces along $S_{s_0}(o)\times\{\phi(s_0)\}.$
Therefore,  $\Sigma$
is a rotational strictly convex $H_r$-hypersur\-face of \,$\hf\times\R$ which is
homeomorphic to $\s^n$ and is symmetric with respect to $\hf\times\{\phi(s_0)\}.$
This proves (i).

Under the hypotheses in (ii), it follows from Corollary \ref{cor-s0}
that  $\tau$ satisfies:
\[
\tau(0)=0<\tau|_{(0,+\infty)}<1,
\]
so that the $(f_s,\phi)$-graph $\Sigma$
with $\rho$-function $\rho:=\sqrt[r]{\tau|_{[0,+\infty)}}$
is  an entire rotational $H_r$-graph of
$\hfr$ over $\hf\times\{0\}.$ Since $\phi(0)=0$ and
$\phi(s)>0$ for any $s>0,$  $\Sigma$ is contained in the closed half-space
$\hf\times[0,+\infty)$ and is tangent to $\hf\times\{0\}$ at  $o.$ Also, the height function
of $\Sigma$ is unbounded above. Indeed, from  Proposition \ref{pro-limttauR}-(i),
$\tau$, and so $\rho$, is increasing.
Thus, for a fixed $\delta >0,$ and any $s>\delta,$ one has
\[
\phi(s)=\int_{0}^{s}\frac{\rho(u)}{\sqrt{1-\rho^2(u)}}du\ge
\int_{\delta}^{s}\rho(u)du\ge\rho(\delta)(s-\delta),
\]
which implies  that $\phi$ is  unbounded above.
Also, arguing as  for the graph
$\Sigma'$ in the preceding paragraphs, we conclude that $\Sigma$ is strictly convex.

Observe that the  mean curvature vector of $\Sigma$
``points upwards'', that is, its mean convex side $\Lambda$ is the connected component of
$(\hf\times\R)-\Sigma$ which contains the axis $\{o\}\times\R.$ In particular,
$\Lambda$ is foliated by the balls $B_s(o)\times\{\phi(s)\}, \,\, s\in(0,+\infty).$

\begin{figure}[htbp]
\includegraphics{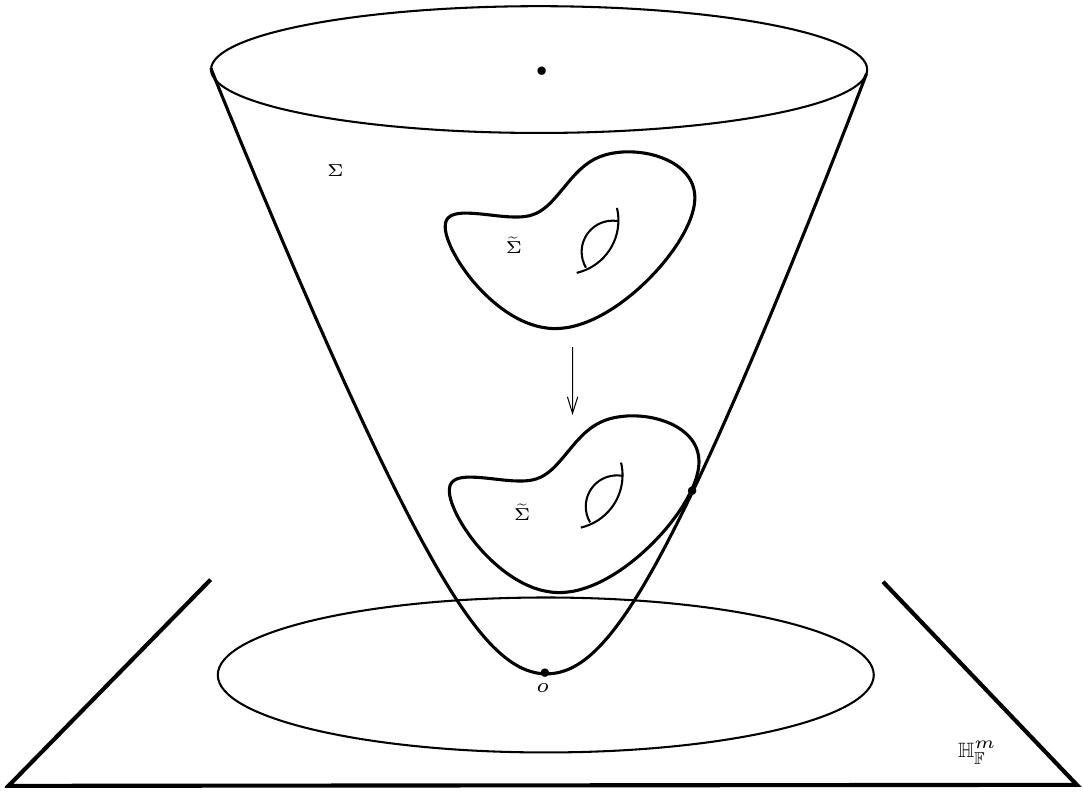}
\caption{\small After a downward translation of $\tilde\Sigma,$ it has a contact with $\Sigma.$}
\label{fig-tangency}
\end{figure}

Let us suppose that
there exists a compact  $H_r$-hypersurface $\widetilde\Sigma$ such that $0<H_r\le C_{\mathbb F}(r).$
Considering the fact that $\Lambda$ is ``horizontally and vertically unbounded'', it is easily seen that, after
a suitable vertical translation, we can assume $\widetilde\Sigma\subset\Lambda$ (Fig. \ref{fig-tangency}).
Now, translate $\widetilde\Sigma$ downward
until it has a first contact with $\Sigma.$ Since $\Sigma$ is strictly convex,
the Maximum-Continuation Principle applies and gives that $\Sigma$ and
$\widetilde\Sigma$ coincide, which is clearly impossible.
This shows that such a $\widetilde\Sigma$ cannot exist and finishes the proof of (ii).
\end{proof}

%\begin{remark}
%The case (ii) of Theorem \ref{th-Hrrotational}  does not occur for $r=n,$
%since $C_{\mathbb F}(n)=0$.
%\end{remark}

\begin{remark}  \label{rem-zerokn}
Let $\Sigma\subset\hfr$ be an entire $H_r$-graph as in Theorem \ref{th-Hrrotational}-(ii).
Since $C_{\mathbb F}(n)=0$, we must have $r<n.$ Also,
the associated function  $\tau:[0,+\infty)\rightarrow\R$ is positive, bounded and increasing, so that
\[
\lim_{s\rightarrow+\infty}\tau'(s)=0,
\]
which implies that $\rho'(s)\rightarrow 0$ as $s\rightarrow+\infty.$  %for $\rho'(s)=\frac{1}{r}(\tau(s))^{(1-r)/r}\tau'(s).$
This, together with \eqref{eq-principalcurvatures}, gives that the principal curvature $k_n=\rho'(s)$ goes to zero as
$s\rightarrow+\infty.$ In particular, \emph{the least principal curvature function of $\Sigma$ is not bounded away from zero}.
\end{remark}

%%%%%%%%%%%%%%%%%%%%%%%%%%%%%%%%%%%%%

Next, we apply the
method of $(f_s,\phi)$-graphs to produce
one-parameter families of $H_r(>0)$-annuli  in $\hfr,$ $1\le r<n.$
For that,  fix $s_0=\lambda>0$ and consider
the solution $\tau(s)$ of $y'=ay+b$ given
in \eqref{eq-generalsolution}, which  satisfies
the initial condition $\tau_0=\tau(\lambda)=1.$
From \eqref{eq-a&bagain},  we have
\begin{equation} \label{eq-tau'lambda}
\tau'(\lambda)=a(\lambda)+b(\lambda)=\frac{r(H_r-|H_r^\lambda|)}{|H_{r-1}^\lambda|}\,\cdot
\end{equation}

Therefore, given $H_r>0,$  $r\in\{1,\dots, n-1\},$ it follows from \eqref{eq-tau'lambda}
that $\tau'(\lambda)<0$ if and only if $H_r<|H_r^\lambda|.$
Thus, considering \eqref{eq-hrinfity}, we can define $\delta_{H_r}$ as the largest positive
constant with the following property:
\begin{equation} \label{eq-deltaHr}
\lambda\in (0,\delta_{H_r})\quad\Leftrightarrow\quad \tau(\lambda)=1 \,\,\, \text{and} \,\,\, \tau'(\lambda)<0.
\end{equation}

Observing that  $\tau|_{[\lambda,+\infty)}$ is a  positive function for any
$\lambda\in (0,\delta_{H_r}),$ we distinguish the cases:
\begin{itemize}[parsep=1ex]
  \item [i)] There exists $\overbar\lambda\in (\lambda,+\infty)$ such that $\tau(\overbar\lambda)=1$ \,(Fig. \ref{fig-graph11}).
  \item [ii)] $\tau(s)<1$ for all $s\in(\lambda,+\infty)$ \,(Fig. \ref{fig-graph21}).
\end{itemize}

\begin{figure} %\label{fig-graphs}
  \begin{subfigure}[hbt]{.5\textwidth}
    \includegraphics{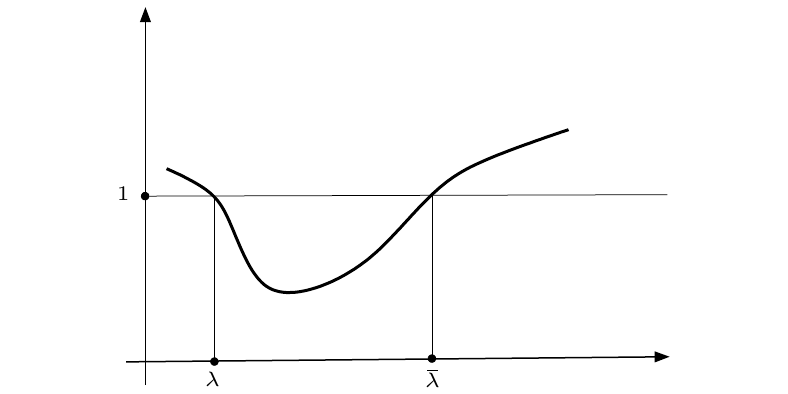}
    \caption{}
    \label{fig-graph11}
  \end{subfigure}
  \hspace{1.5cm}%
  \begin{subfigure}[h]{.5\textwidth}
    \includegraphics{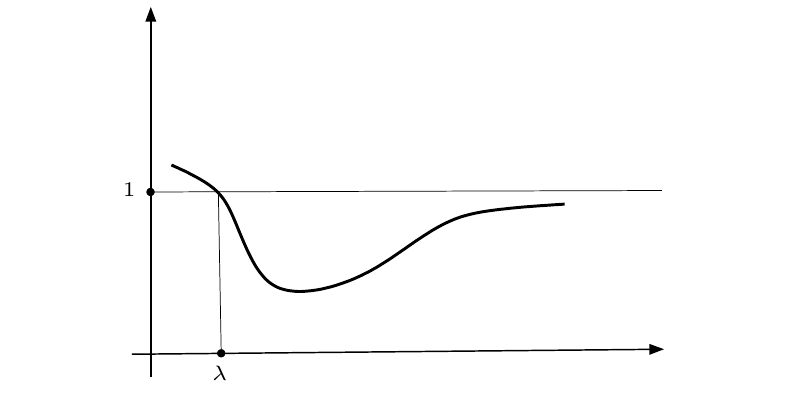}
    \caption{}
    \label{fig-graph21}
  \end{subfigure}
  \caption{\small The two types  of  solutions $\tau$, as in \eqref{eq-generalsolution},  satisfying $\tau(\lambda)=1$.}
\end{figure}

\begin{theorem} \label{th-Hrannuli}
Given $r\in\{1,\dots ,n-1\}$ and $H_r>0,$ let $\delta_{H_r}>0$ be as in
\eqref{eq-deltaHr}.
Then there exists a one-parameter family
\[
\mathscr S\:=\{\Sigma(\lambda)\,;\,\lambda\in(0,\delta_{H_r})\}
\]
of  properly embedded  rotational $H_r$-hypersurfaces
in $\hf\times\R$ which are all homeomorphic to the $n$-annulus  $\s^{n-1}\times\R.$
In addition, the following assertions hold:

\begin{itemize}[parsep=1ex]
\item[\rm i)] If \,$H_r>C_{\mathbb F}(r),$
each $\Sigma(\lambda)\in\mathscr S$ is Delaunay-type, i.e., it
is periodic in the vertical direction, and has unduloids as the trajectories of
the gradient of its  height function.

\item[\rm ii)] If $0<H_r\le C_{\mathbb F}(r)$,
each hypersurface $\Sigma(\lambda)\in\mathscr S$ is symmetric with respect to $\hf\times\{0\}$
and has unbounded height function.
\end{itemize}
\end{theorem}
\begin{proof}
Fix $\lambda\in(0,\delta_{H_r})$ and recall that $\tau$ is decreasing near $\lambda.$
Thus, if $H_r>C_{\mathbb F}(r),$ it follows from   \eqref{eq-limF} that
there exists $\bar\lambda\in (\lambda,+\infty)$ such that
\[
0<\tau|_{(\lambda,\bar\lambda)}< 1=\tau(\bar\lambda)=\tau(\lambda).
\]

Let us observe that a critical point $s_1$ of $\tau$ is necessarily a minimum, since
$\tau''(s_1)=a'(s_1)\tau(s_1)+b'(s_1)>0.$
Therefore,  $\tau$ must have a unique  local minimum at a point  between $\lambda$
and $\overbar\lambda.$ In particular,  $\tau'(\bar\lambda)>0$ (Fig. \ref{fig-graph11}).

\begin{figure}[htbp]
\includegraphics{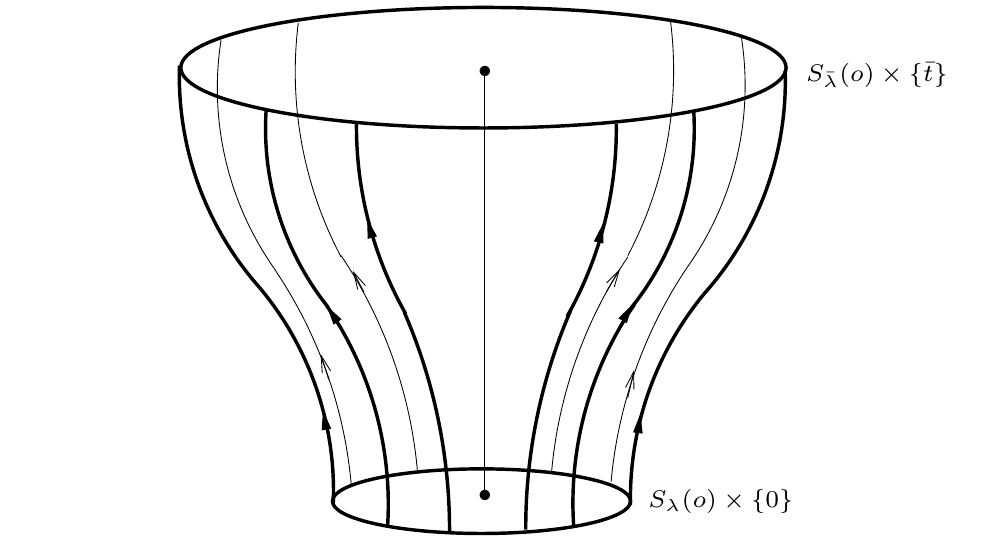}
\caption{\small The ``block''  of a Delaunay-type $H_r$-hypersurface in $\hf\times\R.$}
\label{fig-minimaldelaynay}
\end{figure}

Setting $\tau_\lambda:=\tau|_{(\lambda,\bar\lambda)}$,
it follows from the above considerations and Lemmas \ref{lem-parallel} and  \ref{lem-convergenceintegral}  that
the $(f_s,\phi)$-graph $\Sigma'(\lambda)$   with $\rho$-function
$\rho(s)=\sqrt[r]{\tau_\lambda}$ is a  bounded
$H_r$-hypersurface of $\hf\times\R.$ Moreover, $\Sigma'(\lambda)$
is  homeomorphic to
$\s^{n-1}\times (\lambda,\bar\lambda)$ and  has boundary (see Fig. \ref{fig-minimaldelaynay}):
\[
\partial\Sigma'(\lambda)=(S_\lambda(o)\times \{0\})\cup (S_{\bar\lambda}(o)\times \{\phi(\bar\lambda)\}).
\]
%\[
%\phi(s)=\int_{\lambda}^{s}\frac{\rho(u)}{\sqrt{1-\rho^2(u)}}du, \,\,\, s\in(\lambda,\bar\lambda).
%\]

We also have that the tangent spaces of $\Sigma'(\lambda)$ are vertical
along its boundary $\partial\Sigma'(\lambda)$,
for  $\rho(\lambda)=\rho(\bar\lambda)=1.$
Therefore, we obtain a properly embedded rotational $H_r$-hypersurface
$\Sigma(\lambda)$ from $\Sigma'(\lambda)$ by
continuously reflecting it with respect to the horizontal hyperplanes
$\hf\times\{k\phi(\bar\lambda)\}, \, k\in\Z.$ This proves (i).

Now, let us suppose that $0<H_r\le C_{\mathbb F}(r).$
In this case,  \eqref{eq-limF} gives that
\[
0<\tau|_{(\lambda,+\infty)}<1,
\]
so that the $(f_s,\phi)$-graph $\Sigma'(\lambda)$
determined by  $\rho=\tau|_{(\lambda,+\infty)}^{1/r}$ is an $H_r$-hypersurface
of $\hf\times\R$ with boundary
$\partial\Sigma'(\lambda)=S_\lambda(o)\times\{0\}$ (Fig. \ref{fig-hourglass}).
 By reflecting  $\Sigma'(\lambda)$ with respect to
$\hf\times\{0\},$ as we did before,  we obtain the embedded $H_r$-hypersurface
$\Sigma(\lambda)$ as stated.

It remains to show that the height function of $\Sigma(\lambda)$ is unbounded.
For that, we have just to observe that the infimum of $\tau$ in $[\lambda,+\infty)$
is positive, since $\tau$ itself is positive in this interval, and its limit as
$s\rightarrow+\infty$ is $H_r/C_{\mathbb F}(r)>0.$ So, the same is true for $\rho=\tau^{1/r}.$
Therefore,
\[
\phi(s)=\int_{\lambda}^{s}\frac{\rho(u)}{\sqrt{1-\rho^2(u)}}du>
\int_{\lambda}^{s}\rho(u)du > \inf\rho|_{[\lambda,+\infty)}(s-\lambda),
\]
from which we conclude that $\phi$ is  unbounded.
\end{proof}

\begin{remark}
The case $\mathbb F=\R$ of Theorem \ref{th-Hrrotational}
was previously established in \cite{elbert-earp}, whereas the
case $\mathbb F=\R$ and $r=1$ of Theorem \ref{th-Hrannuli}
was considered in \cite{berard-saearp}.
Nevertheless, the methods employed in these works are different from ours, and
are not applicable to the products $\hfr,$ $\mathbb F\ne\R.$
\end{remark}

\begin{figure}[htbp]
\includegraphics{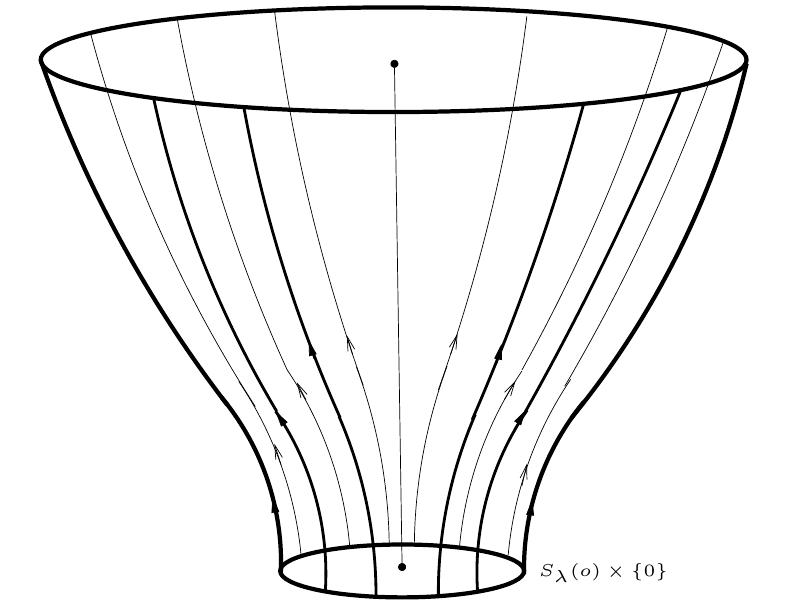}
\caption{\small The $(f_s,\phi)$-graph $\Sigma'(\lambda),$ on which all the
trajectories of $\nabla\xi$ emanate from $\partial\Sigma'(\lambda)$ orthogonally.}
\label{fig-hourglass}
\end{figure}

We proceed now to the classification of the complete
rotational $H_r$-hypersurfaces of $\hfr$ whose height functions
are Morse-type, i.e.,  have isolated critical points (if any).
As we shall see, besides cylinders over geodesic spheres,
these hypersurfaces are precisely the ones we obtained
in Theorems \ref{th-Hrrotational} and \ref{th-Hrannuli}. In particular,
any of them is embedded.
We point out that, in \cite{berard-saearp}, it was shown that, for any $H_1>0,$  there exist
complete  rotational $H_1$-hypersurfaces in $\h^n\times\R$  which are not embedded.
In accordance with our results, the height function of none of these $H_1$-hypersurfaces is Morse-type.

Firstly, let us recall that the $H_r$-hypersurfaces in Theorems \ref{th-Hrrotational} and \ref{th-Hrannuli}
were constructed from a single $(f_s,\phi)$-graph whose  associated  $\tau$-function is a solution of
the ODE $y'=ay+b,$ where $a$ and $b$ are  as in \eqref{eq-a&bagain}. For such a $\tau$,
there is a maximal interval $(s_0,s_1)$, $0\le s_0<s_1\le+\infty,$ such that
$0<\tau|_{(s_0,s_1)}<1.$

Notice that each choice of $H_r$ determines the function $b$ and, so, the equation $y'=ay+b.$
The corresponding graph, then, is determined by the ordering of the constants $H_r$ and $C_{\mathbb F}(r),$
as well as by the values of $s_0$ and $\tau(s_0).$

Below,  we list all the occurrences of  $s_0$ and $\tau(s_0)$ in
Theorems \ref{th-Hrrotational} and \ref{th-Hrannuli} with respect to the
ordering of $H_r$ and $C_{\mathscr F}(r)$:
\begin{itemize}[parsep=1ex]
  \item [C1)] $s_0=0$, $\tau(s_0)=0$,  $H_r>C_{\mathbb F}(r)$.
  \item [C2)] $s_0=0$, $\tau(s_0)=0$,  $H_r\le C_{\mathbb F}(r)$.
  \item [C3)] $s_0>0$, $\tau(s_0)=1$,  $H_r>C_{\mathbb F}(r)$.
  \item [C4)] $s_0>0$, $\tau(s_0)=1$,   $H_r\le C_{\mathbb F}(r)$.
\end{itemize}

The cases C1 and C2 correspond to Theorem \ref{th-Hrrotational}-(i) and Theorem \ref{th-Hrrotational}-(ii), respectively,
whereas C3 and C4 correspond to Theorem \ref{th-Hrannuli}-(i) and Theorem \ref{th-Hrannuli}-(ii). We also remark that
$s_1<+\infty$ in cases C1 and C3, with $\tau(s_1)=1,$ and that $s_1=+\infty$ in cases C2 and C4.

Let $M_0$ be a hypersurface of a Riemannian manifold $M.$ It is easily seen
that $\Sigma:=M_0\times\R$ is a hypersurface of $M\times\R$ whose tangent spaces are all vertical, so
that $\partial_t$ is a principal direction of $\Sigma$
with vanishing principal curvature. In particular,
$H_n=0$ on $\Sigma.$ Also,
for all $r\in\{1,\dots ,n-1\},$ the $r$-th mean curvatures of
$M_0$ and $\Sigma$ at $x\in M_0$ and $(x,t)\in\Sigma$ coincide. In particular,
$M_0$ is an $H_{r(<n)}$-hypersurface of $M$ if and only if $\Sigma$ is an
$H_{r(<n)}$-hypersurface of $M\times\R.$
We call $\Sigma:=M_0\times\R$ the \emph{cylinder} over $M_0.$

\begin{theorem}\label{th-classificationHr>0}
  Let $\Sigma$ be a connected complete  rotational $H_r(>0)$-hyper\-surface of
  $\hfr$ whose height function is Morse-type. Then,
  $\Sigma$ is either a cylinder over a geodesic sphere of \,$\hf$ or
  one of the embedded $H_r$-hypersurfaces of Theorems \ref{th-Hrrotational}--\ref{th-Hrannuli}.
\end{theorem}

\begin{proof}
  Suppose that $\Sigma$ is not a cylinder. In this case, we have that the
  open set $\Sigma_0\subset\Sigma$
  on which $\theta\nabla\xi$ never vanishes is nonempty. Since $\Sigma_0$ contains no vertical points, for a given
  $x_0\in\Sigma_0$\,, there is an open neighborhood  $\Sigma'$ of $x_0$ in $\Sigma_0$ which is a graph
  over an open set $\Omega$ of $\hf$. Thus,
  since $\Sigma$ is rotational and $\Sigma_0$ contains no horizontal points,
  after possibly a reflection with respect
  to a horizontal hyperplane, we can assume that $\Sigma'$ is an $(f_s,\phi)$-graph over $\Omega.$
  (Recall that, in our setting,  the $\phi$-function of an $(f_s,\phi)$-graph is required to be radially increasing.)

   By Lemma \ref{lem-parallel}, the function $\tau=\rho^r$ associated to $\Sigma'$ is a solution
   of $y'=ay+b,$ with $a$ and $b$ as in \eqref{eq-a&bagain}. In addition, since $\Sigma$ is  complete, there
   exists a maximal interval $(s_0,s_1)$, $0\le s_0<s_1\le+\infty,$ such that $0<\tau|_{(s_0,s_1)}<1.$
   In particular, we have the following two possibilities:
   \[
   \tau(s_0)=0 \quad\text{or}\quad  \tau(s_0)=1.
   \]

  Suppose that $\tau(s_0)=0.$ After a vertical translation, we can assume that
  \[
  \phi(s)=\int_{s_0}^{s}\frac{\rho(u)}{\sqrt{1-\rho^2(u)}}du.
  \]
  In particular, $\phi(s_0)=\phi'(s_0)=0.$
  If $s_0>0,$ these equalities imply that the sphere $S_{s_0}(o)\times\{0\}$ of
  $\hf$ is contained in  $\partial\Sigma',$ and that $\nabla\xi$ vanishes
  at all of its points. This, however, contradicts that the height function of
  $\Sigma$ is Morse-type. Hence, $s_0=0,$
  so that the $\tau$-function of $\Sigma'$ satisfies the initial condition $\tau(0)=0.$

  If $H_r>C_{\mathbb F}(r),$ by the uniqueness of solutions of linear ODE's satisfying an initial condition,
  the function $\tau$ such that $\tau(0)=0$ coincides with the one in case C1 above. Thus, the corresponding
  $\phi$-functions also coincide, which clearly implies that $\Sigma'$ is an open set of
  the (strictly convex) $H_r$-sphere obtained in Theorem \ref{th-Hrrotational}-(i).
  Therefore, by the Maximum-Continuation Principle, $\Sigma$ coincides with this $H_r$-sphere.
  If $H_r\le C_{\mathbb F}(r),$ then  $\tau$ coincides
  with the solution of  case C2. Analogously, we conclude that $\Sigma$ is
  an entire graph as in Theorem \ref{th-Hrrotational}-(ii).

  Let us suppose now that $\tau(s_0)=1.$
  Since $0<\tau<1$ in $(s_0,s_1),$ $\tau$ is decreasing
  near $s_0$\,, which implies that $r<n.$ (Indeed, for
  $r=n,$ we have $\tau'=b>0.$) In this case, as we have discussed,
  $|a(s)|\rightarrow+\infty$ as $s\rightarrow 0$ (cf. \eqref{eq-|a|}), and $b(0)$ is $0$ or $H_1$.
  In particular, any solution $\tau$ of $y'=ay+b$ at $s=0$ must satisfy $\tau(0)=0,$
  so that  $s_0\ne 0.$   Hence,   $s_0\in (0,\delta_{H_r}),$ where $\delta_{H_r}$ is
  the positive constant defined in \eqref{eq-deltaHr}.

  Setting $s_0=\lambda$ and observing that any of the hypersurfaces obtained in Theorem \ref{th-Hrannuli}
  is  strictly convex at some of its points, we can argue as in the second from the last paragraph
  and conclude that $\Sigma$ is  the $H_r$-hypersurface $\Sigma(\lambda)$ of
  Theorem  \ref{th-Hrannuli}-(i) or Theorem \ref{th-Hrannuli}-(ii), according to whether
  $H_r>C_{\mathbb F}(r)$ or $H_r\le C_{\mathbb F}(r).$ This finishes the proof.
\end{proof}

\subsection{Rotational $H_r(>0)$-hypersurfaces of $\s^n\times\R$}
In this section, we apply the method of $(f_s,\phi)$-graphs to construct
and classify rotational $H_r(>0)$-hypersurfaces in $\s^n\times\R.$

As we did before, let us fix a
point $o\in\s^n$ and consider a family
\begin{equation}  \label{eq-parallelspheresSn}
\mathscr F:=\{f_s:\s^{n-1}\rightarrow \s^n\,;\, s\in(0,\pi)\}
\end{equation}
of parallel geodesic spheres $f_s$ of $\s^n$ with radius $s$
and center  $o.$ As is well known,
each $f_s$ is totally umbilical, having principal curvatures all equal
to $-\cot s$ with respect to the outward orientation. In particular,
$\mathscr F$ is isoparametric.

From a direct computation, we get that the coefficients
$a$ and $b$ of the ODE $y'=ay+b$ determined by
$\mathscr F$ and any given $H_r>0$ are
\begin{equation}\label{eq-a&bsn}
a(s)=-(n-r)\cot s \quad\text{and}\quad b(s)=b_r\tan^{r-1}(s)\,, \,\,\, b_r=rH_r{{{n-1}\choose{r-1}}}^{-1},
\end{equation}
and that the corresponding general solution is:
%Assume that $s\in (0,\pi/2).$ In this case,
%all principal curvatures of $f_s$ are negative. Thus, for any
%$r\in\{1,\dots , n-1\},$ the $r$-th mean curvature $H_r^s$ of
%$f_s$ never vanishes, so that the coefficients $s\in (0,\pi/2).$ In the particular case
%$r=1,$ $a$ and $b$ are well defined in $(0,\pi).$
\begin{equation} \label{eq-solutionodesn}
\tau(s):= \left(\frac{\sin s_0}{\sin s}\right)^{n-r}\left(\tau_0+
\frac{b_r}{\sin^{n-r}(s_0)}\int_{s_0}^{s}\frac{\sin^{n-1}(u)}{\cos^{r-1}(u)}du\right), \,\,s_0\,, s\in(0,\mathcal R),
\end{equation}
where $\tau_0=\tau(s_0)\in\R$ and
\[
\mathcal R:=\left\{
\begin{array}{ccc}
\pi/2 & \text{if} & r>1\\
\pi  & \text{if} & r=1
\end{array}
\right. .
\]
Also, it is easily checked that
\begin{equation}\label{eq-tausn}
\tau(s):=\left\{
\begin{array}{lcl}
   \frac{b_r}{\sin^{n-r}(s)}\int_{0}^{s}\frac{\sin^{n-1}(u)}{\cos^{r-1}(u)}du& \text{if} & s\in (0,\mathcal R) \\[1.5ex]
  0 & \text{if} & s=0
\end{array}
\right.
\end{equation}
is a well defined solution of $y'=ay+b$ satisfying $y(0)=0.$

Given an integer $n\ge 2,$ it  will be convenient to introduce the following constant:
\begin{equation}\label{eq-sn}
  S(n)=\int_{0}^{\pi/2}\sin^{n-1}(s)ds, \,\, n\ge 2.
\end{equation}

\begin{proposition} \label{pro-limttauRSn}
Let $\tau$ be the solution \eqref{eq-tausn}. Then, the following  hold:
\begin{itemize}[parsep=1ex]
  \item [\rm i)]  $\tau'>0$ in $(0,\mathcal R).$
  \item [\rm ii)] $\displaystyle \lim_{s\rightarrow\mathcal R}\tau(s)=+\infty.$
  \item [\rm iii)] $\displaystyle \lim_{s\rightarrow\frac\pi 2}\tau(s)=H_1S(n)$ if  $r=1.$
\end{itemize}
\end{proposition}

\begin{proof}
  Since the functions $a$ and $b$ in \eqref{eq-a&bsn} are both increasing (when nonconstant), the proof of
  (i) is entirely analogous to the one given in Proposition \ref{pro-limttauR}-(i).

  To prove (ii), let us first assume $r=1.$ In this case, since $\sin \pi=0$ and the integral
  $\int_{0}^{\pi}\sin^{n-1}(u)du$ is positive, we have that $\tau$ satisfies
   (ii) for $\mathcal R=\pi.$

  If $r>1$, for a fixed $\delta\in(0,\pi/2)$ and any $s\in (\delta,\pi/2),$ one has
  \[
  \int_{0}^{s}\frac{\sin^{n-1}(u)}{\cos^{r-1}(u)}du\ge
  \int_{\delta}^{s}\tan^{r-1}(u){\sin^{n-r}(u)}du\ge\sin^{n-r}(\delta)\int_{\delta}^{s}\tan^{r-1}(u)du,
  \]
  which implies that the first integral goes to infinity as $s\rightarrow\pi/2,$ since the same is
  true for the integral $\int_{\delta}^{s}\tan^{r-1}(u)du$. It follows from this fact that
  $\tau(s)\rightarrow+\infty$ as $s\rightarrow\pi/2$ if $r>1,$ which proves (ii).

  The identity in (iii) follows directly from the definitions of $\tau$ (for $r=1$)
  and $S(n)$ (as in \eqref{eq-sn}).
\end{proof}

From the above proposition, we get the following existence result for
$H_r(>0)$-hypersurfaces of $\s^n\times\R.$

\begin{theorem} \label{Hr-rotationalSnxR}
Given $r\in\{1,\dots, n\}$ and a constant $H_r>0,$ there exists an
$H_r$-sphere $\Sigma$ in $\s^n\times\R$ which is symmetric with respect to a horizontal
hyperplane. Furthermore:
\begin{itemize}[parsep=1ex]
  \item [\rm i)]  $\Sigma$ is strictly convex if  $r>1$ or $r=1$ and $H_1 > 1/S(n).$
  \item [\rm ii)] $\Sigma$ is  convex if  $r=1$ and $H_1=1/S(n).$
  \item [\rm iii)] $\Sigma$ is non convex if $r=1$ and $0<H_1 < 1/S(n).$
\end{itemize}
\end{theorem}

\begin{proof}
Let $\mathscr F$ be an arbitrary  family of parallel geodesic spheres of
$\s^n$ as given in \eqref{eq-parallelspheresSn}. Consider the functions
$a$ and $b$ defined in \eqref{eq-a&bsn} and let
$\tau$ be the solution \eqref{eq-tausn} of the ODE $y'=ay+b.$

From Proposition \ref{pro-limttauRSn}-(ii),
there exists a positive $s_0<\mathcal R$ such that
\[
0=\tau(0)<\tau|_{(0,s_0)}<1=\tau(s_0),
\]
so that $\tau|_{[0,s_0)}$ determines an $(f_s,\phi)$-graph
$\Sigma'$ over $B_{s_0}(0)\subset\s^n.$
Since $\tau(s_0)=1$  and $\tau'(s_0)>0$ (by Proposition \ref{pro-limttauRSn}-(i)),
we can proceed just as in the proof of Theorem \ref{th-Hrrotational}-(i) to
obtain from $\Sigma'$ the embedded
$H_r$-sphere $\Sigma$ of $\s^n\times\R$  which is symmetric with respect to
$P_{\phi(s_0)}:=\s^n\times\{\phi(s_0)\}.$

If  $r>1$ or $r=1$ and $H_1 > 1/S(n),$ we have from  Proposition \ref{pro-limttauRSn}, items (ii) and (iii),
that $0<s_0<\pi/2.$ Hence,  for $s\in(0, s_0),$ all spheres $f_s$ have negative
principal curvatures, which,  together with equalities
\eqref{eq-principalcurvatures}, gives that $\Sigma$ is strictly convex. This proves (i).

If $r=1$ and $H_1=1/S(n),$ Proposition \ref{pro-limttauRSn}-(iii) yields
$s_0=\pi/2.$ However, $f_{\pi/2}$ is totally geodesic in $\s^n,$ which implies that,
except for $k_n=H_1>0,$  the principal curvatures of $\Sigma$ vanish at all points of the horizontal
section $\Sigma_{\phi(\pi/2)}=\Sigma\cap P_{\phi(\pi/2)}.$
Therefore, $\Sigma$ is  convex on $\Sigma_{\phi(\pi/2)}$ and
strictly convex on $\Sigma -\Sigma_{\phi(\pi/2)}$\,.

Finally, assuming $r=1$ and $0<H_1<1/S(n),$ we have from
Proposition \ref{pro-limttauRSn}-(iii) that $s_0>\pi/2.$
Observing that, for  $\pi/2<s<\pi,$ $f_s$ has positive principal curvatures,
we conclude, as in the last paragraph,  that
$\Sigma$ is strictly convex (resp. convex, non convex) on
$\Sigma_{\phi(s)}$ if $s<\pi/2$ (resp. $s=\pi/2,$ $s>\pi/2$).
In particular, $\Sigma$  is non convex. This shows (iii) and concludes our
proof.
\end{proof}

\begin{remark}
Except for the assumptions on the
convexity of $\Sigma,$ the case $r=1$ of Theorem \ref{Hr-rotationalSnxR} was
proved in \cite{pedrosa}. The case $r=n=2$ was considered in \cite{cheng-rosenberg}.
It should also be mentioned that, for $n=2$ and $r=1,$ the non convexity of $\Sigma$ as stated in (iii)
was pointed out in \cite[Remark 2.8]{abresch-rosenberg}.
\end{remark}

In our next theorem we show the existence of one-parameter families
of rotational Delaunay-type $H_r(>0)$-annuli in $\s^n\times\R.$ This result, then,
generalizes the analogous one obtained in \cite{pedrosa-ritore} for $r=1.$

First, let us introduce the  constant
\[
C_r:=\frac{n-r}{n}{{n}\choose{r}}
\]
and observe that, for $1\le r<n,$  the positive constants $H_r$\,, $b_r$ (as in \eqref{eq-a&bsn}),
and $C_r$ satisfy the following relation:
\begin{equation} \label{eq-relation1}
\frac{n-r}{b_r}=\frac{C_r}{H_r}\,\cdot
\end{equation}

In this setting, if we define
\begin{equation} \label{eq-sr1}
\delta_{H_r}:={\rm arctan}\,(C_r/H_r)^{1/r}\,\in (0,\pi/2),
\end{equation}
then a solution $\tau$ of $y'=ay+b$ such that
$\tau(s_0)=1$, $s_0\in(0,\pi/2),$  satisfies:
\begin{equation} \label{eq-sr2}
\tau'(s_0)<0 \quad\Leftrightarrow\quad 0<s_0<\delta_{H_r}.
\end{equation}

\begin{theorem}  \label{th-delaunaytypesn}
Given $n\ge 2,$ $r\in\{1,\dots ,n-1\}$, and $H_r > 0,$
there exists a one-parameter family \,$\mathscr S=\{\Sigma(\lambda)\,;\,0<\lambda<\delta_{H_r}\}$
of  properly embedded Delaunay-type rotational $H_r$-hypersurfaces
in $\s^n\times\R$. % that is, each $\Sigma(\lambda)$ is  periodic, homeomorphic to $\s^{n-1}\times\R$,
%and has unduloids as the trajectories of the gradient of its  height function.
\end{theorem}

\begin{proof}
  Given $\lambda\in (0,\delta_{H_r}),$
  consider the solution $\tau$
  as in \eqref{eq-solutionodesn} such that $s_0=\lambda$ and $\tau_0=\tau(\lambda)=1.$
  From \eqref{eq-sr2}, we have  that $\tau$ is decreasing
  in a neighborhood of $\lambda.$

  Observe that $\tau$ is positive in $(0,\mathcal R).$ Also, setting
  \[
  \mu(s)=\left(\frac{\sin s}{\sin\lambda}\right)^{n-r}\,,
  \]
  we have that $\mu>1$ on $(\lambda,\pi/2).$ So, for $r>1$  and $s>\lambda,$
  \[
  \tau(s)>\frac{1}{\mu(s)}\left(1+\int_{\lambda}^{s}b(u)du\right)=\frac{1}{\mu(s)}\left(1+b_r\int_{\lambda}^{s}\tan^{r-1}(u)du\right),
  \]
  which implies that  $\tau(s)\rightarrow +\infty$ as $s\rightarrow\pi/2.$

  If $r=1,$ since $\sin\pi=0$ and $\int_{\lambda}^{\pi}\mu(s)ds$ is positive, we have that
  $\tau(s)\rightarrow+\infty$ as $s\rightarrow\pi.$

  It follows from the above considerations that  there exists
  $\overbar\lambda\in (0,\mathcal R)$ such that
\begin{equation} \label{eq-tau'delaunay}
\tau(\lambda)=\tau(\bar\lambda)=1 \,\,\, \text{and} \,\,\, \tau'(\lambda)<0<\tau'(\bar\lambda).
\end{equation}
From this point on, the proof is entirely analogous to that of Theorem \ref{th-Hrannuli}-(i).
\end{proof}

A classification result for rotational $H_r$-hypersurfaces of $\s^n\times\R$ can be achieved
in the same way we did for their congeners in $\hfr.$ To see this, assume
that $\Sigma$ is a complete connected
rotational $H_r(>0)$-hypersurface of $\s^n\times\R$
whose height function is Morse-type. Assuming that $\Sigma$ is non cylindrical,
we have, as before, that there exists an open set $\Sigma'\subset\Sigma$  which is an
$(f_s,\phi)$-graph, $f_s\in\mathscr F.$
The corresponding $\tau$-function, restricted to a maximal interval
$(s_0,s_1),$ satisfies:
\[
0<\tau|_{(s_0\,,s_1)}<1, \,\,\, 0\le s_0<s_1\le\mathcal R,
\]
which yields $\tau(s_0)=0$ or $\tau(s_0)=1.$

If $\tau(s_0)=0,$ then $s_0=0.$ (Otherwise, the height function of $\Sigma$ would not
be Morse-type.) In this case, $\tau$ coincides with the $\tau$-function of the $H_r$-sphere
of Theorem \ref{Hr-rotationalSnxR}, and then $\Sigma$ itself coincides with this sphere.
(Notice that any of the spheres obtained in Theorem \ref{Hr-rotationalSnxR} is strictly
convex on an open set.)

If $\tau(s_0)=1,$ then $\tau$ is decreasing in a neighborhood of $s_0$\,. Thus,
$r<n.$ In particular, $|a(s)|\rightarrow +\infty$ as $s\rightarrow 0,$ so that
$s_0\in(0,\delta_{H_r}).$ Analogously, this gives that $\Sigma$ coincides with
the $H_r$-annulus $\Sigma(\lambda)$ of Theorem \ref{th-delaunaytypesn}, $\lambda=s_0.$

Summarizing, we have the following result.

\begin{theorem}\label{th-classificationHr>0SnxR}
  Let $\Sigma$ be a  connected complete  rotational $H_r(>0)$-hyper\-surface of
  $\s^n\times\R$ whose height function is Morse-type. Then, $\Sigma$ is either
  a cylinder over a strictly convex geodesic sphere of \,$\s^n$ or one of the embedded $H_r$-hypersurfaces of
  Theorems \ref{Hr-rotationalSnxR}--\ref{th-delaunaytypesn}.
\end{theorem}

\begin{remark}
Regarding the hypothesis on the height function of $\Sigma$ in Theorem \ref{th-classificationHr>0SnxR},
we point out that  a rotational embedded
$H_1(>0)$-torus in $\s^n\times\R$ whose height
function is non Morse-type was obtained in \cite{pedrosa}.
\end{remark}

\section{Rotational $r$-minimal Hypersurfaces of $\hfr$ and $\s^n\times\R.$} \label{sec-rotationalr-minimal}

In this section, we shall see that the method of $(f_s,\phi)$-graphs can be used
for construction  and classification of rotational $r$-minimal hypersurfaces of \,$\hfr$ and $\s^n\times\R.$
A major distinction from the case of $H_r(>0)$-hypersurfaces  is that the Maximum-Continuation
Principle is no longer available. % for establishing  uniqueness results.

\begin{theorem} \label{th-r-minimalrotationalHxR}
Given $r\in\{1,\dots ,n\},$ there exists a one-parameter family
\[\mathscr S=\{\Sigma(\lambda)\,;\, \lambda >0\}\] of
complete rotational $r$-minimal $n$-annuli in \,$\hfr$ with the following
properties:
\begin{itemize}[parsep=1ex]
  \item [\rm i)] If $r=n,$  $\Sigma(\lambda)$ is a cylinder over a geodesic sphere
  of \,$\hf$ of radius $\lambda>0.$
  \item [\rm ii)] If $r<n$, $\Sigma(\lambda)$ is catenoid-type. More precisely,  it is symmetric
  with respect to $P_0=\hf\times\{0\},$ and $P_0\cap\Sigma(\lambda)$ is the geodesic sphere
  of \,$\hf$ of radius $\lambda$  centered at the point $o\in\hf$  of the axis.
  In addition, each of the parts of $\Sigma(\lambda)$ above and below $P_0$ is a rotational graph
  over $\hf-B_\lambda(o)$ (Fig. \ref{fig-cat}).
\end{itemize}
Furthermore, up to ambient isometries, any complete connected
rotational $r$-minimal hypersurface of
\,$\hfr$ is either a member of $\mathscr S$ or a horizontal hyperplane.
\end{theorem}

\begin{proof}
Given $\lambda>0,$ is immediate that a cylinder over a geodesic sphere of $\hf$ of radius $\lambda$ is
an $n$-minimal rotational annulus of $\hfr,$ which yields (i).

Assume that $1\le r<n$ and let
$\mathscr F=\{f_s\,;\, s\in (0,+\infty)\}$ be the parallel  family of
geodesic spheres of $\hf$ centered at the axis point $o\in\hf.$
The ODE determined by $\mathscr F$ and $H_r=0$ is given by
\begin{equation}  \label{eq-odeminimal}
y'=ay, \quad a(s)=-\frac{r|H_r^s|}{|H_r^{s-1}|}\,, \,\,\, s\in (0,+\infty).
\end{equation}

Since $a<0,$ given $\lambda>0,$ the function
\[
\tau_\lambda(s)=\exp\left(\int_{\lambda}^{s}a(u)du\right), \,\,\, s\in [\lambda,+\infty),
\]
is clearly a solution of \eqref{eq-odeminimal}  which satisfies
\[
0<\tau_\lambda(s)\le\tau_\lambda(\lambda)=1 \,\,\, \forall s\in [\lambda,+\infty).
\]
In addition, %for all $s\in[\lambda,+\infty),$  $\tau_\lambda'(s)=a(s)\tau_\lambda(s)<0.$
$\tau_\lambda'(\lambda)=a(\lambda)<0.$ So, setting
$\rho_\lambda=\tau_\lambda^{1/r},$ it follows from Lemma \ref{lem-convergenceintegral} that
%the function $\phi_\lambda:[\lambda,+\infty)\rightarrow\R$ given by
\[
\phi_\lambda(s):=\int_{\lambda}^{s}\frac{\rho_\lambda(u)}{\sqrt{1-\rho_\lambda^2(u)}}du, \,\,\, s\in [\lambda,+\infty),
\]
is well defined. Therefore, by Lemma \ref{lem-parallel}, the
$(f_s,\phi_\lambda)$-graph $\Sigma'(\lambda)$  is an $r$-minimal hypersurface
of $\hf\times\R.$ Notice that $\Sigma'(\lambda)$ is a graph over $\hf-\overbar{B_\lambda(o)}$
with boundary $\partial\Sigma'(\lambda)=S_\lambda(o)$  (Fig. \ref{fig-cat}).
\begin{figure}[htbp]
\includegraphics{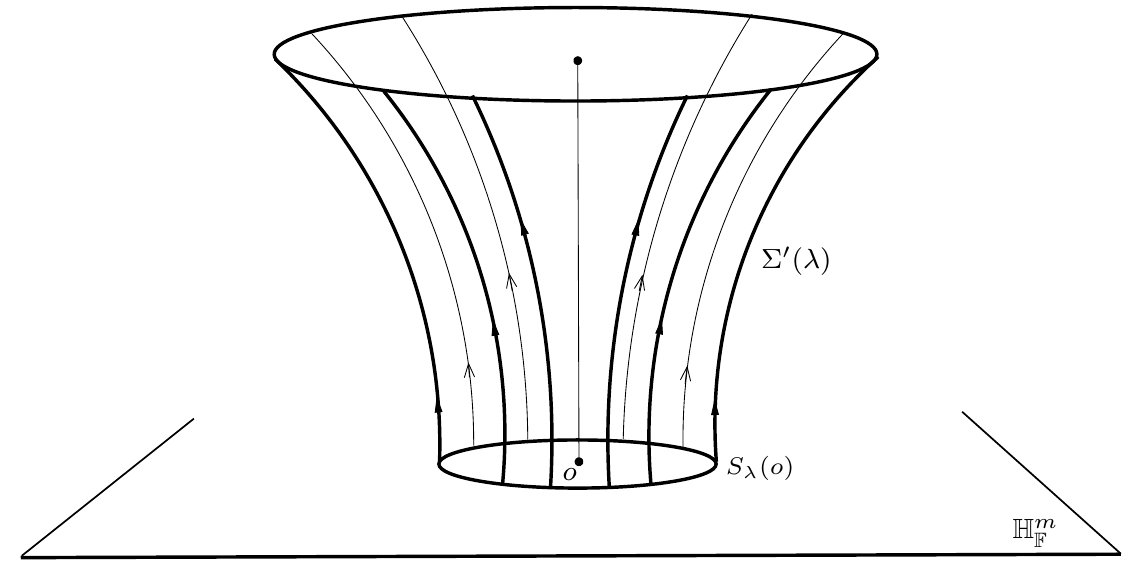}
\caption{\small The half $r$-minimal catenoid $\Sigma'(\lambda),$ on which all the
trajectories of $\nabla\xi$ emanate from $\partial\Sigma'(\lambda)$ orthogonally.}
\label{fig-cat}
\end{figure}
Also, since  $\rho_\lambda(\lambda)=1,$
the tangent spaces of $\Sigma'(\lambda)$ along $\partial\Sigma'(\lambda)$ are all vertical. Thus, considering
the reflection $\Sigma''(\lambda)$ of $\Sigma'(\lambda)$ with respect
to $\hf\times\{0\},$ as before, we have that
$\Sigma(\lambda):={\rm closure}\,(\Sigma'(\lambda))\cup {\rm closure}\,(\Sigma''(\lambda))$
is the desired $r$-minimal hypersurface.

Suppose now that $\Sigma$ is a complete connected rotational
$r$-minimal hypersurface of $\hfr$, $r\in\{1,\dots, n\},$ and set
\[
\Sigma_0:=\{x\in\Sigma\,;\, \theta(x)\nabla\xi(x)\ne 0\}.
\]
Notice that $\Sigma$ is either a horizontal hyperplane or a cylinder if and only
if $\Sigma_0=\emptyset.$ So, we can assume $\Sigma_0\ne\emptyset.$
We can also assume, without loss of generality, that $\Sigma$ and all members of
$\mathscr S$ share the same axis $\{o\}\times\R.$

As we argued in previous proofs, under the above hypotheses, there exists
an $(f_s,\phi)$-graph $\Sigma'\subset\Sigma_0$ and a maximal open interval $(s_0,s_1),$
$0\le s_0<s_1\le+\infty,$ such that the $\tau$-function of $\Sigma'$ satisfies
$0<\tau|_{(s_0,s_1)}<1.$

For $r=n,$ we have that $\tau,$ and so $\rho$, is constant. Hence,
up to a vertical translation,  one has $\phi(s)=cs, \,\, s>0,$ for some constant $c>0.$
However, $\phi'(0)=c>0,$ which implies that the closure of $\Sigma'$ in $\Sigma$
meets the rotation axis non orthogonally, i.e., $\Sigma$ is not smooth at $\partial\Sigma'$ --- a contradiction.
So, $\Sigma_0=\emptyset$ if $r=n.$

For $r<n,$ we have that
$\tau$ is a solution of \eqref{eq-odeminimal}.  In particular, $\tau$ is decreasing,
which implies that  $\tau(s_0)=1.$ As before, this yields
$s_0>0.$ Thus, setting $s_0=\lambda,$ we have $\tau=\tau_\lambda$, which implies that,
up to a vertical translation, $\phi=\phi_\lambda$ and, then, $\Sigma'$ coincides with the
half-catenoid $\Sigma'(\lambda).$

We conclude from the above that  $\Sigma_0$ is the union of open
half-catenoids $\Sigma'(\lambda)$, where $\Sigma(\lambda)\in\mathscr S.$
In addition, $\Sigma_1:=\Sigma-\Sigma_0$ has empty interior in $\Sigma.$ Otherwise,
there would exist a nonempty  open set $U\subset\Sigma_1$ of $\Sigma$ which would be either
horizontal or vertical, and whose boundary in $\Sigma$ would be contained in some
half-catenoid. Since catenoids have no horizontal points, $U$ should be vertical and, so, part
of a vertical rotational cylinder. However, rotational cylinders in $\hfr$ are $r$-minimal if and only if
$r=n.$ Therefore, $\Sigma_1$ has  empty interior, which implies that $\Sigma_0$ is open and dense in $\Sigma.$
Clearly, the intersection of two distinct members of $\mathscr S$ is always transversal. This, together with
the connectedness of $\Sigma$ and the density of $\Sigma_0$\, in $\Sigma,$
gives that  $\Sigma$ coincides with some $\Sigma(\lambda)\in\mathscr S$, which concludes
our proof.
\end{proof}

\begin{theorem}  \label{th-rminimaldelaunaytypesn}
Given $r\in\{1,\dots ,n\},$ there exists a one-parameter family
\[\mathscr S=\{\Sigma(\lambda)\,;\,0<\lambda<\mathcal R_0\}\] of
complete rotational $r$-minimal $n$-annuli in \,$\s^n\times\R$ with the following
properties:
\begin{itemize}[parsep=1ex]
  \item [\rm i)] If $r=n,$ then $\mathcal R_0=\pi$ and  $\Sigma(\lambda)$ is a cylinder over a geodesic sphere
  of \,$\s^n$ of radius $\lambda.$
  \item [\rm ii)] If $r<n$, then  $\mathcal R_0=\pi/2$ and $\Sigma(\lambda)$ is  Delaunay-type.
\end{itemize}
Furthermore, up to ambient isometries, any complete connected
rotational $r$-minimal hypersurface of
\,$\s^n\times\R$ is either a member of $\mathscr S$ or a horizontal hyperplane.
\end{theorem}

\begin{proof}
Statement (i) is trivial. So, assume $r<n$ and
let $\mathscr F=\{f_s\,;\, s\in(0,\pi)\}$ be the family of parallel geodesic spheres
of $\s^n$ centered at some point $o\in\s^n.$   In this setting, the ODE determined by
$\mathscr F$ and $H_r=0$ is
\begin{equation}  \label{eq-odeminimalsn}
y'=ay, \quad a(s)=-(n-r)\cot s, \quad s\in(0,\pi).
\end{equation}

Given $\lambda\in(0,\pi/2),$ the function
\[
\tau_\lambda(s)=\left(\frac{\sin \lambda}{\sin s}\right)^{n-r}\,, \,\, s\in (0,\pi),
\]
is easily seen to be the solution of \eqref{eq-odeminimalsn} satisfying
\[
0<\tau_\lambda|_{(\lambda,\pi-\lambda)}<1=\tau(\lambda)=\tau(\pi-\lambda).
\]

Henceforth, the reasoning in the proof of Theorem \ref{th-delaunaytypesn} applies and
leads to the construction of the Delaunay-type $r$-minimal hypersurface $\Sigma(\lambda)$ as stated in (ii).

Now, suppose that $\Sigma$ is a complete connected
rotational $r$-minimal hypersurface of $\s^n\times\R.$
In this setting, define
\[
\Sigma_0:=\{x\in\Sigma\,;\, \theta(x)\nabla\xi(x)\ne 0\}.
\]

As in the preceding proof, $\Sigma_0=\emptyset$ if $r=n.$ Thus, in this case,
$\Sigma$ is either a horizontal hyperplane or a cylinder over a geodesic sphere of
$\s^n.$

Suppose that $r<n$ and that the axis of $\Sigma$ is $\{o\}\times\R.$ If $\Sigma_0\ne\emptyset,$
then $\Sigma$ is neither a horizontal hyperplane nor a cylinder. In addition,
there exists an $(f_s,\phi)$-graph $\Sigma'\subset\Sigma_0$
and a maximal interval $(s_0,s_1),$ $0\le s_0<s_1\le\pi$,
such that the $\tau$-function of $\Sigma'$ satisfies
$0<\tau|_{(s_0,s_1)}<1$. So, $\tau(s_0)=0$ or $\tau(s_0)=1.$

The formula of the general solution of the ODE \eqref{eq-odeminimalsn} gives that
$\tau$ is positive, not defined at $s=0,\pi$, and bounded away from zero. In particular,
$s_0\ne 0, s_1\ne \pi$ and $\tau(s_0)=\tau(s_1)=1,$ so that $\tau$ is given by
\[
\tau(s)=\left(\frac{\sin s_0}{\sin s}\right)^{n-r}\,, \,\, s_0\,, s\in (0,\pi).
\]

It is clear from this last equality and the considerations preceding it
that $s_0<\pi/2<s_1<\pi,$ which implies that  $\tau$ coincides with $\tau_\lambda$\,, $\lambda=s_0$\,.
Therefore, $\Sigma'$ coincides with the ``block'' $\Sigma'(\lambda)$ that generates $\Sigma(\lambda)$
(see Fig. \ref{fig-minimaldelaynay}), so that  $\Sigma_0$ is a union of open sets of
members of $\mathscr S.$

Let $U\subset\Sigma-\Sigma_0$ be an open set of $\Sigma.$ If $U$ is
nonempty, it cannot be horizontal, for no member of $\mathscr S$ has horizontal points.
If $U$ is vertical, then it is part of the totally geodesic cylinder $S_{\pi/2}\times\R.$
In this case, a boundary point of $U$ is vertical and lies on a geodesic sphere centered
at the axis and of radius $\pi/2.$ However, such a boundary point also lies on some
$\Sigma(\lambda)\in\mathscr S,$ which contradicts the fact that any vertical point
of $\Sigma(\lambda)$ lies on a geodesic sphere of radius different from $\pi/2.$ (In fact, these
vertical points are on  geodesic spheres of radiuses $\lambda<\pi/2$ and $\pi-\lambda>\pi/2.$)

We conclude from the above that $\Sigma_0$ is open and dense in $\Sigma.$
Since $\Sigma$ is connected and two
distinct members of $\mathscr S$ are never tangent, it follows that, for some
$\lambda\in (0,\pi/2),$  $\Sigma$ coincides with $\Sigma(\lambda)\in\mathscr S.$
\end{proof}

\begin{remark}
The case $\mathbb F=\R$ of Theorem \ref{th-r-minimalrotationalHxR} was considered in \cite{elbert-nelli-santos},
whereas Theorem \ref{th-rminimaldelaunaytypesn} was proved in \cite{pedrosa-ritore} for $r=1.$
Again, the methods employed in these works is different from ours and  cannot be applied
to  general products $M\times\R$,  since they all rely  on the Euclidean
and Lorentzian geometries of the underlying spaces of $\s^n\times\R$ and $\h^n\times\R.$
\end{remark}

\section{Translational $H_r(>0)$-hypersurfaces of $\hfr.$}  \label{sec-nonrotationalHrhyp}

Given a Hadamard manifold $M,$ recall that
the \emph{Busemann function} $\mathfrak b_\gamma$ of $M$ corresponding to an arclength geodesic
$\gamma\colon(-\infty,+\infty)\rightarrow M$ is defined as
\[
\mathfrak b_\gamma(p):=\lim_{s\rightarrow +\infty}({\rm dist}_M(p,\gamma(s))-s), \,\,\, p\in M.
\]

The level sets $\mathscr{H}_s:=\mathfrak b_\gamma^{-1}(s)$ of a Busemann function
$\mathfrak b_\gamma$ are called \emph{horospheres} of $M.$
In this setting, as is well  known, $\{\mathscr H_s\,;\, s\in(-\infty, +\infty)\}$ is a parallel family
which foliates  $M.$ Furthermore, any horosphere $\mathscr
H_s$ is homeomorphic to $\R^{n-1}$,  and
any geodesic of $M$ which is asymptotic to $\gamma$ --- i.e., with  the same point $p_\infty$
on the asymptotic boundary $M(\infty)$ of $M$ ---  is orthogonal to each horosphere $\mathscr H_s$\,.
In this case, we say that the horospheres $\mathscr H_s$ are \emph{centered} at $p_\infty$\,.

\begin{figure}[htbp]
\includegraphics{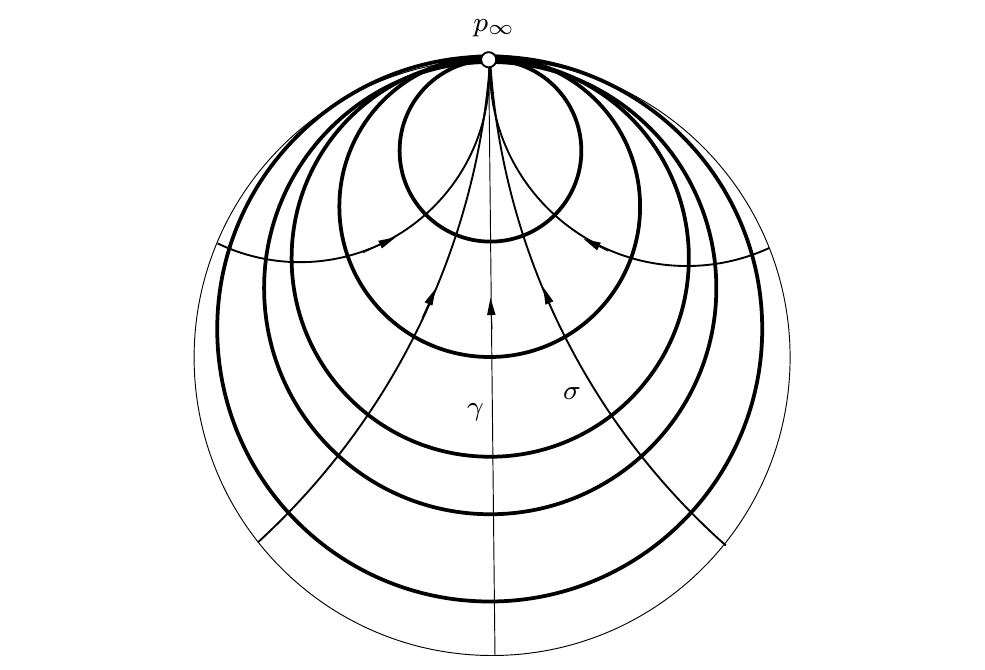}
\caption{A ``ball model'' for a Hadamard manifold.}
\label{fig-hadamardmanifold}
\end{figure}

%\begin{figure}
%\includegraphics{poincareball.pdf}
%\caption{A ``ball model'' for a Hadamard manifold.}
%\label{fig-hadamardmanifold}
%\end{figure}

Therefore, in what concerns  its horospheres, a Hadamard manifold
can be pictured just as the Poincaré ball model of  hyperbolic space $\h^n,$ where
the horospheres centered at a point $p_\infty\in\h^n(\infty)$ are the
Euclidean $(n-1)$-spheres in $\h^n$ which are tangent to $\h^n(\infty)$ at
$p_\infty$ (Fig. \ref{fig-hadamardmanifold}).

In the real hyperbolic space $\h^n,$ any horosphere is totally umbilical
with constant principal curvatures equal to $1.$
As shown  in \cite[Proposition-(vi), pg. 88]{berndtetal},
any horosphere of $\h_\mathbb{F}^m,$  $\mathbb F\ne\R,$
has principal curvatures $1$ and $1/2$ with
multiplicities $1$ and $n-2,$ respectively.
Therefore, \emph{any family $\mathscr F$ of parallel horospheres of
\,$\h_\mathbb{F}^m$ is isoparametric and its members are pairwise congruent.
In addition, for any integer $r\in\{1,\dots, n-1\},$
all horospheres of  \,$\h_\mathbb{F}^m$ have the same (positive)
$r$-th mean curvature, which we denote by $H_r^0.$}

%In this section, we consider Hadamard manifolds $M$ which admit
%special foliations by isoparametric hypersurfaces.
%For such an  $M,$ we provide examples of non rotational
%properly embedded $H_r(\ge 0)$-hypersurfaces in $M\times\R.$

%Hence, $\mathscr F$ is parallel, in the sense of the previous section, with each
%$\mathscr H_s$ having normal field
%\[
%\eta_s(p):=\gamma_p'(s), \,\,\,\, p\in\mathscr H_s\,,
%\]
%where $\gamma_p$ is the geodesic through $p$ such that
%$\gamma(\infty)=p_\infty.$

\begin{theorem}  \label{th-Hr-horospheretype}
Let $\mathscr F:=\{\mathscr H_s\,;\, s\in(-\infty, +\infty)\}$  be a family
of parallel horospheres in  hyperbolic space $\hf.$
Then, for any  even integer \,$r\in\{2,\dots ,n-1\},$ and any
constant $H_r\in(0,H_r^0),$
there exists  a  properly embedded,
everywhere non convex $H_r$-hypersurface $\Sigma$ in $\hfr$  which is homeomorphic to
\,$\R^n.$ Furthermore, $\Sigma$ is foliated by horospheres,
is symmetric with respect to the horizontal hyperplane $\hf\times\{0\}$,
and its height function is unbounded above and below.
\end{theorem}

\begin{proof}
%Let $p_\infty$ be the center of the horospheres $\mathscr H_s$\,.
For each $s\in (-\infty,\infty)$, consider the isometric immersion
$f_s:\R^{n-1}\rightarrow\hf$  such that $f_s(\R^{n-1})=\mathscr H_s$\,.
Since all the principal curvatures of $f_s$ are constant and independent of $s,$
the coefficients  $a$ e $b$ of the ODE $y'=ay+b$ associated to this family
are constants. Also, since
$r$ is even and $0<H_r<H_r^0$\,, we have
\[
b<0<a \quad\text{and}\quad 0<-\frac{b}{a}<1.
\]

In this setting,   consider the   solution $\tau:(-\infty, 0]\rightarrow\R$ of $y'=ay+b$:
\begin{equation} \label{eq-tauspecific}
\tau(s)=e^{a(s-s_0)}-\frac{b}{a}\,,  \,\, \, s_0=\log(1+b/a)^{-1/a},
\end{equation}
and observe that it satisfies:
\begin{equation}  \label{eq-boudfortau}
0<-\frac{b}{a}<\tau(s)\le 1=\tau(0) \,\,\, \forall s\in (-\infty, 0].
\end{equation}

\begin{figure}[htbp]
\includegraphics{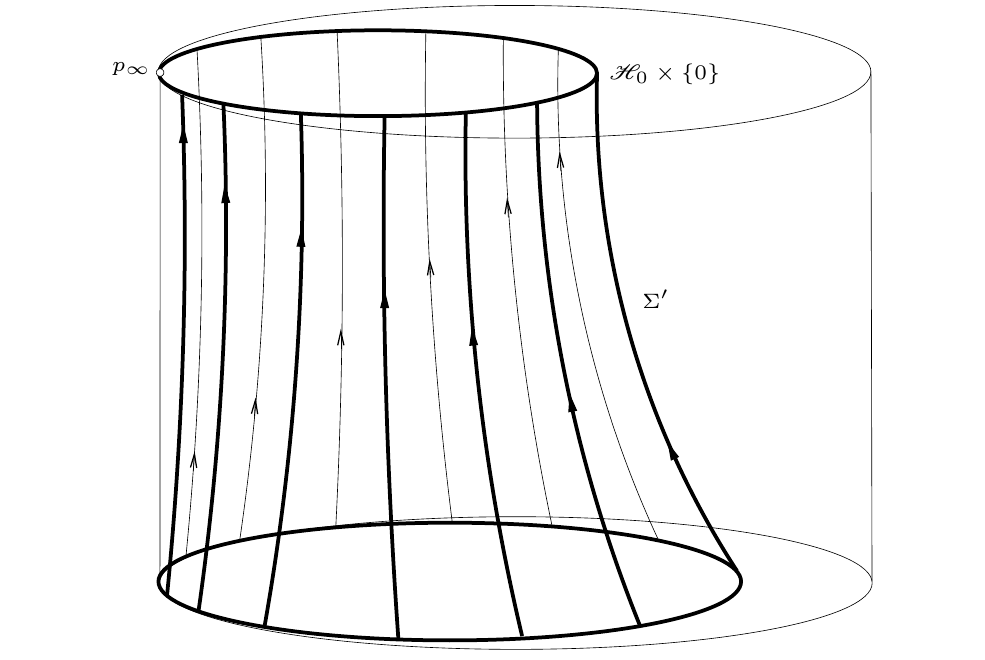}
\caption{\small A piece of the graph $\Sigma'$, on which
all the trajectories of  $\nabla\xi$
meet $\mathscr H_{0}\times\{0\}$ orthogonally.}
\label{fig-Horograph}
\end{figure}

By Lemma \ref{lem-parallel},
the $(f_s,\phi)$-graph $\Sigma'$
with $\rho=\sqrt[r]{\tau}$
is an $H_r$-hypersurface of $\hf\times\R.$ The function $\phi,$ in this case, is given by
\[
\phi(s):=-\int_{s}^{0}\frac{\rho(u)}{\sqrt{1-\rho^2(u)}}du, \,\,\, s\in(-\infty, 0).
\]
Notice that, by \eqref{eq-boudfortau}, one has
\[
\tau'(s)=a\tau(s)+b>a\frac{-b}{\phantom{-}a}+b=0,
\]
so that $\tau'(0)>0.$ Thus, by Lemma \ref{lem-convergenceintegral},
$\phi$ is well defined. Also, $\phi$ is negative on $(-\infty, 0)$ and is unbounded.
Indeed, for all  $s\in(-\infty, 0),$
\[
-\phi(s) = \int_{s}^{0}\frac{\rho(u)}{\sqrt{1-\rho^2(u)}}du \ge  \int_{s}^{0}\rho(u)du
 \ge -\inf\rho|_{[s,0]}s=-s\rho(s)\,,
\]
which implies that $\phi$ is unbounded, since $\rho(s)\rightarrow (-b/a)^{1/r}>0$ as
$s\rightarrow -\infty.$
%\[
%\lim_{s\rightarrow-\infty}\phi(s)=-\infty.
%\]

Denoting by $B_0$ the horoball of $\hf$ with boundary $\mathscr H_0$\,,
it follows from the above considerations that $\Sigma'$ is
an $H_r$-graph over $\hf-B_0$ which is unbounded
and has boundary $\partial\Sigma'=\mathscr H_0\times\{0\}$
(Fig. \ref{fig-Horograph}).
In particular, $\Sigma'$ is homeomorphic to $\R^n$.
Furthermore, $\Sigma'$ is everywhere non convex, since,
from the identities \eqref{eq-principalcurvatures}, one has
\[
k_i=-\lambda_i\rho<0 \,\,\, (1\le i\le n-1) \,\,\,\, \text{and}\,\,\,  k_n=\rho'>0,
\]
where $\lambda_i$ is the (positive constant) $i$-th principal curvature of $f_s$\,.

Finally, since $\rho(0)=1,$  as in the previous theorems, we have that
any trajectory of $\nabla\xi$ on $\Sigma'$ meets  $\partial\Sigma'$ orthogonally.
Consequently, setting $\Sigma''$ for the reflection of $\Sigma'$ with respect to
$\hf\times\{0\}$ and defining
\[
\Sigma:={\rm closure}\,(\Sigma')\cup {\rm closure}\,(\Sigma''),
\]
we have that $\Sigma$ is a properly embedded
$H_r$-hypersurface of $\hf\times\R$ which is foliated
by horospheres and is homeomorphic to $\R^n$ (Fig. \ref{fig-CompleteK-nonconvex}),
as we wished to prove.
\end{proof}

\begin{figure}[htbp]
\includegraphics{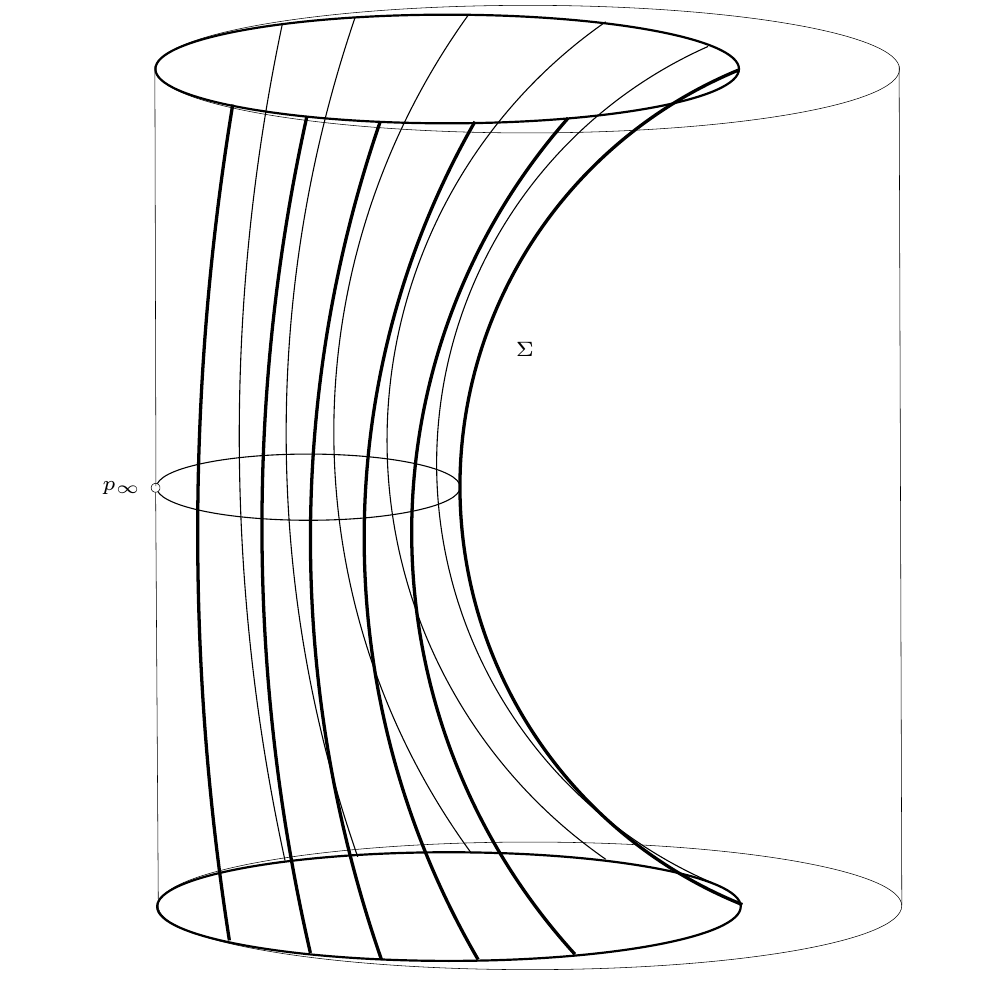}
\caption{\small A piece of a properly embedded everywhere non convex
$H_r(> 0)$-hypersurface of $\hf\times\R$ which is foliated by horospheres.}
\label{fig-CompleteK-nonconvex}
\end{figure}

%%%%%%%%%%%%%%%%%%%%%%%%%%%%%%%%%%%%%%%%%%%%%%%%%%%%%%%%%%%%%%%%%

Our next result establishes that the conditions on the parity of $r$ and on the
sign of $H_r-H_r^0$ in Theorem \ref{th-Hr-horospheretype} are necessary to the conclusion.

\begin{theorem} \label{th-nonexistencehorospheres}
Let $\mathscr F$ be a family of parallel  horospheres in
 $\hf.$
Assume that, for some $r\in\{1,\dots, n\},$  $\Sigma$ is a complete connected $H_r(>0)$-hypersurface
of \,$\hf\times\R$ with no horizontal points,
and that each connected component of any horizontal section $\Sigma_t\subset\Sigma$ is a
(vertically translated) horosphere of $\mathscr F.$ Under these conditions, one has $r<n.$
Assume, in addition, that  either of the following assertions holds:
\begin{itemize} %[parsep=1ex]
  \item [\rm i)] $r$ is even and $H_r\ge H_r^0.$
  \item [\rm ii)] $r$ is odd.
\end{itemize}
Then, $\Sigma=\mathscr H_s\times\R$ for some $\mathscr H_s\in\mathscr F.$  In particular, $H_r=H_r^0.$
\end{theorem}

\begin{proof}
  Let $\Sigma_0\subset\Sigma$ be the open set of points $x\in\Sigma$
  satisfying $\theta(x)\ne 0.$ Our aim is to prove that $\Sigma_0$ is empty.
  Assuming otherwise, choose  $x_0\in\Sigma_0$\,.
  Since $\Sigma$ has no horizontal points, we can suppose
  (after possibly a reflection about a horizontal hyperplane) that
  there is an open neighborhood  $\Sigma'\subset\Sigma_0$ of $x_0$ which is
  an $(f_s,\phi)$-graph, $f_s\in\mathscr F.$

  The $\tau$-function associated to $\Sigma'$ satisfies
  $\tau'=a\tau+b,$ where $a$ and $b\ne 0$ are the (constant) functions \eqref{eq-a&b} determined by $\mathscr F$
  and $H_r$\,. Also,  there is a maximal interval $I=(s_1, s_2)\subset\R,$ $-\infty\le s_1<s_2\le+\infty$,
  such that $\tau(I)\subset (0,1).$

  Let us suppose that $r=n.$ In this case, we have $a=0,$ which gives
  $\tau'(s)=b\ne 0,$ that is,  $\tau(s)=bs+c,$ $c\in\R.$ In particular, $s_1>-\infty$ and $s_2<+\infty$,
  and $\tau$ is increasing (if $b>0$), or decreasing (if $b<0$) in $(s_1,s_2).$
  So, $\tau$ vanishes in $s_1$ or $s_2$\,. Assuming the former, we have that $\phi$ is defined at $s_1$ and
  $\phi'(s_1)=0.$ Thus, for any $p\in\R^{n-1},$  the point
  $x=(f_{s_1}(p),\phi(s_1))\in \Sigma'$ is  horizontal, contrary to our assumption.
  Therefore, if $r=n,$ then $\Sigma_0=\emptyset,$ which implies that $\Sigma=\mathscr H_s\times\R$ for some
  $s\in\R.$ But this contradicts the assumed positiveness of $H_n$\,. Hence, we must have $r<n.$

  Let us assume now that (i) holds. Then, we have $b<0<a.$ Also, on $(s_1,s_2),$
  \[
  \tau'=a\tau+b<a+b=\frac{r(H_r^0-H_r)}{H_{r-1}^0}\le 0 \quad\text{and}\quad \tau''=a\tau'+b<0,
  \]
  that is, $\tau$ is decreasing and concave in $(s_1,s_2),$ which clearly implies that
  $s_2<+\infty$ and $\tau(s_2)=0.$ As in the preceding paragraph, this leads to the
  existence of a horizontal point of $\Sigma.$
  Therefore, $\Sigma_0=\emptyset$ if (i) holds, which implies
  that $\Sigma=\mathscr H_s\times\R$ for some $s\in\R.$

  Finally, let us assume that (ii) holds. In this case, one has $a, b>0,$ which
  gives that $\tau$ is increasing and convex. From this point, we get easily to the conclusion
  by reasoning just as in the last paragraph.
\end{proof}

An isometry $\varphi$ of $\hf$ which fixes only one point $p_\infty\in \hf(\infty)$ is called
\emph{parabolic.} Such isometries have the following fundamental property:
\emph{The horospheres of $\hf$ centered at $p_\infty\in \hf(\infty)$ are invariant by parabolic
isometries of $\hf$ that fix $p_\infty$} (cf. \cite[Proposition 7.8]{eberlein-o'neil}).

We point out that any isometry $\varphi$ of $\hf$  has a natural extension
to an isometry $\Phi$  of $\hfr.$ Namely,
\[
\Phi(p,t)=(\varphi(p),t), \,\,\, (p,t)\in \hfr.
\]
We call $\Phi$ \emph{parabolic} if $\varphi$ is parabolic. More specifically, if
\[
\mathscr F=\{f_s:\R^{n-1}\rightarrow\hf \,;\,  s\in(-\infty,+\infty)\}, \,\,\, f_s(\R^{n-1})=\mathscr H_s\,,
\]
is the family of parallel horospheres which are invariant by $\varphi,$ we say that $\varphi$ and $\Phi$ are
$\mathscr F$-\emph{parabolic} isometries.

%%%%%%%%

In the upper half-space model of $\h_\R^n$,
Euclidean horizontal translations in a fixed direction are parabolic.
As for the other hyperbolic spaces, the parabolic isometries are more involved
(see, e.g.,  \cite{kim}). Nevertheless, inspired by the real case,
we say that parabolic isometries are \emph{translational}.

%%%%%%%%

Finally,  let us remark that,
given a family $\mathscr F$ of
parallel horospheres in $\hf,$  if
a hypersurface $\Sigma$ of $\hf\times\R$ is invariant by
$\mathscr F$-parabolic isometries of $\hfr$, then
any connected component of any horizontal section $\Sigma_t\subset\Sigma$
is contained in a (vertically translated) horosphere of $\mathscr F.$

Now, we are in position to classify all complete connected $H_r(>0)$-hypersurfaces of
$\hfr$  with no horizontal points which are invariant by parabolic isometries.

\begin{theorem} \label{th-classificationhorospheres}
Let $\mathscr F$ be a family of parallel horospheres of $\hf.$
Assume that   $\Sigma$ is a complete connected $H_r(>0)$-hypersurface
of $\hfr$, $r\in\{1,\dots, n\},$ with no horizontal points, which is
invariant by $\mathscr F$-parabolic isometries.
Then, up to ambient isometries,  $\Sigma$ is either
a cylinder over a horosphere of \,$\hf$ or the embedded $H_r$-hypersurface obtained in
Theorem \ref{th-Hr-horospheretype}.
\end{theorem}

\begin{proof}
  Assume that $\Sigma$ is not a cylinder over a horosphere of $\hf.$
  By Theorem \ref{th-nonexistencehorospheres}, we have that $r(<n)$ is even and
  $0<H_r<H_r^0.$
  %In addition, by the tangency principle, $\theta$ does not vanish in an open set
  %of $\Sigma,$ so that the set $\Sigma_0\subset\Sigma$ on which $\theta\nabla\xi$ never vanishes
  %is open and dense in $\Sigma.$ From this fact, as we did before,
  In this case, the open set
  \[
  \Sigma_0:=\{x\in\Sigma\,;\, \theta(x)\ne 0\}
  \]
  is dense in $\Sigma.$ Indeed, the $r(<n)$-th mean curvature
  of any nonempty open $\mathscr F$-invariant vertical
  set $U\subset\Sigma$ would be $H_r^0>H_r.$ Thus, given $x_0\in\Sigma_0\,,$
  there exists an $(f_s,\phi)$-graph $\Sigma'\owns x_0$ in $\Sigma_0$ with $f_s\in\mathscr F.$
  In addition,  the associated $\tau$-function, restricted to a maximal interval
  $(s_0,s_1),$ $-\infty\le s_0<s_1\le+\infty,$ satisfies $0<\tau<1.$

  We have that $\tau$ is a solution of the ODE $y'=ay+b$ determined by $\mathscr F$ and $H_r\,.$
  The conditions on the parity of $r$ and the sign of $H_r-H_r^0$, as in the proof of Theorem \ref{th-Hr-horospheretype},
  give that $\tau$ is increasing and convex, which implies that $s_1<+\infty$ and that $\tau(s_1)=1.$

  Consider the solution \eqref{eq-tauspecific} of $y'=ay+b$ and denote it by  $\tilde\tau.$ Since
  $\tilde\tau(0)=\tau(s_1)=1,$ and the coefficients $a$ and $b$ are constants, by the uniqueness of
  solutions satisfying initial conditions,
  we have that $\tau(s)=\tilde\tau(s-s_1).$ This, together with the homogeneity of the horospheres of $\hf,$
  implies that $\Sigma'$ coincides with the $(f_s,\phi)$-graph determined by $\tilde\tau.$ From this fact
  and the density of $\Sigma_0$ in $\Sigma,$ we conclude that
  $\Sigma$ coincides with the embedded $H_r$-hypersurface obtained in Theorem \ref{th-Hr-horospheretype},
  as we wished to prove.
\end{proof}

Given a totally geodesic hyperplane $\mathscr E_0$ of $\h^n,$ let us recall that there exists
a family $\mathscr F:=\{\mathscr E_s\,;\, s\in(-\infty, +\infty)\}$ of parallel hypersurfaces of $\h^n$
such that the distance of any point of $\mathscr E_s$ to $\mathscr E_0$ is $|s|.$
The family $\mathscr F$ foliates $\h^n,$
and each member $\mathscr E_s$  of $\mathscr F$, which is called an \emph{equidistant hypersurface},
is properly embedded and homeomorphic to $\R^{n-1}$ (Fig. \ref{fig-equidistants}).

\begin{figure}[htbp]
\includegraphics{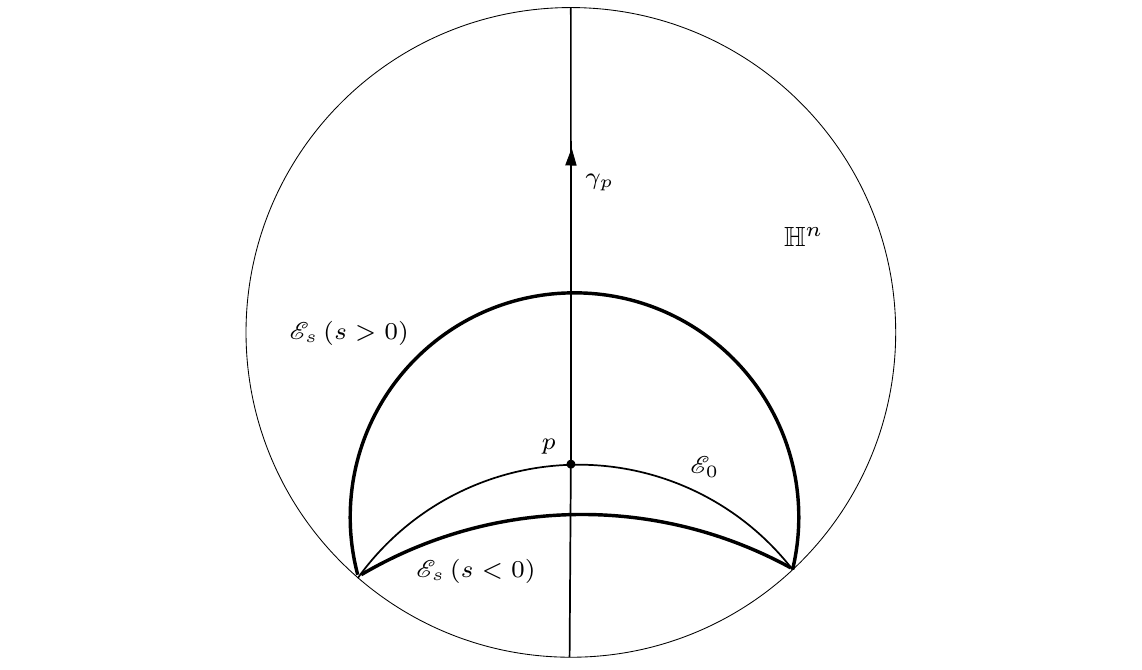}
\caption{\small Equidistant hypersurfaces in the Poincaré ball model of $\h^n.$}
\label{fig-equidistants}
\end{figure}

We shall also write $\mathscr F$ as a family of immersions:
\[
\mathscr F=\{f_s:\R^{n-1}\rightarrow\h^n\,;\, s\in(-\infty,+\infty)\},
\]
that is, for each $s\in (-\infty,+\infty),$ $f_s(\R^{n-1})$ is the equidistant $\mathscr E_s$ to
$\mathscr E_0=f_0(\R^{n-1})$\,.

Given a geodesic $\gamma_p$ orthogonal to the members of $\mathscr F,$ $p\in\mathscr E_0,$
any equidistant hypersurface $\mathscr E_s$ is totally umbilical with constant principal curvatures all equal
to
\[
k^s=-\tanh(s)
\]
with respect to the unit normal $\eta_s=\gamma_p'$ (see Section \ref{sec-Hrgraphs}).
In particular, $\mathscr F$ is isoparametric. %We also remark that, for all $s\ne 0,$ one has $0<|k^s|<1.$
Also, given a constant $H_r$\,, the coefficients $a$ and $b$ of
the  differential equation $y'=ay+b$ associated to $\mathscr F$ and $H_r$ are:
\begin{equation} \label{eq-a&bequidistant}
a(s)=-(n-r)\tanh(s)\quad\text{and}\quad b(s)=b_r\tanh^{1-r}(s), \,\,\,  b_r=rH_r{{n-1}\choose{r-1}}^{-1}.
\end{equation}

It will be convenient to reconsider the  constant
$
C_r:=\frac{n-r}{n}{{n}\choose{r}}
$
and recall  that, for $1\le r<n,$  the following identity holds:
\begin{equation} \label{eq-relation}
\frac{b_r}{n-r}=\frac{H_r}{C_r}\,\cdot
\end{equation}

Our next result establishes that, for $H_r\in(0, C_r)$, $1\le r<n,$ there exists
a one-parameter family of properly embedded $H_r$-hypersurfaces in $\h^n\times\R$ which are foliated by
(vertical translations of) parallel equidistant hypersurfaces of $\h^n.$
In this setting, we have $0<H_r/C_r<1,$ so we can define:
\begin{equation} \label{eq-sr}
s_r:={\rm arctanh}\,(H_r/C_r)^{1/r}.
\end{equation}

\begin{theorem}  \label{th-Hr-equidistanttype}
Given  $r\in\{1,\dots ,n-1\}$, let $H_r\in (0,C_r).$
Then, there exists  a one parameter family
$\mathscr S:=\{\Sigma(\lambda)\,;\, \lambda\in(s_r,+\infty)\}$
of properly embedded and
everywhere non convex $H_r$-hypersurfaces of \,$\h^n\times\R.$  Each member
$\Sigma(\lambda)$ of $\mathscr S$ is homeomorphic to
$\R^n$ and is foliated by equidistant hypersurfaces. Moreover, $\Sigma(\lambda)$
is symmetric with respect to %the horizontal hyperplane
$\h^n\times\{0\}$, and its height function is unbounded. % above and below.
\end{theorem}

\begin{proof}
 Let $\mathscr F:=\{f_s\,;\, s\in(0,+\infty)\}$ be a family of parallel
 equidistant hypersurfaces of $\h^n.$ Since $0<H_r<C_r$\,, it  follows from the relation \eqref{eq-relation} that
 $0<b_r/(n-r)<1.$ Thus, we can choose $\lambda >0$ such that
\begin{equation} \label{eq-inequalitytanh}
\tanh^r(\lambda)>\frac{b_r}{n-r}\,\cdot
\end{equation}
In particular, $\lambda\in(s_r,+\infty).$

Let $\tau_\lambda$ be the solution of $y'=ay+b$ satisfying $y(\lambda)=1,$
where $a$ and $b$ are the functions in \eqref{eq-a&bequidistant}.
Then, from \eqref{eq-inequalitytanh}, we have
\[
\tau_\lambda'(\lambda)=-(n-r)\tanh(\lambda)+b_r\tanh^{1-r}(\lambda)<0,
\]
so that $\tau_\lambda$ is decreasing near $\lambda$\,.

We claim that $\tau_\lambda$ is decreasing on the whole interval $[\lambda,+\infty)$.
To show that, it suffices to prove that $\tau_\lambda$ has no critical points in $(\lambda,+\infty).$
Assuming otherwise, consider $s_1>\lambda$ satisfying $\tau_\lambda'(s_1)=0.$
Since $\tau_\lambda'(s_1)=a(s_1)\tau_\lambda(s_1)+b(s_1),$ we have that
$\tau_\lambda(s_1)=-b(s_1)/a(s_1)>0.$ We also have  $a'<0$ and $b'\le 0.$ Thus,
\[
\tau_\lambda''(s_1)=a'(s_1)\tau_\lambda(s_1)+b'(s_1)<0,
\]
which implies that $s_1$ is necessarily a local maximum for $\tau_\lambda.$
This proves the claim, for $\tau_\lambda$ is decreasing near $\lambda,$ so that
a local maximum $s_1>\lambda$ for $\tau_\lambda$ should be preceded by a local minimum.

We also have that  $\tau_\lambda$ is positive in $[\lambda,+\infty).$ Indeed,
if we had $\tau_\lambda(s)\le 0$ for some $s>\lambda,$  it would give $\tau_\lambda'(s)=a(s)\tau(s)+b(s)>0,$
and then $\tau_\lambda$ would be increasing near $s.$

It follows from the above considerations that
\[
0<\tau_\lambda(s)\le 1=\tau_\lambda(\lambda) \,\,\, \forall s\in [\lambda,+\infty).
\]
Furthermore, since $\tau_\lambda$ is decreasing and positive, we have that
$\tau_\lambda'(s)\rightarrow 0$ as $s\rightarrow +\infty.$ This, together with
the equalities $\tau_\lambda'=a\tau_\lambda+b$ and $\rho_\lambda^r=\tau_\lambda$, gives
\begin{equation} \label{eq-limit}
\lim_{s\rightarrow+\infty}\rho_\lambda(s)=\left({H_r}/{C_r}\right)^{1/r}>0.
\end{equation}

Therefore, the $(f_s,\phi)$ graph $\Sigma'(\lambda)$ associated
to $\rho_\lambda$ (see Fig. \ref{fig-equidistant})
is an $H_r$-hypersurface of $\h^n\times\R$ whose $\phi$-function  is
\[
\phi_\lambda(s)=\int_{\lambda}^{s}\frac{\rho_\lambda(u)}{\sqrt{1-\rho_\lambda^2(u)}}du, \,\,\, s\in [\lambda,+\infty).
\]

\begin{figure}[htbp]
\includegraphics{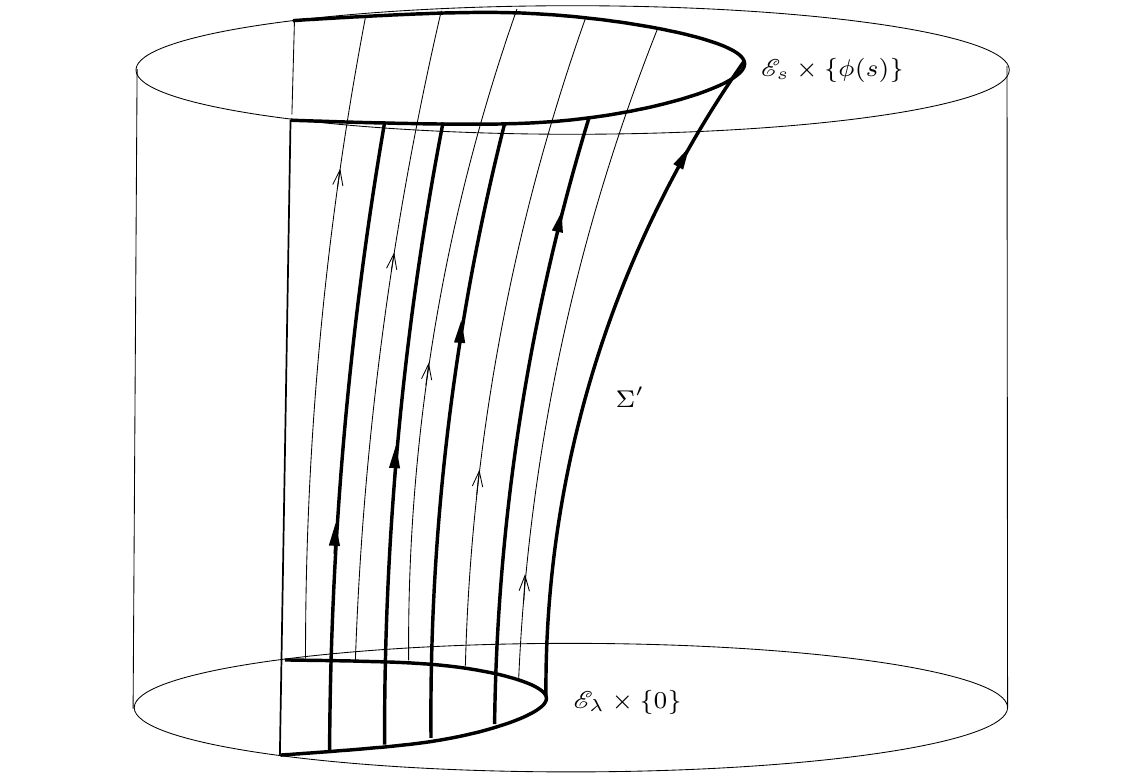}
\caption{\small A piece of the graph $\Sigma'$, on which  all  trajectories of  $\nabla\xi$
emanate from $\mathscr E_{\lambda}\times\{0\}$ orthogonally.}
\label{fig-equidistant}
\end{figure}

As in the preceding proofs, we obtain a properly embedded $H_r$-hypersurface
$\Sigma(\lambda)$ of $\h^n\times\R$ by reflecting $\Sigma'(\lambda)$ with respect to $\h^n\times\{0\},$
since $\rho_\lambda(\lambda)=1$ and $\rho_\lambda'(\lambda)<0$.
It is also clear from equalities \eqref{eq-principalcurvatures} that,
except for $k_n=\rho'$, its principal curvatures  $k_i$ are all positive, so that
$\Sigma(\lambda)$ is nowhere convex.

Finally, considering \eqref{eq-limit} and the fact that $\rho_\lambda$ is decreasing, we have
\[
\phi_\lambda(s)=\int_{\lambda}^{s}\frac{\rho_\lambda(u)}{\sqrt{1-\rho_\lambda^2(u)}}du\ge \int_{\lambda}^{s}\rho_\lambda(u)du\ge \left({H_r}/{C_r}\right)^{1/r}(s-\lambda),
\]
which clearly implies that the height function of $\Sigma(\lambda)$ is unbounded.
\end{proof}

%%%%%%%%%%%%%%%%%%%%%%%%%%%%%%%%%%%%%%%%%%%%%

\begin{theorem} \label{th-graphequidistants}
Let $\mathscr F:=\{f_s\,;\, s\in(0,+\infty)\}$ be a family of equidistant
hypersurfaces in $\h^n.$
Given $r\in\{1,\dots ,n-1\}$ and $H_r\in (0, C_{r}),$ let $s_r>0$ be the constant defined in
\eqref{eq-sr}. Then,
there exists a complete everywhere non convex  $H_r$-hypersurface $\Sigma$ in $\h^n\times\R$ which
is an $(f_s,\phi)$-graph, $s\in (s_r,+\infty).$  Furthermore,
the height function of $\Sigma$ is unbounded above and below, and $\Sigma$ is
asymptotic to $\mathscr E_{s_r}\times(-\infty, 0).$ %in $\h^n\times(-\infty, 0).$
\end{theorem}

\begin{proof}
  Let $\tau$ be the solution of the differential equation
  $y'=ay+b$ associated to $H_r$ and $\mathscr F$ (i.e., with $a$ and $b$ as
  in \eqref{eq-a&bequidistant}) which satisfies the initial condition
  $\tau(s_r)=1.$

 From its definition,  we have that $s_r$ satisfies $\tau'(s_r)=0.$ In addition,
  \[
  \tau''(s_r)=a'(s_r)+b'(s_r)<0,
  \]
  so that $s_r$ is a local  maximum of $\tau.$ Reasoning  as in the preceding proof, we get
  that $\tau$,  and so $\rho=\tau^{1/r}$, is positive and decreasing in $(s_r,+\infty).$
  From this, we conclude analogously that $\rho(s)\rightarrow (H_r/C_r)^{1/r}$ as
  $s\rightarrow+\infty.$

  Now, for a fixed $s_0>s_r$\,, define
  \[
  \phi(s):=\int_{s_0}^{s}\frac{\rho(u)}{\sqrt{1-\rho^2(u)}}du, \,\,\, s\in (s_r,+\infty),
  \]
  and let $\Sigma$ be the corresponding $(f_s,\phi)$-graph. As before, we have that
  $\Sigma$ is nowhere convex. Denoting by $\Omega_{s_r}$ the convex connected component of
  $\h^n -\mathscr E_{s_r},$ we also have that $\Sigma$ is a graph over $\h^n-\Omega_{s_r}$
  (Fig. \ref{fig-equidistant02}).

  For $s>s_0$\,, we have
  \[
  \phi(s)=\int_{s_0}^{s}\frac{\rho(u)}{\sqrt{1-\rho^2(u)}}du\ge \int_{s_0}^{s}\rho(u)du\ge (H_r/C_r)^{1/r}(s-s_0),
  \]
  which gives that $\phi$, and so the height function of $\Sigma,$ is unbounded above.

  Finally, given a constant $C>0,$ there exists $\bar s\in (s_r, s_0)$ such that
  \[
  \frac{1}{\rho'(s)}<-C \,\,\,\forall s\in(s_r,\bar s),
  \]
  for $\rho'(s_r)=0.$ Thus, for such values of $s,$ one has
  \begin{eqnarray}
  \phi(s) &= &\int_{s_0}^{s}\frac{\rho'(u)\rho(u)}{\rho'(u)\sqrt{1-\rho^2(u)}}du \le
  -C\rho(\bar s)\int_{s_0}^{s}\frac{\rho'(u)}{\sqrt{1-\rho^2(u)}}du \nonumber \\
  & = & -C\rho(\bar s)\int_{\rho(s_0)}^{\rho(s)}\frac{d\rho}{\sqrt{1-\rho^2}}= -C\rho(\bar s)(\arcsin\rho(s)-\arcsin\rho(s_0))\nonumber\\
  & \le & -C\rho(\bar s)(\arcsin\rho(\bar s)-\arcsin\rho(s_0)), \nonumber
  \end{eqnarray}
  which implies that $\phi(s)\rightarrow-\infty$ as $s\rightarrow s_r$\,, since
  $\rho(\bar s)$ is bounded away from zero. Therefore, the height
  function of $\Sigma$ is unbounded below, and $\Sigma$ is
  asymptotic to $\mathscr E_{s_r}\times(-\infty, 0)$ in $\h^n\times(-\infty, 0),$
  as we wished to prove.
  \end{proof}

\begin{figure}[htbp]
\includegraphics{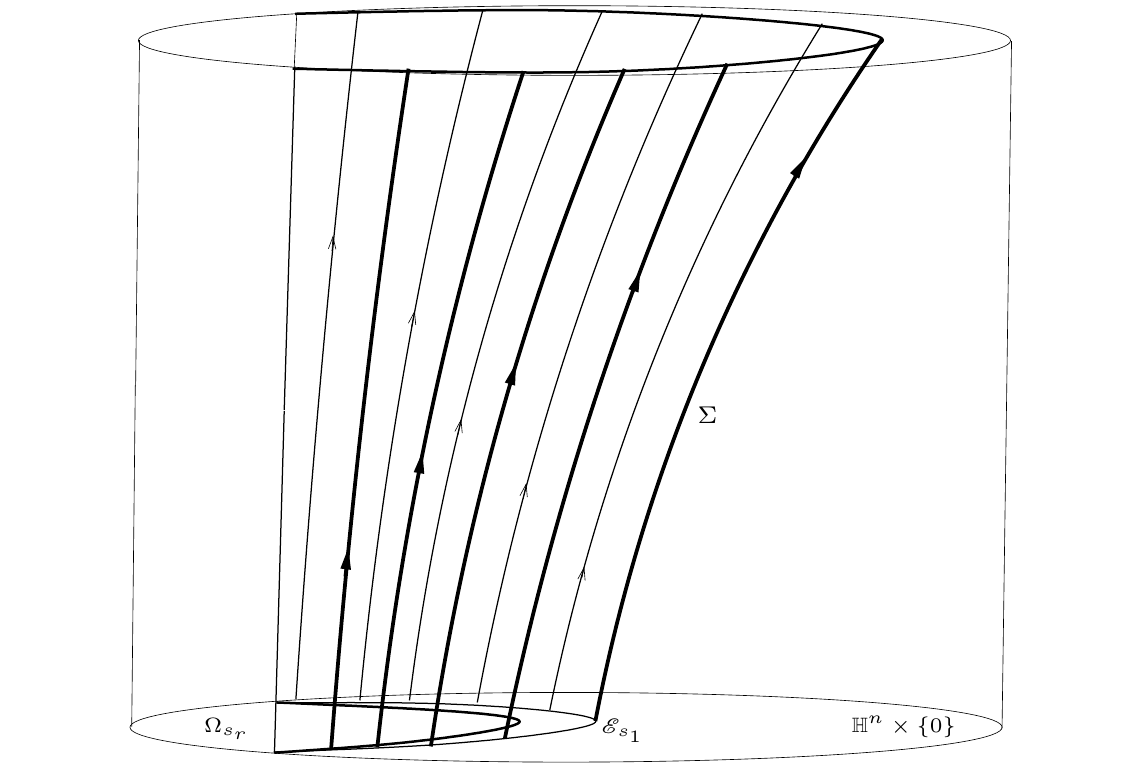}
\caption{\small A piece of the $(f_s,\phi)$-graph $\Sigma$ which is above $\h^n\times\{0\}.$  As $s\rightarrow-\infty$,
the  trajectories of  $\nabla\xi$ converge asymptotically to   $\mathscr E_{s_r}\times (0,-\infty).$}
\label{fig-equidistant02}
\end{figure}

A family $\mathscr F=\{\mathscr E_s\,;\, s\in(-\infty,+\infty)$ of equidistant hypersurfaces in
$\h^n$  determines a group of translational isometries which we shall call $\mathscr F$-\emph{hyperbolic}. In the upper half-space model of
$\h^n,$ taking $\mathscr E_0$ as a Euclidean half vertical hyperplane orthogonal to $\partial_\infty\h^n$
through  the ``origin'' $o,$ we have
that the $\mathscr F$-hyperbolic isometries are the Euclidean homotheties from $o.$
It should be noticed that the equidistant hypersurfaces
of $\mathscr F$ are all invariant by $\mathscr F$-hyperbolic isometries.

The natural extension of an $\mathscr F$-hyperbolic isometry
of $\h^n$ to $\h^n\times\R$ will  also be called
$\mathscr F$-\emph{hyperbolic}.
If $\Sigma$ is a hypersurface of $\h^n\times\R$ which is invariant by
$\mathscr F$-hyperbolic isometries, it is clear that any connected component of
any horizontal section $\Sigma_t$ of $\Sigma$ is contained in
$\mathscr E_s\times\{t\}$ for some $s\in(-\infty,+\infty).$

Next, we classify $H_r(>0)$-hypersurfaces of $\h^n\times\R$
(without horizontal points or totally geodesic horizontal sections)
which are invariant by hyperbolic translations.

%%%%%%%%%%%%%%%%%%%%%%%%%%%%%%%%%%%%%%%%%%%%%

\begin{theorem} \label{th-nonexistenceequidistant}
Let $\mathscr F=\{f_s \,;\, s\in(-\infty,+\infty)\}$ be a family of parallel equidistant hypersurfaces
in $\h^n.$ Assume that, for some $r\in\{1,\dots, n\},$  $\Sigma$ is a complete connected $H_r(>0)$-hypersurface
of \,$\h^n\times\R$ which is invariant by  $\mathscr F$-hyperbolic translations. Assume further that
$\Sigma$ has no horizontal points, and that no  horizontal section $\Sigma_t$ of $\Sigma$
is totally geodesic in $\h^n\times\{t\}$ (i.e., $\Sigma_t\not\subset\mathscr E_0=f_0(\R^n)).$
Under these conditions,  the following  assertions hold:
\begin{itemize}[parsep=1ex]
  \item [\rm i)] $r<n$.
  \item [\rm ii)]  $0<H_r<C_r$\,.
  \item [\rm iii)] $\Sigma$ is either the cylinder over the equidistant \,$\mathscr E_{s_r}$
  or, up to an ambient isometry, one of the embedded hypersurfaces obtained
  in Theorems \ref{th-Hr-equidistanttype}--\ref{th-graphequidistants}.
\end{itemize}
\end{theorem}

\begin{proof}
  Set $\Sigma_0:=\{x\in\Sigma\,;\, \theta(x)\ne 0\}$ and assume $\Sigma_0\ne\emptyset.$  Given $x_0\in\Sigma_0$\,,
  as in previous proofs, we can assume there is an
  open set $\Sigma'\subset\Sigma_0$ which is an $(f_s,\phi)$-graph containing $x_0$\,.
  Its $\tau$ function satisfies
  $\tau'=a\tau+b$, where  $a$ and $b$ are the functions given in \eqref{eq-a&bequidistant}. Also, $\tau$ is
  defined in a maximal interval $(s_0\,, s_1)\subset\R$ such that $0<\tau|_{(s_0,s_1)}<1.$
  Since no horizontal section of $\Sigma$ is totally geodesic, we can assume  $0<s_0<s_1\le+\infty.$

  The maximality of $(s_0,s_1)$ gives that $\tau(s_0)=0$ or $\tau(s_0)=1.$ In the former case,
  we have $\tau'(s_0)=b(s_0)>0.$ Then, $\phi(s_0)$ is well defined (by Lemma \ref{lem-convergenceintegral})
  and $\phi'(s_0)=0,$ so that $x=(f_{s_0}(p),\phi(s_0)),$ $p\in\R^{n-1},$
  is a horizontal point of  $\Sigma$, contrary to our hypothesis.
  Then, we must have $\tau(s_0)=1.$ In particular, near $s_0$\,, $\tau$ is decreasing
  in $(s_0,s_1),$  which implies that $r<n.$ Indeed, for $r=n,$  $\tau'=b>0.$

  Assume now that $H_r\ge C_r$, $r<n.$ Then, we have
  \begin{eqnarray}
  \tau'(s_0) & = & a(s_0)+b(s_0)=-(n-r)\tanh(s_0)+b_r\tanh^{1-r}(s_0)\nonumber\\
             & = & (n-r)((H_r/C_r)\tanh^{1-r}(s_0)-\tanh(s_0))\nonumber\\
             &\ge &  (n-r)(\tanh^{1-r}(s_0)-\tanh(s_0))>0,\nonumber
  \end{eqnarray}
  which contradicts that $\tau$ is decreasing near $s_0$\,.

  It follows from the above considerations that, if $\Sigma_0\ne\emptyset,$ then
  $r<n$ and $H_r<C_r$\,. Furthermore,  a direct computation gives that $\tau'(s_0)\le 0$ if and only
  if $s_0\ge s_r$\,.
  If $s_0=\lambda>s_r$\,,  then $\tau$ coincides with the function $\tau_\lambda$ of the $(f_s,\phi)$-graph
  associated to the hypersurface $\Sigma(\lambda)$ of Theorem \ref{th-Hr-equidistanttype}. From this,
  arguing as in preceding proofs, we conclude that $\Sigma=\Sigma(\lambda).$
  By the same token, if $s_0=s_r,$ then $\Sigma=\Sigma'$ is the complete graph obtained
  in Theorem \ref{th-graphequidistants}.

  Let us suppose now that $\Sigma_0=\emptyset.$  In this case, we must have
  $\Sigma=\mathscr E_s\times\R,$ where $\mathscr E_s=f_s(\R^{n-1})$ is an equidistant hypersurface with $r$-th mean curvature
  $H_r^s=H_r$\,, so that $s=s_r$\,.  Therefore, we have $r<n$ (since we are assuming $H_r>0$) and
  \[
  H_r=|H_r^{s_r}|={{n-1}\choose{r}}\tanh^r(s_r)=C_r \tanh^r(s_r)<C_r\,,
  \]
  which concludes our proof.
  \end{proof}

%%%%%%%%%%%%%%%%%%%%%%%%%%%%%%%%%%%%%%%%%%%%%%%%%%%%%%%%%%%%%%%%%%%%%%%

\section{Translational $r$-minimal Hypersurfaces of $\hfr.$} \label{sec-translationalr-minimal}

In this section, we construct and classify
$r$-minimal hypersurfaces in $\hfr$ which are invariant by translational isometries.
%as we did for $H_r(>0)$-hypersurfaces in the preceding section.
It will be convenient to
consider first the case of  hyperbolic isometries of $\h^n.$

\begin{theorem} \label{th-rminimalequidistant}
Let  $\mathscr F=\{f_s\,;\, s\in(-\infty,+\infty)\}$ be a family of parallel equidistant hypersurfaces
to a totally geodesic hyperplane $\mathscr E_0=f_0(\R^{n-1})$ of \,$\h^n.$
Then, for each  $r\in\{1, \dots ,n\},$  there exists a one-parameter family
$\mathscr S=\{\Sigma(\lambda)\,;\, \lambda>0\}$ of properly embedded $r$-minimal
hypersurfaces of \,$\h^n\times\R$ which are all homeomorphic to
$\R^n$ and invariant by $\mathscr F$-hyperbolic translations.
Each member $\Sigma(\lambda)\in\mathscr S$ has the following additional properties:
\begin{itemize}
  \item [\rm i)] For $r=n,$ $\Sigma(\lambda)$  is a constant angle entire $r$-minimal graph over $\h^n$ whose
  height function is unbounded above and below.
\end{itemize}

For $r<n,$ we distinguish the following cases:
\begin{itemize}[parsep=1ex]
  \item [\rm ii)] $\lambda>1:$  $\Sigma(\lambda)$ is symmetric with respect to
  the horizontal hyperplane $P_0=\h^n\times\{0\}$, and is contained
  in a slab $\h^n\times(-\alpha,\alpha),$ $\alpha>0.$
  \item [\rm iii)] $\lambda=1:$  $\Sigma(\lambda)$ is an $(f_s,\phi)$-graph ($s>0$) which is bounded above, unbounded below,
  and asymptotic to $\mathscr E_0\times(0,-\infty).$
  \item [\rm iv)] $\lambda<1:$  $\Sigma(\lambda)$ is an entire graph over $\h^n$ which is symmetric with respect to
  $\mathscr E_0$, and is contained is a slab $\h^n\times(-\alpha,\alpha),$ $\alpha>0.$
\end{itemize}
Furthermore, except for the cylinders $\mathscr E_s\times\R, \,s\ne 0$ (in the case $r=n$), and
up to ambient isometries,
the members of $\mathscr S$ are the only complete non totally geodesic $r$-minimal
hypersurfaces of \,$\h^n\times\R$ which are invariant by hyperbolic translations.
\end{theorem}

\begin{proof}
The equation \eqref{eq-difequation} determined by $\mathscr F$ and $H_r=0$ is:
\begin{equation} \label{eq-equid01}
y'=a(s)y, \,\,\, a(s)=-(n-r)\tanh(s).
\end{equation}
For $r=n,$ its solution $\tau$ is constant. So, given $\lambda>0,$  defining
\[
\phi(s)= \lambda s, \,\, s\in (-\infty,+\infty),
\]
we have that the corresponding $(f_s,\phi)$-graph $\Sigma(\lambda)$ is an entire
$n$-minimal graph   whose level hypersurfaces are
the leaves of $\mathscr F.$ Clearly, the height function
of $\Sigma(\lambda)$ is unbounded above and below. Moreover, it follows from
\eqref{eq-thetaparallel} that $\Sigma(\lambda)$ is a constant angle
hypersurface. This proves (i).

Let us suppose now that  $1\le r<n.$  Given $\lambda>0,$ set
\[
\tau_\lambda(s):=\lambda\left(\frac{1}{\cosh s} \right)^{n-r}, \,\, s\in (-\infty, +\infty).
\]
It is easily checked that $\tau_\lambda$ is the solution of \eqref{eq-equid01}
satisfying  $\tau_\lambda(0)=\lambda.$

Assume that $\lambda>1.$ Then, defining
$s_\lambda:={\rm arccosh}\,(\lambda^{1/(n-r)}),$ one has
\[
0<\tau_\lambda(s)\le 1=\tau_\lambda(s_\lambda) \,\,\, \forall s\in [s_\lambda,+\infty).
\]
Hence, setting
\begin{equation}  \label{eq-philambda}
\phi_\lambda(s):=\int_{s_\lambda}^{s}\frac{\rho_\lambda(u)}{\sqrt{1-\rho_\lambda^2(u)}}du, \quad  \rho_\lambda
=\tau_\lambda^{1/r}, \quad  s\in (s_\lambda,+\infty),
\end{equation}
we have that the $(f_s,\phi_\lambda)$-graph $\Sigma'(\lambda)$ is a well
defined  $r$-minimal hypersurface, for $\tau'(s_\lambda)<0.$ Also, since $\tau_\lambda(s_\lambda)=1,$
the closure of $\Sigma'(\lambda)$ intersects $P_0$ orthogonally. Thus, we obtain
an $r$-minimal hypersurface $\Sigma(\lambda)$ by reflecting $\Sigma'(\lambda)$ about $P_0.$

As for the boundedness of  $\phi_\lambda$, we first observe that,
from the equalities $\tau_\lambda=\rho_\lambda^r$ and $\tau_\lambda'=a\tau_\lambda,$ we have  $\rho_\lambda=(r/a)\rho_\lambda'.$
In addition, the function $1/a$ is bounded above by $-1/(n-r)$ in $(0,+\infty).$ Hence,
\begin{eqnarray} \label{eq-equid02}
\phi_\lambda(s) &= & \int_{s_\lambda}^{s}\frac{r\rho_\lambda'(u)}{a(u)\sqrt{1-\rho_\lambda^2(u)}}du
\le-\frac{r}{n-r}\int_{\rho_\lambda(s_\lambda)}^{\rho_\lambda(s)}\frac{d\rho_\lambda}{\sqrt{1-\rho_\lambda^2}}\\[1ex]
         &= & \frac{r}{n-r}(\arcsin\rho_\lambda(s_\lambda)-\arcsin\rho_\lambda(s))\le\frac{\pi r}{2(n-r)}\,,\nonumber %
\end{eqnarray}
which finishes the proof of (ii).

Assuming now $\lambda=1,$ let us fix $s_0>0$ and define
\[
\phi(s)=\int_{s_0}^{s}\frac{\rho(u)}{\sqrt{1-\rho^2(u)}}du, \quad  \rho=\tau_1^{1/r}, \quad  s\in (0,+\infty).
\]

Since $\tau'(0)=0,$ we can mimic the final part of the proof of Theorem \ref{th-graphequidistants}
and conclude that $\phi$ is unbounded below, and that the corresponding $(f_s,\phi)$-graph
$\Sigma$ is asymptotic to $\mathscr E_0\times (-\infty, 0).$ Also, proceeding as in \eqref{eq-equid02}, we
can show that $\phi$ is bounded above. This proves (iii).

Given $0<\lambda<1,$ we have that $0<\tau_\lambda<1$
in $(0,+\infty).$ So, we can define $\phi_\lambda$ as in \eqref{eq-philambda}, replacing
$s_\lambda$ by $0$. Analogously, we have that $\phi_\lambda$ is bounded above, and that the boundary
of the $(f_s,\phi_\lambda)$-graph $\Sigma'(\lambda)$ is $\mathscr E_0\times\{0\}\subset P_0$\,.

Notice that $\phi_\lambda'(0)=\rho_\lambda(0)/\sqrt{1-\rho_\lambda^2(0)}$ is well defined and
positive, since $\rho_\lambda(0)$ is neither $0$ nor $1.$ Thus, we  obtain a complete properly
embedded $r$-minimal hypersurface $\Sigma(\lambda)$ from $\Sigma'(\lambda)$ by reflecting it
with respect to $P_0$\,, and then with respect to the  totally geodesic vertical
hyperplane $\mathscr E_0\times\R$ (Fig. \ref{fig-equidistant03}).
This shows (iv).

Assume now that $\Sigma$ is a complete non totally geodesic $r$-minimal
hypersurface of $\h^n\times\R$
which is invariant by $\mathscr F$-hyperbolic translations. Set
\[
\Sigma_0:=\{x\in\Sigma\,;\, \theta(x)\nabla\xi(x)\ne 0\}
\]
and suppose that $1\le r<n.$ Then, $\Sigma_0\ne\emptyset.$ Otherwise, $\Sigma$ would be either
a horizontal hyperplane or a cylinder over the hyperplane $\mathscr E_0$ of $\h^n.$ In both cases,
$\Sigma$ would be totally geodesic, which is contrary to our assumption.

Therefore, if $1\le r<n,$ for each $x_0\in\Sigma_0$\,,
there is an $(f_s,\phi)$-graph $\Sigma'\subset\Sigma_0$ which contains $x_0$\,, and  whose
$\tau$-function  is a solution of \eqref{eq-equid01}. More precisely, for some $c>0,$
one has
\[
\tau(s)=c\left(\frac{\cosh s_0}{\cosh s} \right)^{n-r}, \,\, s_0\,, s\in (-\infty, +\infty).
\]

Now,  recall that in the cases (ii)--(iv) above, the corresponding
function $\tau_\lambda$ satisfies $\tau_\lambda(0)=\lambda.$ Thus, for
$\lambda:=\tau(0)=c\cosh^{n-r}(s_0),$
the function $\tau$ coincides with  $\tau_\lambda,$ which implies that
$\Sigma'\subset\Sigma(\lambda).$ In addition, no $\Sigma(\lambda)\in\mathscr S$ has
horizontal or totally geodesic points, which gives that $\Sigma_0$ is open and dense in $\Sigma.$
Therefore, $\Sigma$ coincides with $\Sigma(\lambda).$

Finally, let us suppose that $r=n.$ If $\Sigma_0=\emptyset,$ then $\Sigma=\mathscr E_s$ for
some $s\ne 0.$ If $\Sigma_0\ne\emptyset,$ then there exists an $(f_s,\phi)$-graph $\Sigma'\subset\Sigma_0$
whose $\tau$-function is constant. In particular, up to a vertical translation, we have $\phi(s)=\lambda s$
for some $\lambda>0,$ so that $\Sigma'=\Sigma$ is the entire graph $\Sigma(\lambda)$ given in (i).
\end{proof}

\begin{figure}[htbp]
\includegraphics{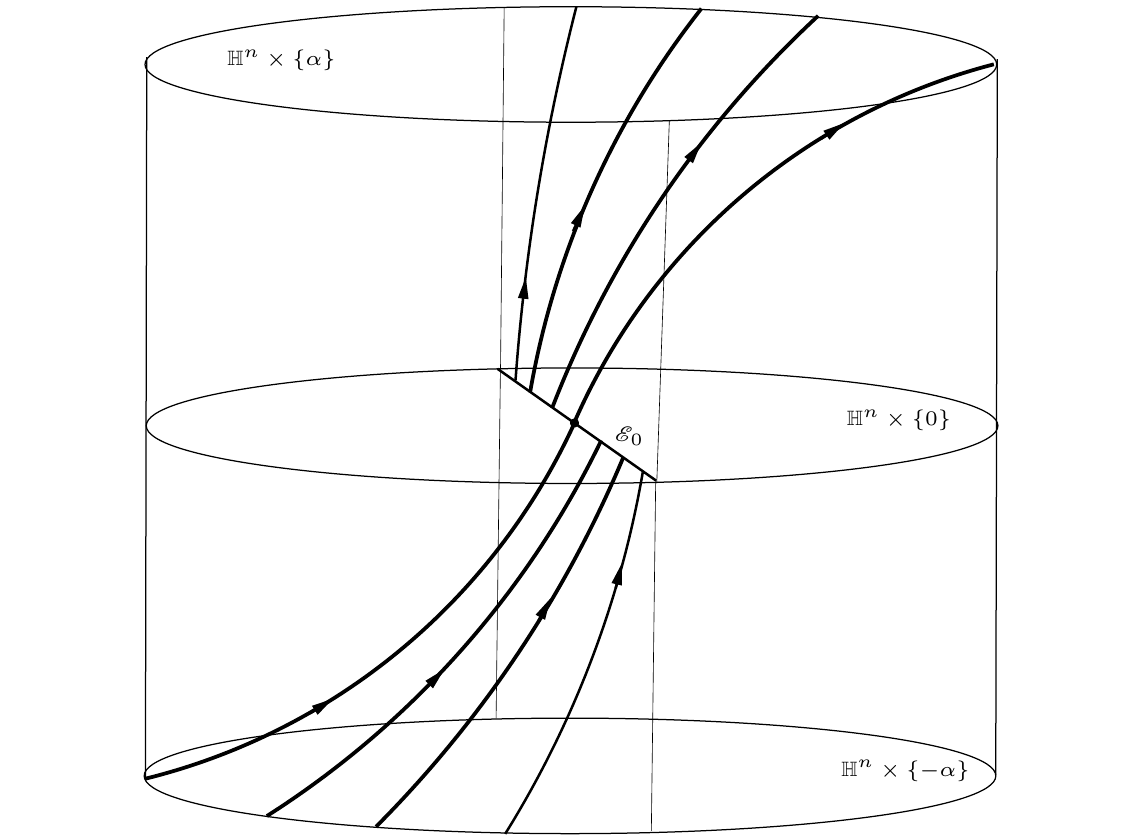}
\caption{\small The figure shows half of the $(f_s,\phi)$-graph $\Sigma'(\lambda)$ (above $\h^n\times\{0\}$)
and half of its reflection with respect to $\mathscr E_0$ (below $\h^n\times\{0\}$).}
\label{fig-equidistant03}
\end{figure}

\begin{remark}
It should be mentioned that the particular case $r=1$
of Theorem \ref{th-rminimalequidistant} was considered in \cite{berard-saearp02}.
\end{remark}

Next, we obtain all complete non totally geodesic $r$-minimal hypersurfaces
of $\hfr$ which are invariant  by parabolic isometries.

\begin{theorem} \label{th-rminimal-horospheres}
Let $\mathscr F=\{f_s\,;\, s\in(-\infty,+\infty)\}$ be a family of
parallel horospheres in $\hf.$
Then, for any $r\in\{1,\dots ,n\},$  there exists
a properly embedded
$r$-minimal hypersurface $\Sigma$ of \,$\hf\times\R$
which is invariant by $\mathscr F$-parabolic isometries.
In addition, $\Sigma$ is homeomorphic to $\R^n$ and has the following  properties:
\begin{itemize}[parsep=1ex]
\item [\rm i)] For $r=n,$ $\Sigma$ is a constant angle  entire graph over $\hf$ whose height function
is unbounded above and below.
\item [\rm ii)] For $r<n,$ $\Sigma$ is symmetric with respect to $\hf\times\{0\}$  and
is contained in a slab $\hf\times(-\alpha,\alpha).$
\end{itemize}
Furthermore, except for the cylinders $\mathscr H_s\times\R$ (in the case $r=n$), and
up to ambient isometries, $\Sigma$ is the only complete non totally geodesic $r$-minimal
hypersurface of \,$\h^n\times\R$ which is invariant by parabolic isometries.
\end{theorem}

\begin{proof}
The proof of the existence of $\Sigma$ as in (i) is analogous
to the one given in the preceding
theorem. So, let us assume $r<n.$ In this case,
the equation \eqref{eq-difequation} determined by $\mathscr F$ and  $H_r=0$  takes the form
\begin{equation}  \label{eq-edohorosphere}
y'=ay, \,\,\, a=\frac{rH_r^0}{H_{r-1}^0}>0\,,
\end{equation}
and its positive solutions are
\begin{equation}  \label{eq-solutiondifequation1}
\tau_\lambda(s)=\lambda e^{as}, \, \,\,\, \lambda>0, \,\,  s\in(-\infty,+\infty).
\end{equation}

It is easily checked that, since the horospheres of $\hf$ are pairwise congruent,
an $(f_s,\phi)$ graph with $\tau$-function
$\tau_\lambda$ does not depend on $\lambda.$ More precisely, two such graphs obtained
from functions $\tau_{\lambda_1}$ and $\tau_{\lambda_2}$\,, $\lambda_1\ne\lambda_2,$ are
isometric. Therefore, we can assume $\lambda=1$ and set $\tau:=\tau_1$\,. Then, we have
\[
0<\tau(s)< 1=\tau(0) \,\,\, \forall s\in (-\infty, 0).
\]

Since $\tau'(0)=a\tau(0)=a>0,$ writing $\rho^r=\tau|_{(-\infty, 0)}$, we have
\[
\phi(s)=\int_{0}^{s}\frac{\rho(u)}{\sqrt{1-\rho^2(u)}}du, \,\,\, s\in (-\infty,0),
\]
is well defined, and so is the corresponding $(f_s,\phi)$-graph $\Sigma'.$
Also,  $\tau(0)=1,$ so that the tangent spaces of $\Sigma'$ along its boundary are all
vertical. Therefore, we obtain the stated $r$-minimal hypersurface $\Sigma$ by reflecting
$\Sigma'$ about $P_0=\hf\times\{0\}.$

Observe that, for all $s\in (-\infty, 0),$ one has
\[
-\phi(s)=\int_{s}^{0}\frac{e^{au/r}}{\sqrt{1-e^{2au/r}}}du=\frac{r}{a}(\pi/2-\arcsin(e^{as/r})). % \le\frac{\pi r}{2a}\,,
\]
Hence, setting $\alpha:={\pi r}/{2a}>0,$ we have that
$\phi(s)\rightarrow -\alpha$ as $s\rightarrow-\infty,$
which proves that $\Sigma$ is contained in the slab $\hf\times(-\alpha,\alpha).$

As for the uniqueness of $\Sigma$, notice that the following hold:
\begin{itemize}
  \item The $\tau$-function of any $r$-minimal $(f_s,\phi)$-graph, $f_s\in\mathscr F,$ is a positive solution
  of \eqref{eq-edohorosphere} (if $r<n$) or is a positive constant (if $r=n$).
  \item $\Sigma$ has no horizontal points.
  \item A vertical $\mathscr F$-invariant hypersurface of $\hfr$ is $r$-minimal if and only if $r=n.$
  \item The graph in (i) has no vertical points.
\end{itemize}

These facts allow us to argue  as in preceding proofs, and then show the uniqueness of
$\Sigma$ as asserted.
\end{proof}

\section{Uniqueness of Rotational $H_r$-spheres of $\q_\epsilon^n\times\R$} \label{sec-uniqueness}

In this concluding section, we concern the uniqueness of the rotational $H_r$-spheres we
constructed in Section \ref{sec-rotationalHrhyp}.
We restrict ourselves to $\q_\epsilon^n\times\R$, with $\epsilon\in\{-1,1\}$  and $n\ge 3.$
As we mentioned before, the case $n=2$ was considered in \cite{abresch-rosenberg,esp-gal-rosen}.

We obtain a Jellett--Liebmann type theorem by showing that a compact, connected and strictly convex
$H_r$-hypersurface of $\q_\epsilon^n\times\R$ is a rotational embedded sphere (cf. Theorem \ref{th-uniqueness}).
We also show the uniqueness of these spheres
under  completeness or properness assumptions, instead of compactness
(cf. Theorem \ref{th-uniquenessforcomplete} and Corollary \ref{cor-uniqueness}).

For the  proof of Theorem \ref{th-uniquenessforcomplete},  we make
use of a height estimate for convex graphs in $M\times\R$ which we establish in the next proposition.
First, we compute the Laplacian of both the height function $\xi$ and
the angle function $\theta$ of an arbitrary  hypersurface $\Sigma$ of a general product $M\times\R.$

Given a smooth function $\zeta$ on $\Sigma,$ let us denote its Laplacian by $\Delta\zeta,$ i.e.,
\[
\Delta\zeta:={\rm trace}({\rm Hess}\,\zeta).
\]
In particular, from equation \eqref{eq-hessianheightfunction}, the Laplacian of $\xi$ is given by
\begin{equation}\label{laplacianheightfunction}
\Delta\xi = \Theta H, \,\,\, H=H_1\,.
\end{equation}
%where $H={\rm trace}\,A$ is the (non normalized) mean curvature function of $\Sigma.$

Recall that, for $X, Y\in T\Sigma,$  the Codazzi equation reads as
\[
(\overbar{R}(X,Y)N)^\top = \left(\nabla_YA\right)X - \left(\nabla_XA\right)Y,
\]
where $\overbar{R}$ is the curvature tensor of $M\times\R,$ $\top$ denotes the tangent component of
the tangent bundle $T\Sigma$ of $\Sigma,$ and, by definition,
\[
\left(\nabla_YA\right)X:=\nabla_YAX-A\nabla_YX.
\]

Observing that
\[
\nabla_X\nabla\xi=\left(\overbar\nabla_X\nabla\xi\right)^\top= -\left(\overbar\nabla_X\Theta N\right)^\top=\Theta AX,
\]
we have  from \eqref{eq-gradxi} that
\[-\nabla_X\nabla\Theta = \nabla_X A\nabla\xi=\left(\nabla_XA\right)\nabla\xi+A\nabla_X\nabla\xi=\left(\nabla_XA\right)\nabla\xi+\Theta A^2X,\]
which yields
\begin{equation}\label{eq-theta1}
\nabla_X\nabla\Theta=-(\overbar{R}(\nabla\xi,X)N)^\top - \left(\nabla_{\nabla\xi}A\right)X-\Theta A^2X.
\end{equation}

Now, let us fix $x\in\Sigma$ and an orthonormal frame
$\{X_1\,, \dots ,X_n\}$ in a neighborhood of $x$ in $\Sigma,$
which is geodesic at $x,$ that is
\[\nabla_{X_i}X_j\,(x)=0 \,\,\,\forall i, j=1,\dots ,n.\]
Writing $\xi_j=X_j(\xi),$ we have $\nabla\xi=\sum_j \xi_jX_j.$
Therefore
\begin{eqnarray}
  \sum_{i=1}^{n}\langle \left(\nabla_{\nabla\xi}A\right)X_i\,,\,X_i\rangle &=&\sum_{i=1}^{n}\left(\langle\nabla_{\nabla\xi}AX_i\,,\,X_i\rangle-
  \langle A\nabla_{\nabla\xi}X_i\,,\,X_i\rangle\right) \nonumber  \\
   &=& \sum_{i,j=1}^{n}\xi_j(\langle\nabla_{X_j}AX_i\,,\,X_i\rangle-\langle A\nabla_{X_j}X_i\,,\,X_i\rangle)  \nonumber\\
   &=& \sum_{i,j=1}^{n}\xi_j(X_j\langle AX_i,X_i\rangle-\langle AX_i,\nabla_{X_j}X_i\rangle-\langle A\nabla_{X_j}X_i\,,\,X_i\rangle) \nonumber\\
   &=& \langle\nabla\xi,\nabla H \rangle-\sum_{i,j=1}^{n}\xi_j(\langle AX_i,\nabla_{X_j}X_i\rangle-\langle A\nabla_{X_j}X_i\,,\,X_i\rangle), \nonumber
\end{eqnarray}
which implies that, at the chosen point $x\in\Sigma,$
\[
\sum_{i=1}^{n}\langle \left(\nabla_{\nabla\xi}A\right)X_i\,,\,X_i\rangle = \langle\nabla\xi,\nabla H\rangle.
\]

Since $x$ is arbitrary, we get from this last equality and \eqref{eq-theta1} that, on $\Sigma,$
\begin{equation}\label{eq-laplaciantheta}
\Delta\Theta=\overbar{\rm Ric}(\nabla\xi,N)-\langle\nabla\xi,\nabla H\rangle-\Theta\|A\|^2,
\end{equation}
where $\overbar{\rm Ric}$ denotes the Ricci curvature tensor of $M\times\R$ and
$\|A\|^2:={\rm trace}\,A^2.$

\begin{remark}
For the next results, except for Theorem \ref{th-uniqueness2},  we  order the principal curvatures
of a hypersurface $\Sigma$ of $M\times\R$ as
\[k_1\le k_2\le\cdots\le k_{n-1}\le k_n\,.\]
\end{remark}

\begin{proposition} \label{prop-vhe}
Consider an arbitrary Riemannian manifold \,$M,$ and let $\Sigma\subset M\times\R$ be a
compact vertical graph of a  nonnegative function  defined on a domain $\Omega\subset M\times\{0\}.$
Assume $\Sigma$  strictly convex up to $\partial\Sigma\subset M\times\{0\}.$ Under these conditions,
the following height estimate holds:
\begin{equation} \label{eq-delta}
\xi(x)\le\frac{1}{\inf_\Sigma k_1} \,\,\, \forall x\in\Sigma.
\end{equation}
\end{proposition}

\begin{proof}
Consider in $\Sigma$ the ``inward'' orientation, so that
its angle function $\theta$ is non positive. Choose $\delta>0$ satisfying
$1/\delta<\inf_\Sigma k_1$ and define on $\Sigma$ the function
\[
\varphi=\xi+\delta\Theta.
\]

We claim that $\varphi$ has no interior maximum.  Indeed, assuming otherwise,
let $x\in\Sigma-\partial\Sigma$ be a maximum point of $\varphi$. In this case, from \eqref{eq-gradxi}, we have
\[
0=\nabla\varphi(x)=\nabla\xi(x)+\delta\nabla\Theta(x)=\nabla\xi(x)-\delta A\nabla\xi(x).
\]
Hence, if we had $\nabla\xi(x)\ne 0,$ then $1/\delta$ would be an eigenvalue of
$A$ at $x,$ which is impossible, by our choice of $\delta.$
Thus, $x$ is a critical point of $\xi.$ Since $\Sigma$ is strictly convex,
$x$ is necessarily its highest point. In particular,
$\Theta(x)=-1.$ This, together with identities \eqref{laplacianheightfunction}
and \eqref{eq-laplaciantheta}, gives that, at  $x$,
\begin{equation} \label{eq-laplacian10}
0\ge\Delta\varphi=-H+\delta\|A\|^2.
\end{equation}
However,  from our choice of $\delta$  and the strict convexity of $\Sigma$, we have
\[
\frac{H}{\delta}<k_1H=k_1(k_1+\cdots +k_n)\le k_1^2+\cdots+k_n^2=\|A\|^2,
\]
which contradicts \eqref{eq-laplacian10}.
Therefore, $\varphi$ attains its maximum on $\partial\Sigma,$
which implies that $\varphi\le 0$ on $\Sigma,$ for
$\varphi|_{\partial\Sigma}=\delta\Theta\le 0.$ Hence,
\[
\xi(x)\le -\delta\Theta(x)\le\delta \,\,\, \forall x\in\Sigma.
\]

The result, then, follows from this last inequality,
since it holds for any positive $\delta>1/ \inf_\Sigma k_1.$
\end{proof}

\begin{remark}
Proposition \ref{prop-vhe} has its own importance,
since  it establishes height estimates for vertical graphs
in $M\times\R$ making
no assumptions on $M.$ In addition, no curvature of such a graph is
assumed to be constant.
\end{remark}

In the next two theorems, we  apply the Alexandrov reflection technique.
Since the arguments  are standard, the proofs  will be somewhat sketchy on this matter
(see, e.g., \cite[Theorems 4.2 and 5.1]{cheng-rosenberg} and \cite[Theorem 1.1]{nelli-rosenberg}).
We add that the proof of Theorem \ref{th-uniqueness} is, essentially, the one for
\cite[Corollary 1]{delima}, in which  the case $r=1$ was considered.

\begin{theorem}[Jellett--Liebmann-type theorem]\label{th-uniqueness}
Let $\Sigma$ be a {compact}  connected strictly convex $H_r(>0)$-hypersurface of
\,$\q_\epsilon^n\times\R$ ($n\ge 3$).
Then, $\Sigma$ is an {embedded} rotational $H_r$-sphere.
\end{theorem}
\begin{proof}
Since $\Sigma$ is compact, its  height function $\xi$ has a maximal point $x.$
This, together with the strict convexity of $\Sigma$, allows us to apply
\cite[Theorems 1 and 2]{delima} and conclude that $\Sigma$ is  embedded and homeomorphic to $\s^n$.
Thus, for $\epsilon=-1,$ the result follows from \cite[Theorem 7.6]{elbert-earp},  the Alexandrov-type
theorem  we mentioned in the introduction.

For $\epsilon=1,$ we can perform Alexandrov reflections on $\Sigma$ with respect to horizontal hyperplanes
$P_t:=\s^n\times\{t\}$ coming down  from above $\Sigma.$  For some $t_0<\xi(x),$
the reflection of the part of $\Sigma$ above $P_{t_0}$ will have a
first contact with $\Sigma.$ Then, by the Maximum-Continuation Principle, %and the analyticity of $\Sigma,$
$\Sigma$ is symmetric with respect to $P_{t_0}$.
Therefore, assuming $t_0=0$ and identifying
$\s^n\times\{0\}$ with $\s^n,$ we conclude that
$\Sigma$ is a ``bigraph'' over  its projection $\pi(\Sigma)$ to $\s^n$.
As a consequence, $\Sigma_0:=\Sigma\cap\s^n$ is the boundary of
$\pi(\Sigma)$ in $\s^n$.

By \cite[Lemma 1]{delima}, the second fundamental form of $\Sigma_{0}$,
as a hypersurface of \,$\s^n$, is
positive definite. In particular, $\Sigma_{0}$ is non totally geodesic in \,$\s^n$. Thus,
by \cite[Theorem 1]{docarmo-warner}, $\Sigma_{0}$ is contained in an open hemisphere $\s_+^n$
of $\s^n$, which implies that the same is true for $\pi(\Sigma),$
that is, $\Sigma\subset \s_+^n\times\R.$ In this setting,
we can apply  Alexandrov reflections on ``vertical hyperplanes'' $(\s^{n-1}\cap\s_+^n)\times\R,$ where $\s^{n-1}\subset\s^n$
is a totally geodesic $(n-1)$-sphere of $\s^n$\,,
and  conclude that $\Sigma$ is rotational.
\end{proof}

Let us show now that, regarding Theorem \ref{th-uniqueness},  the compactness hypothesis can be replaced
by completeness if we add a one point condition on the height function of $\Sigma.$ In the case $\epsilon=-1$,
we also have  to impose a condition on the second fundamental form of $\Sigma,$ which turns out to be a necessary hypothesis
(see Remark \ref{rem-necessaryhiphotesis}, below).

\begin{theorem} \label{th-uniquenessforcomplete}
Let $\Sigma$ be a complete   connected strictly convex  $H_r(>0)$-hypersurface of
$\q_\epsilon^n\times\R$ ($n\ge 3$) whose height function $\xi$ has a local extreme point.
For $\epsilon=-1$, assume further that the least principal curvature $k_1$ of $\Sigma$ is bounded away from zero.
Then, $\Sigma$ is an embedded rotational sphere.
\end{theorem}
\begin{proof}
As in the previous theorem, $\Sigma$ fulfills the hypotheses of \cite[Theorems 1 and 2]{delima}, which implies that
$\Sigma$ is properly embedded and homeomorphic to either
$\s^n$ or $\R^n.$ Furthermore, in the latter case,  the height function of $\Sigma$ is unbounded and has a single
extreme point $x,$ which we assume to be a maximum.

For  $\epsilon=1,$  the height estimates obtained
in \cite[Theorem 4.1-(i)]{cheng-rosenberg} forbid $\xi$ to be unbounded. Thus, in this case,
$\Sigma$ is homeomorphic to $\s^n$ and the result follows from Theorem \ref{th-uniqueness}.

Let us consider now the case $\epsilon=-1.$
Assume, by contradiction, that $\Sigma$ is homeomorphic to $\R^n,$  so that
$\xi$ is unbounded below. Hence, given a  horizontal hyperplane
$P_t=M\times\{t\}$ with $t<\xi(x),$  the part $\Sigma_t^+$ of $\Sigma$ which lies above $P_t$
must be a vertical graph with boundary in $P_t$.
If not, for some $t'$ between $t$ and $\xi(x),$ $P_{t'}$ would be orthogonal
to $\Sigma$ at one of its points. Then, the Alexandrov reflection method  would give that $\Sigma$ is
symmetric with respect to $P_{t'},$ which is impossible, since we are assuming
$\xi$   unbounded, and the closure of $\Sigma_{t'}^+$ in $\Sigma$ is compact.

It follows from the above that, for
$|t|$  sufficiently large, one has
\[\xi(x)-t>\frac{1}{\inf_\Sigma k_1}\ge\frac{1}{\inf_{\Sigma_t^+} k_1}\,,\]
which clearly
contradicts Proposition \ref{prop-vhe}. Therefore, $\Sigma$ is homeomorphic to
$\s^n$ and, again, the result follows from Theorem \ref{th-uniqueness}.
\end{proof}

\begin{remark}
In Theorems \ref{th-uniqueness} and \ref{th-uniquenessforcomplete}, the hypothesis of
strict convexity of $\Sigma$ is automatically satisfied for $r=n,$ so  it can be dropped in this case.
Indeed, in both theorems,  the height function $\xi$ has a critical point $x\in\Sigma$, which can be assumed to be a maximum.
Then, taking the inward orientation on $\Sigma,$
we have that  $\theta(x)=-1,$ which, together with equality \eqref{eq-hessianheightfunction}, yields
\[\langle AX,X\rangle=-{\rm Hess}\,\xi(X,X)\ge 0 \,\,\, \forall X\in T_{x}\Sigma.\]
However, $H_n=\det A>0$ on $\Sigma$. Thus, at $x,$ and then on all of $\Sigma$,
the second fundamental form  is positive definite, that is, $\Sigma$ is strictly convex.
%For $r<n$ and $\epsilon=1,$ Theorem \ref{th-delaunaytype} shows that the hypothesis on the height function
%$\xi$ is necessary for the conclusion.
\end{remark}

\begin{remark} \label{rem-necessaryhiphotesis}
It follows from the considerations of
Remark \ref{rem-zerokn} that,
for  $r<n,$ the hypothesis on the least principal curvature
of $\Sigma$ in Theorem \ref{th-uniquenessforcomplete} is necessary for the conclusion.
As shown by Theorem \ref{th-delaunaytypesn}, the same is true for the hypothesis on the height function
$\xi$ in the case $\epsilon=1$ and  $r<n.$
\end{remark}

Next, we consider the dual case of  Theorem \ref{th-uniquenessforcomplete}, assuming
now that the height function of the hypersurface $\Sigma$ has no critical points. First, we recall
that a  hypersurface $\Sigma\subset \h^n\times\R$
is said to be \emph{cylindrically bounded}, if there exists a closed geodesic ball
$B\subset\h^n$ such that $\Sigma\subset B\times\R.$

%Our next result  contrasts with the  previous theorem
%regarding the assumption on the height function of $\Sigma.$

\begin{theorem} \label{th-uniqueness2}
Let $\Sigma$ be a proper, convex,   connected  $H_r(> 0)$-hypersurface of
\,$\q_\epsilon^n\times\R$ ($n\ge 3$) with no horizontal points.
For $\epsilon=-1,$ assume further that $\Sigma$ is cylindrically
bounded. Then, $\Sigma$ is a cylinder over a geodesic sphere
of \,$\q_\epsilon^n$. In particular,  $r<n.$
\end{theorem}

\begin{proof}
From the hypothesis and \cite[Theorem 3]{delima},
$\Sigma=\Sigma_0\times\R,$ where $\Sigma_0$ is an embedded convex
topological sphere of $\q_\epsilon^n$\,.  Moreover, in the
case $\epsilon=1,$ $\Sigma_0$ is contained in an open hemisphere of $\s^n.$

At a given  point  $x\in \Sigma,$ the principal curvatures
are $k_1, \,\dots ,k_{n-1}, 0,$ where $k_1\,, \dots ,k_{n-1}$ are the principal curvatures of
$\Sigma_0\subset\q_\epsilon^n$  at $\pi_{\q^n_\epsilon}(x)\in\Sigma_0$\,. In particular,
$\Sigma_0$ has constant $r$-th mean curvature $H_r$ if $r<n,$ which implies that, in this case,
$\Sigma_0$ is a geodesic sphere of $\q_\epsilon^n$ (see \cite{korevaar, montiel-ros}). Also,
$H_n=0$ on $\Sigma$, so we must have $r<n,$ since
we are assuming $H_r>0.$
\end{proof}

Since a cylinder $\Sigma_0\times\R\subset\s^n\times\R$ is nowhere strictly convex,
it follows from the above theorem that, for $n\ge 3,$  a connected, proper, and  strictly convex $H_r$-hypersurface
of $\s^n\times\R$ must have a horizontal point. This fact, together with Theorem \ref{th-uniquenessforcomplete},
gives our last result:

\begin{corollary} \label{cor-uniqueness}
For $n\ge 3,$ any connected,  properly immersed, and strictly convex $H_r(>0)$-hypersurface of  \,$\s^n\times\R$
is necessarily  an embedded rotational $H_r$-sphere.
\end{corollary}

\section{Acknowledgments}
We are indebt to Antonio Martinez, Pablo Mira, and Miguel Domínguez-Vázquez
for their  valuable suggestions.
Fernando Manfio is supported by Fapesp, grant 2016/23746-6.
João Paulo dos Santos is supported by FAPDF, grant 0193.001346/\\2016.

\end{document}